\documentclass[UTF8,10pt]{article}
\usepackage[left=1.91cm,right=1.91cm,top=2.54cm,bottom=2.54cm]{geometry}
\usepackage{amsfonts}
\usepackage{amsmath}
\usepackage{amssymb}
\usepackage{pgfplots}
\usepackage{tikz}
\usetikzlibrary{arrows}
\usepackage{appendix}
\usepackage{tikz-cd}
\usepackage{mathtools}
\usepackage{amsthm}
\usepackage{setspace}
\usepackage{enumitem}
\setenumerate[1]{itemsep=0pt,partopsep=0pt,parsep=\parskip,topsep=5pt}
\setitemize[1]{itemsep=0pt,partopsep=0pt,parsep=\parskip,topsep=5pt}
\setdescription{itemsep=0pt,partopsep=0pt,parsep=\parskip,topsep=5pt}
\usepackage[colorlinks,linkcolor=black,anchorcolor=black,citecolor=black]{hyperref}
\usepackage{cases}
\usepackage{fancyhdr}
\usepackage[scaled=0.92]{helvet}	
\usepackage{txfonts}

\theoremstyle{definition}
\newtheorem{thm}{Theorem}[section]

\newtheorem{defn}{Definition}[section]
\newtheorem{lem}[thm]{Lemma}
\newtheorem{prop}[thm]{Proposition}

\newtheorem*{rmk}{Remark}

\newcommand{\R}{\mathbb{R}}

\newcommand{\N}{\mathbb{N}}

\newcommand{\T}{\mathbb{T}}

\newcommand{\TP}{\overline{\partial}{}}
\newcommand{\TL}{\overline{\Delta}{}}
\newcommand{\tpl}{\overline{\partial}^2\overline{\Delta}}
\newcommand{\curl}{\text{curl }}
\newcommand{\dive}{\text{div }}

\newcommand{\p}{\partial}

\newcommand{\DD}{\mathcal{D}}

\newcommand{\ak}{\tilde{a}}
\newcommand{\Ak}{\tilde{A}}
\newcommand{\ek}{\tilde{\eta}}
\newcommand{\Jk}{\tilde{J}}
\newcommand{\ar}{\mathring{a}}

\newcommand{\ark}{\mathring{\tilde{{a}}}}
\newcommand{\Ark}{\mathring{\tilde{{A}}}}
\newcommand{\Jrk}{\mathring{\tilde{{J}}}}
\newcommand{\er}{\mathring{\eta}}
\newcommand{\hr}{\mathring{h}}
\newcommand{\erk}{\mathring{\tilde{{\eta}}}}
\newcommand{\vr}{\mathring{v}}
\newcommand{\Jr}{\mathring{J}}
\newcommand{\psir}{\mathring{\psi}}
\newcommand{\park}{\nabla_{\mathring{\tilde{a}}}}

\newcommand{\lapark}{\Delta_{\mathring{\tilde{a}}}}

\newcommand{\lapak}{\Delta_{\tilde{a}}}

\newcommand{\divr}{\text{div}_{\mathring{\tilde{a}}}}

\newcommand{\curlr}{\text{curl}_{\mathring{\tilde{a}}}}

\newcommand{\pa}{\nabla_a}

\newcommand{\pak}{\nabla_{\tilde{a}}}

\newcommand{\diva}{\text{div}_{\tilde{a}}}

\newcommand{\curla}{\text{curl}_{\tilde{a}}}

\newcommand{\lkk}{\Lambda_{\kk}}

\newcommand{\lap}{\Delta}

\newcommand{\dx}{\,dx}
\newcommand{\dy}{\,dy}
\newcommand{\dz}{\,dz}
\newcommand{\dt}{\,dt}
\newcommand{\dS}{\,dS}

\newcommand{\kk}{\kappa}
\newcommand{\eps}{\varepsilon}

\newcommand{\EE}{\mathbb{E}}
\newcommand{\EEE}{\mathfrak{E}}
\newcommand{\CC}{\mathfrak{C}}
\newcommand{\h}{\mathfrak{h}}

\newcommand{\VV}{\mathbf{V}}
\newcommand{\HH}{\mathbf{H}}
\newcommand{\VVV}{\mathfrak{V}}
\newcommand{\HHH}{\mathfrak{H}}
\newcommand{\GG}{\mathbf{G}}
\newcommand{\G}{\mathbf{G}^0}
\newcommand{\F}{\mathbf{F}^0}
\newcommand{\GP}{(\mathbf{G}^0_j\cdot\partial)}
\newcommand{\FP}{(\mathbf{F}^0_j\cdot\partial)}
\newcommand{\FF}{\mathbf{F}}
\newcommand{\ff}{\mathfrak{f}}
\newcommand{\KK}{\mathfrak{K}}
\newcommand{\NN}{\mathfrak{N}}
\newcommand{\BB}{\mathfrak{B}}
\newcommand{\XX}{\mathbf{X}}
\newcommand{\YY}{\mathbf{Y}}

\newcommand{\FT}{F^{\top}}

\newcommand{\FFT}{\mathbf{F}^{\top}}
\newcommand{\EEr}{\mathring{\mathbb{E}}}
\newcommand{\VVr}{\mathring{\mathbf{V}}}
\newcommand{\HHr}{\mathring{\mathbf{H}}}
\newcommand{\GGr}{\mathring{\mathbf{G}}}
\newcommand{\FFr}{\mathfrak{F}}

\newcommand{\PP}{\mathcal{P}}

\newcommand{\ww}{\textbf{w}_0}
\newcommand{\vvv}{\textbf{v}_0}

\newcommand{\wt}{e'(h)}
\newcommand{\swt}{\sqrt{e'(h)}}

\newcommand{\idt}{\int_{\mathcal{D}_t}}

\newcommand{\io}{\int_{\Omega}}

\newcommand{\ig}{\int_{\Gamma}}
\numberwithin{equation}{section}
\usepackage{xcolor}

\begin{document}
\bibliographystyle{plain}
\title{\textbf{Local Well-posedness and Incompressible Limit \\ of the Free-Boundary Problem in \\ Compressible Elastodynamics}}
\author{{\sc Junyan Zhang}\thanks{{Johns Hopkins University, 3400 N Charles St, Baltimore, MD 21218, USA. Email: \texttt{zhang.junyan@jhu.edu}}
}}
\date{\today}
\maketitle

\begin{abstract}
We consider 3D free-boundary compressible elastodynamic system under the Rayleigh-Taylor sign condition. It describes the motion of an isentropic inviscid elastic medium with moving boundary. The deformation tensor is assumed to satisfy the neo-Hookean linear elasticity. The local well-posedness was proved by Trakhinin \cite{trakhinin2016elastic} by Nash-Moser iteration. In this paper, we give a new proof of the local well-posedness by the combination of classical energy method and hyperbolic approach. In the proof, we apply the tangential smoothing method to define the approximation system. The key observation is that the structure of the wave equation of pressure together with Christodoulou-Lindblad \cite{christodoulou2000motion} elliptic estimates reduces the energy estimates to the control of tangentially-differentiated wave equations despite a potential loss of derivative in the source term. To the best of our knowledge, we first establish the nonlinear energy estimate without loss of regularity for free-boundary compressible elastodynamics. The energy estimate is also uniform in sound speed which yields the incompressible limit, i.e., the solutions of the free-boundary compressible elastodynamic equations converge to the incompressible counterpart provided the convergence of initial datum.

It is worth emphasizing that our method is completely applicable to compressible Euler equations. Our observation also shows that it is not necessary to include the full time derivatives in the boundary energy and analyze higher order wave equations as in Lindblad-Luo \cite{lindblad2018priori} and Luo \cite{luo2018ww} even if we require the energy is uniform in sound speed. Moreover, the enhanced regularity for compressible Euler equations obtained in \cite{lindblad2018priori,luo2018ww} can still be recovered for a slightly compressible elastic medium by further delicate analysis of the Alinhac good unknowns, which is completely different from Euler equations. 
\end{abstract}



\setcounter{tocdepth}{2}
\setcounter{MaxMatrixCols}{14}
\tableofcontents

\section{Introduction}

We consider 3D compressible elastodynamic equations describing the motion of an isentropic inviscid elastic medium with free boundary \cite{dafermos10,Gurtin}
\begin{equation}\label{elastoeq}
\begin{cases}~
\rho D_t u=-\nabla p+\dive(\rho \FF\FFT)~~~&\text{in }\DD,\\
D_t \rho+\rho\dive u=0~~~&\text{in }\DD,\\
D_t \FF=\nabla u\FF~~~&\text{in }\DD.
\end{cases}
\end{equation}Here $\DD:=\bigcup_{0\leq t\leq T}\{t\}\times\DD_t$ and $\DD_t\subseteq\R^3$ is the domain occupied by the elastic medium at time $t$. $\nabla:=(\p_{x_1},\p_{x_2},\p_{x_3})$ is the standard spatial derivative and $\dive X:=\p_{x_i} X^i$ is the standard divergence for any vector field $X$. $D_t:=\p_t+u\cdot\nabla$ is the material derivative. Throughout this paper, $X^i=\delta^{li}X_l$ for any vector field $X$, i.e., we use Einstein summation convention. The quantities $u,\rho,p$ denotes the velocity, density and fluid pressure of the elastic medium. $\FF:=(\FF_{ij})_{3\times 3}$ is the deformation tensor and $\FFT$ denotes its transpose matrix. $\rho\FF\FFT$ is the Cauchy-Green stress tensor in the case of compressible neo-Hookean linear elasticity. $D_t:=\p_t+u\cdot\nabla$ is the material derivative. 

We impose the equation of state as 
\begin{equation}\label{eos1}
p=p(\rho)\text{ being a strictly increasing function of }\rho,\text{ for }\rho\geq \bar{\rho_0}
\end{equation} where $\bar{\rho_0}:=\rho|_{\p\DD}$ is a positive constant and we set $\bar{\rho_0}=1$ throughout this manuscript, which means the elastic medium is isentropic with strictly positive density. We also impose the divergence constraints on the deformation tensor
\begin{equation}\label{divf}
\dive(\rho \FFT)=0.
\end{equation} This will not make the system be over-determined because we only require it holds for the initial data and it automatically propagates to any time (cf. Trakhinin \cite[Proposition 2.1]{trakhinin2016elastic}). Note that the system \eqref{elastoeq} together with \eqref{eos1} and \eqref{divf} describes the motion of elastic waves in a compressible inviscid neo-Hookean linear material corresponding to the elastic energy $W(\FF)=\frac12|\FF|^2$. It also arises as the inviscid limit of the compressible visco-elastodynamics \cite{dafermos10,Gurtin} of the Oldroyd type \cite{old1,old2}.

Next we introduce the boundary conditions
\begin{equation}\label{elastobdry}
\begin{cases}
V(\p\DD_t)=u\cdot n~~~&\text{ on }\p\DD,\\
p=0~~~&\text{ on }\p\DD,\\
\FFT\cdot n=\mathbf{0}~~~&\text{ on }\p\DD.
\end{cases}
\end{equation} The first boundary condition shows that the boundary moves with the velocity of the fluid. It can be considered as the definition of free-boundary problem, and can be equivalently written as $D_t|_{\p\DD}\in\mathcal{T}(\p\DD)$ where $\mathcal{T}(\p\DD)$ denotes the tangential bundle of 
$\p\DD$, or $(1,u)$ is tangent to $\p\DD$. The second condition $p=0$ means that outside the fluid region is the vacuum. The third condition means that the deformation tensor is tangential on the boundary, which is also required for initial data only and then automatically propagates. Here $n$ denotes the exterior unit normal vector to $\p\DD_t$.

\begin{rmk}
The boundary condition $\FF^{\top}\cdot n=\mathbf{0}$ originates from the Rankine-Hugoniot conditions for the vortex sheets in compressible elastodynamics \cite{CHWWY1,CHWWY2,CHWWY3,CHWWY4}.  It leads to $\det\FF=0$ on the boundary. Such ``degeneracy" is due to the mathematical formulation of the singular vortex sheet structure in elastic fluids. The same ``degenaracy" takes place for the free-boundary problems and vortex sheets problems for both incompressible \cite{HaoWang2014,GuWang2018,LWZ1,LWZ2,ZhangYH} and compressible elastodynamics \cite{trakhinin2016elastic}. For more illustrations, we refer readers to \cite[Remark 2.1]{CHWWY2}.
\end{rmk}

\paragraph*{Energy conservation law}

 Throughout this paper, we use Einstein summation convention, i.e., repeated indices is summation on this index. We define $Q(\rho):=\int_1^{\rho}p(R)/R^2~dR$. Under \eqref{elastobdry}, the system \eqref{elastoeq} has the following energy conservation law
\begin{equation}\label{econserve}
\begin{aligned}
&\frac{d}{dt}\left(\frac12\idt\rho|u|^2+\rho|\FF|^2\dx+\idt\rho Q(\rho)\dx\right)\\
=&\idt\rho u^i\cdot D_t u_i\dx+\sum_{j=1}^3\idt\rho\FF_{ij}D_t\FF_{ij}\dx+\idt\frac{p(\rho)}{\rho} D_t\rho\dx\\
=&-\idt (u\cdot\nabla) p\dx+\sum_{j=1}^3\idt \rho u^i \FF_{kj}\p_k\FF_{ij}\dx+\sum_{j=1}^3\idt\rho\FF_{ij}\FF_{kj}\p_k u^i\dx+\idt p(-\dive u)\dx\\
=&-\int_{\p\DD_t}u^i n_ip\dS+\idt p\dive u\dx-\idt p\dive u\dx\\
&-\sum_{j=1}^3\idt\rho\FF_{kj}\p_k u^i\FF_{ij}\dx-\sum_{j=1}^3\io u^i\underbrace{\p_k(\rho\FF_{kj})}_{=0}\FF_{ij}\dx+\int_{\p\DD_t}\rho u^i \underbrace{\FF_{kj}n_k}_{=0}\FF_{ij}\dS+\sum_{j=1}^3\idt\rho\FF_{ij}\FF_{kj}\p_k u^i\dx\\
=&~0.
\end{aligned}
\end{equation}

\paragraph*{Enthalpy formulation}

Before further introducing the Rayleigh-Taylor sign condition and more physical constraints, we could simplify the system by introducing the enthalpy to replace pressure and density. Define $\FF_j:=(\FF_{1j},\FF_{2j},\FF_{3j})$ to be the $j$-th column of $\FF$. Then we have
\[
0=\dive(\rho \FFT)_j:=\p_k(\rho \FF_{kj})\Rightarrow\p_k \FF_{kj}=-\FF_{kj}\frac{\p_k\rho}{\rho}=-\FF_{kj}\p_k\ln\rho=:-(\FF_j\cdot\nabla)\ln\rho,
\] and thus
\[
\dive(\rho \FF\FFT)_i:=\p_k(\rho\FF_{ij} \FF_{kj})=\underbrace{\p_k(\rho \FF_{kj})}_{=0}\FF_{ij}+\rho \FF_{kj}\p_k \FF_{ij}=(\rho(\FF_j\cdot\nabla)\FF_j)_i.
\]
The second equation reads $\dive u=-D_t(\ln\rho)$ and the third equation reads
\[
D_t \FF_{ij}=(\nabla u\FF)_{ij}:=\FF_{kj}\p_k u_i=(\FF_j\cdot\nabla)u_i.
\] The boundary condition on $\FF$ is expressed as $0=(\FFT\cdot n)_j=\FF_{ij}n_i=:\FF_j\cdot n$.

Now we introduce the ``enthalpy" $\h$ as an increasing function of $\rho$ defined by $\h(\rho):=\int_1^{\rho}\frac{p'(r)}{r}~dr$ and $e(\h):=\ln\rho(\h)$. Under the enthalpy formulation, $\frac{1}{\rho}\p_i p=\p_i\h$. Then the system \eqref{elastoeq} can be re-written as
\begin{equation}\label{elastoequ}
\begin{cases}
D_t u=-\nabla \h+\sum\limits_{j=1}^3(\FF_j\cdot\nabla)\FF_j~~~&\text{in }\DD,\\
\dive u=-D_te(\h)=-e'(\h)D_t\h~~~&\text{in }\DD,\\
D_t \FF_j=(\FF_j\cdot\nabla)u~~~&\text{in }\DD,\\
(\dive \FFT)_j:=\p_k\FF_{kj}=-(\FF_j\cdot\nabla)e(\h)=-e'(\h)(\FF_j\cdot\nabla)\h~~~&\text{in }\DD.
\end{cases}
\end{equation}
together with the boundary condition
\begin{equation}\label{elastobdryu}
\begin{cases}
D_t|_{\p\DD}\in\mathcal{T}(\p\DD)~~~&\text{ on }\p\DD,\\
\h=0~~~&\text{ on }\p\DD,\\
\FF_j\cdot n=0~~~&\text{ on }\p\DD.
\end{cases}
\end{equation}

The new system looks quite similar to the incompressible elastodynamic equations, where $\h$ plays the role as $p$ of the incompressible counterpart, while $\dive u$ and $\dive \FFT$ is no longer zero yet determined as a function of $\rho$ and thus of $\h$. In order for the initial-boundary value problem of \eqref{elastoequ}-\eqref{elastobdryu} being solvable, we need to impose some other natural physical conditions and initial data satisfying the compatibility conditions.

\paragraph*{Physical constraints}

First we impose the Rayleigh-Taylor sign condition on the free surface
\begin{equation}\label{sign}
-\frac{\p \h}{\p n}\geq c_0>0~~~\text{ on }\p\DD_t,
\end{equation}where $c_0$ is a given constant. This is equivalent to the natural physical condition $-\frac{\p p}{\p n}\geq c_0'>0$ in the study of the motion of a free-boundary fluid which says that the pressure is larger in the interior than on the boundary. Ebin \cite{ebin1987equations} proved the ill-posedness of incompressible Euler equations when \eqref{sign} is violated. The condition \eqref{sign} is only required for the initial data: We will justify this condition by proving $-\frac{\p \h}{\p n}$ on the boundary is a $C_{t,x}^{0,1/4}$ function, and thus \eqref{sign} propagates in a positive time interval.

Next we impose the following natural conditions on $e(\h)$: For each fixed $m\in\N^*$, there exists a constant $A>1$ such that
\begin{equation}\label{sound}
|e^{(k)}(\h)|\leq A\text{ and }|e^{(k)}(\h)|\leq A|e'(\h)|^k\leq A|e'(\h)|~~\forall k\leq m+1.
\end{equation}

\paragraph*{Compatibility conditions on initial data}

Finally we require the initial data $(u_0,\F,\h_0,\DD_0)$ to satisfy the compatibility conditions at the boundary. From \eqref{elastoequ}-\eqref{elastobdryu} we know $\dive u|_{\{0\}\times\p\DD_0}=0$ and $\h_0|_{\p\DD_0}=0$ is needed, which is called the 0-th order compatibility condition. We define the $m$-th order compatibility conditions to be
\begin{equation}\label{ccd}
D_t^k \h|_{\{0\}\times\p\DD_0}=0,~~~k=0,1,\cdots, m.
\end{equation} We will prove in Section \ref{data} that such initial data must exist provided that the sound speed $c^2:=p'(\rho)$ (or equivalently $1/e'(h)$) is suitably large.

Given initial data $(u_0,\F,\h_0,\DD_0)$ satisfying the compatibility conditions and constraints \eqref{divf}, ${\FF}^{\top}\cdot n|_{\{0\}\times\p\DD_0}=\mathbf{0}$ as well as Rayleigh-Taylor sign condition \eqref{sign} at $t=0$,  where $\DD_0\subseteq \R^3$ is a simply-connected bounded domain, we want to find a set $\DD_t\subseteq\R^3$ and velocity $u$, enthalpy $\h$ and deformation tensor $\FF$ solving the system \eqref{elastoequ}-\eqref{ccd}. In this manuscript, we aim to prove the local well-posedness of \eqref{elastoequ}-\eqref{ccd} with energy estimates and the incompressible limit.

\subsection{History and Background}

The study of both free-surface fluid and elastodynamics has a long history and has blossomed in the recent decades. Let us first review the results on the free-boundary Euler equations. For the incompressible case, Wu \cite{wu1997LWPww, wu1999LWPww} on the local well-posedness (LWP) of full water wave system have been considered as the first breakthrough in the study of free-surface perfect fluid. We also refer to Lannes \cite{lannes2005ww,masmoudi05,mz2009} for the study of local theory of incompressible irrotational water wave. When the vorticity is nonzero, Christodoulou-Lindblad \cite{christodoulou2000motion} first established the a priori estimates without loss of regularity. Lindblad \cite{lindblad2002, lindblad2005well} proved the local well-posedness by using Nash-Moser iteration. Coutand-Shkoller \cite{coutand2007LWP, coutand2010LWP} introduced the tangential smoothing method to proved the local well-posedness with or without surface tension and avoid the loss of regularity. See also \cite{zhangzhang08Euler,alazard2014} for the case without surface tension and \cite{Schweizer05,shatah2008geometry, shatah2008priori,shatah2011local} for the nonzero surface tension case.

The study for a free-surface compressible fluid is much more difficult because the pressure is governed by a wave equation instead of being a Lagrangian multiplier in the incompressible case. We list the results in the case of a liquid. Lindblad \cite{lindblad2003well, lindblad2005cwell} proved the LWP by Nash-Moser iteration and then Trakhinin \cite{trakhiningas2009} extended the LWP to non-isentropic, non-relativistic and relativistic liquid with gravity in an unbounded domain via hyperbolic approach and also Nash-Moser iteration. The first a priori estimates without loss of regularity was established by Lindblad-Luo \cite{lindblad2018priori} and then was extended to a compressible water wave with vorticity Luo \cite{luo2018ww}. Ginsberg-Lindblad-Luo \cite{GLL2019LWP} proved the LWP for the self-gravitating liquid. Then the author joint with Luo \cite{luozhangCWWLWP} proved the LWP for a compressible gravity water wave with vorticity with a simplified method. In the case of nonzero surface tension, we refer to \cite{coutand2013LWP,DK,luo2019limit}. In the case of a gas, we refer to \cite{coutand2010priori,coutand2012LWP,jang2014gas,luoxinzeng2014,Hao2015gas,ITgas} and references therein. 

Now let us review the development of elastodynamics system which describe the motion of an elastic medium. When the domain is fixed or $\R^n$, most of the results focus on the incompressible visco-elastic case because the solution is expected to be global. See \cite{LZ05,LLZ05,LZ08,LLZ08,dafermos10,Lin12,CSelastic} and references therein. For the compressible viscoelastic case, we refer to  \cite{HuWang2010,HuWang2011,HuWu2013,qian1,qian2}. The incompressible ideal elastodynamic equations satisfy the null condition and thus the global solution in $\R^n$ can also be expected. We refer to \cite{ebin1993,Sideris04,Sideris07,LSZ15,Lei16,wxcelastic,CaiLeiM19}.

However, the study of free-surface elastodynamics becomes quite different. On the one hand, the domain is no longer $\R^n$ but with a boundary whose regularity is also limited and thus it is quite difficult to recover the global solution as in the $\R^n$-case. The only related result is Xu-Zhang-Zhang \cite{XZZ2013} that proved the global well-posedness (GWP) of incompressible visco-elastodynamics with surface tension. On the other hand, the contribution of free boundary enters to the highest order term in the energy, and extra stabilizing conditions such as Rayleigh-Taylor sign condtion are required. So far, nearly all the known results only deal with the local-in-time a priori estimates and LWP for the neo-Hookean elastodynamics and most are only available for the incompressible case. For the incompressible case, Hao-Wang \cite{HaoWang2014} proved the Christodoulou-Lindblad type a priori estimates under Rayleigh-Taylor sign condition \eqref{sign}. Gu-Wang \cite{GuWang2018} and Li-Wang-Zhang \cite{LWZ2} proved the LWP under a mixed stability condition. Li-Wang-Zhang \cite{LWZ1} proved the LWP for the incompressible vortex sheets in elastodynamics under a non-collinearity condition. Hu-Huang \cite{HuHuang2018} proved the LWP under Rayleigh-Taylor sign condition by generalizing Lindblad \cite{lindblad2002,lindblad2005well}, and Zhang \cite{ZhangYH} gives an alternative proof by generalizing Gu-Wang \cite{gu2016construction} in incompressible MHD. In the case of nonzero surface tension, Gu-Lei \cite{GuLei2020ST} proved the LWP by vanishing viscosity method. 

Futher difficulty arises in the compressible case compared with the incompressible case due to the coupling of pressure wave and the motion of elastic medium. The only known result on the free-boundary problem in elastodynamics is Trakhinin \cite{trakhinin2016elastic} that proved the LWP under either the non-collinearity condition\footnote{The non-collinearity condition reads $|\FF_{j}\times\FF_{k}|\geq\delta>0$ for $k\neq j$. This non-collinearity condition allows us to express the gradient of the flow map in terms of the deformation tensor, and thus actually enhances extra 1/2-order regularity of the free-boundary than being under Rayleigh-Taylor sign condition.} or Rayleigh-Taylor sign condition. The proof in \cite{trakhinin2016elastic} relies on the Nash-Moser iteration and thus cannot get energy estimates without loss of regularity. Trakhinin \cite{trakhinin2016elastic} also pointed out that the ill-posedness happens if the Rayleigh-Taylor sign condition and the non-collinearity condition fail simultaneously. 

As stated in the remark after \eqref{elastobdry}, the free boundary problem can be considered as the one-phase problem of the vortex sheets in compressible elastodynamics. For that, Chen-Hu-Wang \cite{CHWWY1,CHWWY2} proved the linearized stability of vortex sheets in 2D compressible elastodynamics and Chen-Hu-Wang-Yuan \cite{CHWWY4} for the 3D case.  Chen-Hu-Wang-Wang-Yuan \cite{CHWWY3} also proved the nonlinear stability in 2D. We also mention that Chen-Secchi-Wang \cite{CSWelastic} proved the linear stability of contact discontinuity. Morando-Trakhinin-Trebeschi \cite{elasticshock} proved the structural stability of shocks in compressible elastodynamics. Among all these results, the proofs strongly rely on Nash-Moser iteration. \textbf{So far, no nonlinear energy estimate without loss of regularity is available for the free-boundary compressible elastodynamics system}.

The other topic considered in this manuscript is the incompressible limit. In physics, the incompressible limit helps people to understand the property of a slightly compressible fluid via its incompressible counterpart and vice versa. In the mathematical study of inviscid fluid, most results on the incompressible limit or slightly compressible fluid focus on the Euler equations in fixed domain or $\R^n$  \cite{klainerman1981,klainerman1982,ebin1982,Schochet1986,Sideris91,Metivier2001,alazard2005,Cheng1,Cheng2,disconziebin} and reference therein. The incompressible limit method was also used in elastodynamics \cite{Sideris04,LZ05,LiuXu2021}. In the study of the incompressible limit of a free-surface fluid, only quite few results are available. See Lindblad-Luo \cite{lindblad2018priori} for compressible Euler, Luo \cite{luo2018ww} for compressible water wave, Disconzi-Luo \cite{luo2019limit} for the nonzero surface tension case and the author \cite{ZhangCRMHD1} for compressible resistive MHD. The study of incompressible limit of free-boundary compressible elastodynamics is still open. 

In the presenting manuscript, we prove the local well-posedness of the free-boundary compressible ideal elastodynamics system under Rayleigh-Taylor sign conditon \eqref{sign} and the energy estimate without regularity loss. The energy estimates are also uniform in sound speed and thus yield the incompressible limit. Our proof no longer relies on higher order wave equation as in previous work \cite{lindblad2018priori, luo2018ww, ZhangCRMHD1} and also applies to Euler equations. Besides, we are able to recover the higher order energy established in \cite{lindblad2018priori, luo2018ww, ZhangCRMHD1} for a slightly compressible elastic medium. To the best of our knowledge, we first establish the energy estimates without loss of regularity and the incompressible limit of the free-boundary compressible elastodynamics system under Rayleigh-Taylor sign condition. See Section \ref{stat} for detailed strategy of the proof.

\subsection{Reformulation in Lagrangian coordinates and main results}

We introduce Lagrangian coordinates to reduce the free-boundary problem to a fixed-domain problem. We introduce the reference domain\footnote{The reference domain allows us to work in one coordinate patch and the result for a general simply-connect domain follows from partition of unity. See Coutand-Shkoller \cite{coutand2007LWP,coutand2010LWP} for details.} $\Omega:=\T^2\times(-1,1)$ with boundary $\Gamma:=\p\Omega=\T^2\times(\{-1\}\cup\{1\})$. Denote the coordinates on $\Omega$ by $y=(y_1,y_2,y_3)$ and the spatial derivative in the Lagrangian coordiantes by $\p=\p_y$. We define $\eta:[0,T]\times\Omega\to\DD$ to be the flow map of the velocity $u$ by
\[
\p_t\eta(t,y)=u(t,\eta(t,y)),~~~\eta(0,y)=\eta_0(y),
\] where $\eta_0:\Omega\to \DD_0$ is a diffeomorphism. For technical simplicity, we take $\eta_0=\text{Id}$, i.e., we assume the initial domian $\DD_0=\Omega=\T^2\times(0,1)$. In fact, our proof in the manuscript is also applicable to the case for general data $\eta_0\in H^4$. By the chain rule, one can verify that $\p_t(f(t,\eta(t,y))=(D_t f)(t,\eta(t,y))$, i.e., the material derivative becomes time derivative in Lagrangian coordinates.

Next we introduce the Lagrangian variables
\[
v(t,y):=u(t,\eta(t,y)),~h(t,y):=\h(t,\eta(t,y)),~F_{ij}(t,y):=\FF_{ij}(t,\eta(t,y)),~F_j(t,y):=\FF_{j}(t,\eta(t,y)),
\] the co-factor matrix $a:=[\p\eta]^{-1}$ by $a^{li}=a^l_i:=\frac{\p y^l}{\p x^i}$ where $x^i:=\eta^i(t,y)$ is the $i$-th component of Eulerian coordinate, the Jacobian determinant $J:=\det[\p\eta]$ and $A:=Ja$. These quantities are always well-defined because $\eta$ is around the identity map when $t$ is small. Then the free-boundary compressible elastodynamic system \eqref{elastoequ}-\eqref{elastobdryu} can be re-written in Lagrangian coordinates as
\begin{equation}\label{elastoL}
\begin{cases}
\p_t\eta=v~~~&\text{in }\Omega,\\
\p_t v=-\pa h+\sum\limits_{j=1}^3(F_j\cdot\pa)F_j~~~&\text{in }\Omega,\\
\dive_a v=-e'(h)\p_t h~~~&\text{in }\Omega,\\
\p_t F_j=(F_j\cdot\pa)v~~~&\text{in }\Omega,\\
(\dive_a \FT)_j=-e'(h)(F_j\cdot\pa)h~~~&\text{in }\Omega,\\
\p_t|_{\Gamma}\in\mathcal{T}([0,T]\times\Gamma)~~~&\text{on }\Gamma,\\
h=0~~~&\text{on }\Gamma,\\
F_j\cdot N=0~~~&\text{on }\Gamma,\\
-\frac{\p h}{\p N}\geq c_0>0~~~&\text{on }\Gamma,\\
(\eta,v,h,F)|_{t=0}=(\text{Id},v_0,\h_0,\F).
\end{cases}
\end{equation}Here $N=(0,0,\pm 1)$ is the exterior unit normal on the boundary $\Gamma:=\T^2\times\{\pm 1\}$, and $(\pa f)^i:=a^{li}\p_l f$. Note that the divergence constraint on $F$ and Taylor sign condition are only required for initial data, so the system \eqref{elastoL} is not over-determined.

There are several important geometric identities in Lagrangian coordinates. We have Piola's identity $\p_l A^{li}=0$ for any $i$, and 
\[
Da^{li}=-a^{lr}\p_m D\eta_r a^{mi},~~D=\p\text{ or }\p_t,
\]and $\p_t J=J\dive_a v$. With the help of these identities, one can express the deformation tensor as the $\F$-directional derivative of the flow map. We compute for any $j$
\[
\p_t(F_{ij}a^{li})=\p_t F_{ij}a^{li}+F_{ij}\p_ta^{li}=F_{kj}a^{mk}\p_m v_ia^{li}-F_{ij}a^{lr}\p_mv_ra^{mi}=0,
\]which yields $F_{ij}a^{li}=\F_{ij}\delta^{li}=\F_{lj}$ and thus
\begin{equation}\label{fid}
F_{kj}=F_{ij}a^{li}\p_l\eta_k=\F_{lj}\p_l\eta_k=\FP\eta_k.
\end{equation} From now on, $\dive Y:=\p_i Y^i$ denotes the standard Lagrangian divergence instead of the Eulerian one. Under this setting, the divergence constraint becomes 
\begin{equation}\label{divf0}
\dive({\F}^{\top})_j:=\p_k\F_{kj}=-e'(\h_0)\FP\h_0.
\end{equation}

Therefore, the system \eqref{elastoL} can be further simplified to the following system. 
\begin{equation}\label{elastol}
\begin{cases}
\p_t\eta=v~~~&\text{in }\Omega,\\
\p_t v=-\pa h+\sum\limits_{j=1}^3\FP^2\eta~~~&\text{in }\Omega,\\
\dive_a v=-e'(h)\p_t h~~~&\text{in }\Omega,\\
\dive({\F}^{\top})_j:=\p_k\F_{kj}=-e'(\h_0)\FP\h_0~~~&\text{in }\Omega,\\
\p_t|_{\Gamma}\in\mathcal{T}([0,T]\times\Gamma)~~~&\text{on }\Gamma,\\
h=0~~~&\text{on }\Gamma,\\
\F_j\cdot N=0~~~&\text{on }\Gamma,\\
-\frac{\p h}{\p N}\geq c_0>0~~~&\text{on }\Gamma,\\
(\eta,v,h)|_{t=0}=(\text{Id},v_0,\h_0).
\end{cases}
\end{equation}

We aim to prove the local well-posedness and the incompressible limit of the system \eqref{elastol}, i.e., the free-boundary compressible elastodynamic system in Lagrangian coordinates. Before stating the main results, we introduce our energy functional as follows. Here we denote $\|f\|_{s}: = \|f(t,\cdot)\|_{H^s(\Omega)}$ for any function $f(t,y)\text{ on }[0,T]\times\Omega$ and $|f|_{s}: = |f(t,\cdot)|_{H^s(\Gamma)}$ for any function $f(t,y)\text{ on }[0,T]\times\Gamma$.

\begin{defn}
Define energy functional $\EE$ at time $T$ to be
\begin{equation}\label{energy}
\begin{aligned}
\EE(T)&:=\left\|\eta\right\|_4^2+\left\|v\right\|_4^2+\sum\limits_{j=1}^3\left\|\FP\eta\right\|_4^2+\left\|h\right\|_4^2+\left|a^{3i}\TP^4\eta_i\right|_0^2\\
&+\left\|\p_t v\right\|_3^2+\sum\limits_{j=1}^3\left\|\FP\p_t\eta\right\|_3^2+\left\|\p_t h\right\|_3^2\\
&+\left\|\p_t^2v\right\|_2^2+\sum\limits_{j=1}^3\left\|\FP\p_t^2\eta\right\|_2^2+\left\|\swt\p_t^2 h\right\|_2^2\\
&+\left\|\swt\p_t^3 v\right\|_1^2+\sum\limits_{j=1}^3\left\|\swt\FP\p_t^3 \eta\right\|_1^2+\left\|\wt\p_t^3 h\right\|_1^2\\
&+\left\|\wt\p_t^4 v\right\|_0^2+\sum\limits_{j=1}^3\left\|\wt\FP\p_t^4 \eta\right\|_0^2+\left\|(\wt)^{\frac32}\p_t^4 h\right\|_0^2.
\end{aligned}
\end{equation}
And define higher order energy functional $\EEE(T)$ to be
\begin{equation}\label{energyh}
\begin{aligned}
\EEE(T):=\EE(T)+\left\|\swt\p_t^4 v\right\|_1^2+\left\|\swt\p_t^4 \left(\FP\eta\right)\right\|_1^2+\left|\swt a^{3i}\TP\p_t^4\eta_i\right|_0^2+\left\|(\wt)^{\frac32}\p_t^5 h\right\|_0^2+\left\|\wt\p_t^4 h\right\|_1^2.
\end{aligned}
\end{equation}
\end{defn}

The main results in this manuscript are listed in the following 4 theorems. 

\paragraph*{1. Local well-posedness}
\begin{thm}\label{lwp}
Given initial data $v_0,\F,\h_0\in H^4(\Omega)$ satisfying the compatibility conditions \eqref{ccd} up to 4-th order\footnote{The reason for requiring 4-th order is that $\p_t^4 h$ appears in the boundary integral in the proof.} and the Rayleigh-Taylor sign condition \eqref{sign}, there exists some $T>0$ depending only on the initial data, such that the system \eqref{elastol} has a unique strong solution $(\eta,v,h)$ with the energy estimates
\begin{equation}\label{energy1}
\sup_{0\leq t\leq T}\EE(t)\leq P\left(\|v_0\|_4,\|\F\|_4,\|\h_0\|_4\right),
\end{equation}where $P(\cdots)$ represents a polynomial in its arguments.
\end{thm}
\begin{flushright}
$\square$
\end{flushright}

\paragraph*{2. Incompressible limit}
We parametrize the sound speed by parametre $\eps>0$ in such that
\[
p'_{\eps}(\rho)|_{\rho=1}=\eps.
\] We consider the compressible elastodynamic equations with variable $(\eta^{\eps},v^{\eps},F^{\eps},h^{\eps})$ if the sound speed satisfies $p'_{\eps}(\rho)|_{\rho=1}=\eps$. Under this setting, the density $\rho^{\eps}(h)\to 1$ and thus the enthalpy $e(h^{\eps})$ converges to 0 as $\eps\to\infty$. The above energy estimates are actually uniform in the sound speed, i.e., the energy bound of $\EE(T)$ does not rely on $\wt^{-1}$. 

For every $\eps>0$, let $(v_0^{\eps},{\F}^{\eps},\h_0^{\eps})\in H^4(\Omega)\times H^4(\Omega)\times H^4(\Omega)$ be the initial data with sound speed $\eps$ of the free-boundary compressible elastodynamic system \eqref{elastol} satisfying the compatibility conditions up to 4-th order. Then we can establish the incompressible limit in the following sense: 
\begin{thm}\label{limit}
Let $\vvv\in H^4(\Omega)$ be a divergence-free vector field and $\G\in H^4(\Omega)$ be a divergence-free matrix in the sense of $\p_k\G_{kj}=0$ for all $j$. Define $Q_0$ be the solution to the elliptic equation with constraints $-\frac{\p Q}{\p N}|_{\Gamma}\geq c_0>0$
\[
\begin{cases}
-\Delta Q_0=\p_i\vvv^k\p_k\vvv^i-\p_i\G_{kj}\p_k\G_{ij}~~~&\text{in }\Omega,\\
Q_0=0~~~&\text{on }\Gamma.\\
\end{cases}
\] Let $(V,G,Q)$ be the solution to the free-boundary incompressible elastodynamic system with initial data $(\vvv,\G,Q_0)\in H^4(\Omega)\times H^4(\Omega)\times H^4(\Omega)$:
\begin{equation}\label{elastoi}
\begin{cases}
\p_t\eta=V~~~&\text{in }\Omega,\\
\p_t V=-\pa Q+\sum\limits_{j=1}^3(G\cdot\pa)G~~~&\text{in }\Omega,\\
\dive_a V=0~~~&\text{in }\Omega,\\
(\dive G^{\top})_j:=\p_k G_{kj}=0~~~&\text{in }\Omega,\\
\p_t|_{\Gamma}\in\mathcal{T}([0,T]\times\Gamma)~~~&\text{on }\Gamma,\\
Q=0~~~&\text{on }\Gamma,\\
G_j\cdot N=0~~~&\text{on }\Gamma,\\
-\frac{\p Q}{\p N}\geq c_0>0~~~&\text{on }\Gamma,\\
(\eta,V,G,Q)|_{t=0}=(\text{Id},\vvv,\G,Q_0).
\end{cases}
\end{equation}  
Suppose also there exists a sequence of initial data of \eqref{elastoL} $(v_0^{\eps},{\F}^{\eps},\h^{\eps})\xrightarrow{C^1}(\vvv,\G,Q_0)$ as $\eps\to\infty$ and satisfies the compatibility conditions up to 4-th order. Then there exist some $T_0>0$ independent of $\eps$ such that 
\begin{enumerate}
\item The corresponding energy functionals $\EE^{\eps},\EEE^{\eps}$ are bounded uniformly in $\eps$ in $[0,T_0]$.
\item The corresponding solutions $(v^{\eps},F^{\eps},h^{\eps})\xrightarrow{C^1}(V,G,Q)$ as $\eps\to\infty$ in $[0,T_0]$.
\end{enumerate}
\end{thm}
\begin{flushright}
$\square$
\end{flushright}

\begin{rmk}
In fact, as $\eps\to \infty$ (i.e., $e'(h^{\eps})\to 0$), $\EE(T)$ and $\EEE(T)$ converge exactly to the energy functional of the free-boundary incompressible elastodynamic system under Rayleigh-Taylor sign condition $-\p_3Q\geq c_0'>0$ established by Zhang \cite{ZhangYH} 
\[
\EE^{\infty}:=\left\|\eta\right\|_4^2+\left\|V\right\|_4^2+\sum\limits_{j=1}^3\left\|\GP\eta\right\|_4^2+\left\|Q\right\|_4^2+\left|a^{3i}\TP^4\eta_i\right|_0^2.
\]
 The reason is that one can invoke the second equation to reduce time derivatives of $v$ to spatial derivatives of $h$ and $\FP\eta$ and the weighted terms converge to zero.
\end{rmk}

\paragraph*{3. Enhanced regularity in the slightly compressible case}
In Lindblad-Luo \cite{lindblad2018priori} and Luo \cite{luo2018ww} on the compressible Euler equations, they required the $H^1$-control of $\p_t^4$-derivatives to close the energy estimates and also the incompressible limit. In Theorem \ref{lwp} and \ref{limit} we are able to drop that higher order energy. Moreover, we are still able to prove such higher regularity of full time derivatives for the elastodynamics equations when the elastic medium is slightly compressible, i.e., $|e'(h)|\ll 1$ is sufficiently small. The following result shows that our method are also applicable for compressible Euler equations and also recover the previous results \cite{lindblad2018priori,luo2018ww}.

\begin{thm}\label{enhance}
Suppose the initial data satisfies the compatibility conditions up to 5-th order. When the sound speed $\eps$ is sufficiently large, the solution $(v,F,h)$ constructed in Theorem \ref{lwp} satisfies higher order energy estimates uniform in $\epsilon$ in $[0,T_1]$ where $0<T_1<T$ only depends on the initial data.
\begin{equation}\label{energy2}
\sup_{0\leq t\leq T_1}\EEE(t)\leq P\left(\|v_0\|_4,\|\F\|_4,\|\h_0\|_4\right).
\end{equation}
\end{thm}
\begin{flushright}
$\square$
\end{flushright}

\paragraph*{4. Existence of initial data satisfying the compatibility conditions}

Finally, we proved the existence of initial data which satisfies the compatibility conditions up to 5-th order\footnote{The reason for requiring 5-th order is that $\p_t^5 h$ appears on the boundary integral in the analysis of enhanced regularity.} and strongly converges to the incompressible data as the sound speed goes to infinity.  Define $M_0:=\|v_0\|_4^2+\|\F\|_4^2+\|\h_0\|_4^2$. 

\begin{thm}\label{dataexist}
Given the initial data $(\vvv,\G,Q_0)\in H^5\times H^5\times H^5$ of the incompressible elastodynamics system \eqref{elastoi}, there exist initial data $(v_0^{\eps},{\F}^{\eps},\h_0^{\eps})\in H^5\times H^5\times H^5$ of \eqref{elastoL} such that 
\begin{enumerate}
\item $(v_0^{\eps},{\F}^{\eps},\h_0^{\eps})\xrightarrow{C^1}(\vvv,\G,Q_0)$ as $\eps\to+\infty$.
\item $(v_0^{\eps},{\F}^{\eps},\h_0^{\eps})$ satisfies the compatibility conditions \eqref{ccd} up to 5-th order.
\item The energy functionals satisfy $\EE(0)+\EEE(0)\leq P(M_0)$.
\end{enumerate}
\end{thm}
\begin{flushright}
$\square$
\end{flushright}

\subsection{Strategy of the proof}\label{stat}

Our strategy is different from the previous works on either compressible Euler equations or incompressible elastodynamic equations. In particular, the simultaneous presence of deformation tensor and compressibility makes the fixed-point argument in \cite{coutand2007LWP,coutand2010LWP,GLL2019LWP, luozhangCWWLWP,GuWang2018, gu2016construction, ZhangYH} no longer applicable to solve the linearized system even if one has the a priori estimates without regularity loss. Now we give illustrations on the strategy of the proof.  Keep in mind that $\F_j\cdot N|_{\Gamma}=0$ implies \textbf{$\FP$ is a tangential derivative on the boundary.}

\subsubsection{Control of the energy functional}\label{stat1}

Let us temporarily focus on the energy estimates of the original equation \eqref{elastol} instead of constructing the approximation solution. The first step is div-curl-tangential decomposition
\begin{align*}
\|v\|_4\lesssim&\|v\|_0+\|\curl v\|_3+\|\dive v\|_3+|\TP v\cdot N|_{5/2},\\
\|\FP\eta\|_4\lesssim&\|\FP\eta\|_0+\|\curl \FP\eta\|_3+\|\dive \FP\eta\|_3+|\TP(\FP\eta)\cdot N|_{5/2}.
\end{align*} The curl estimates can be directly controlled via the evolution equation. The boundary part can be reduced to interior tangential estimates and divergence estimates by the normal trace lemma (cf. Lemma \ref{normaltrace}).  To estimate $\|\TP^4 v\|_0$ and $\|\TP^4 \FP\eta\|_0$, one cannot directly take $\TP^4$ in the equation $\p_t v=-\pa h+\sum\limits_{j=1}^3\FP^2\eta$ and compute the $L^2$-type estimates because this requires the control of $\|[\TP^4,a^{li}]\p_l h\|_0$ where $\TP^4 a\approx \TP^4\p\eta\times\p\eta$ cannot be controlled in $L^2$. For Euler equations \cite{coutand2007LWP,coutand2010LWP} one can integrate $\TP^{1/2}$ by parts and control $\|\curl\eta\|_{7/2}$, but this is no longer applicable to elastodynamics equations: The preserved property of irrotationality for Euler equations, which is the key to the enhanced regularity of flow map $\eta$ than $v$, no longer holds for elastodynamics equations. Instead, $\|\FP\eta\|_{9/2}$ is necessary for $\|\curl\eta\|_{7/2}$ but impossible for us to control. 

To overcome such difficulty, we use Alinhac good unknown method which reveals that the ``essential" highest order term in $\TP^4(\pa f)$ should be the covariant derivative of the ``Alinhac good unknown" instead of simply commute $\TP^4$ with $\pa$. In other words, the main idea is to rewrite $\TP^4(\pa f)$ to be the sum of convariance part $\pa \mathsf{F}$ plus an $L^2(\Omega)$-bounded term $C(f)$ such that
\[
\TP^4(\pa f)=\pa\mathsf{F}+C(f), \text{ with } \|\mathsf{F}-\TP^4 f\|_0+\|\p_t(\mathsf{F}-\TP^4 f)\|_0+\|C(f)\|_0\leq P(\EE(t)).
\] The quantity $\mathsf{F}$ is called the ``Alinhac good unknown" of $f$ with respect to $\TP^4$. That is to say, Alinhac's good unknowns allow us to take into account the covariance under the change of coordinates to avoid the extra regularity assumption on the flow map.

To derive the precise form of $\mathsf{F}$, one can use chain rule $\TP_i=\dfrac{\p}{\p y^i}=\dfrac{\p \eta^{\alpha}}{\p y^i}\dfrac{\p}{\p \eta^{\alpha}}=\TP_i \eta\cdot\pa$:
\[
\pa(\TP^4 f)=\pa\left((\TP\eta\cdot\pa)(\TP\eta\cdot\pa)(\TP\eta\cdot\pa)(\TP\eta\cdot\pa)f\right)=\TP^4(\pa f)+\pa(\TP^4\eta\cdot\pa f)+\text{l.o.t.}.\
\]It is not difficult to find that the term $\TP^4\eta\cdot\pa f$ also enters to the highest order, which should be merged into the covariance part. Therefore, $\mathsf{F}:=\TP^4 f-\TP^4\eta\cdot\pa f$ satisfies our requirement. Such crucial fact was first observed by Alinhac \cite{alinhacgood89}, and then widely used in the study of first-order quasilinear hyperbolic system. In the study of free-surface fluid, this was first implicity applied by Christodoulou-Lindblad \cite{christodoulou2000motion}, and later there are some explicit applications such as  \cite{MRgood2017,wangxingood,gu2016construction,GuWang2018,ZhangYH,luozhangCWWLWP,ZhangCRMHD2}. 

Now we introduce the Alinhac good unknowns $\VV:=\TP^4 v-\TP^4\eta\cdot\pa v$ and $\HH:=\TP^4 h-\TP^4\eta\cdot\pa h$ and get the evolution equations. Below, $\cdots$ stands for controllable terms.
\begin{equation}\label{goodeq0}
\p_t\VV+\pa\HH-\sum\limits_{j=1}^{3} \FP\left(\TP^4\FP\eta\right)=\cdots,
\end{equation}subject to
\begin{align}
\label{0hhbdry}\pa\cdot\VV=\TP^4(\dive_a v)+C(v)\text{ in }\Omega\text{ and }\HH=(-\p_3h)a^{3k}\TP^4\eta_k\text{ on }\Gamma,
\end{align}where $C(f)$ is defined to be the error terms satisfying $\TP^4\pa f=\pa(\TP^4 f-\TP^4\eta\cdot\pa f)+C(f)$ and is controllable. Then the standard $L^2$-estimate of \eqref{goodeq0} gives the boundary integral contributed by the free surface
\begin{equation}\label{tgB00}
\begin{aligned}
&\ig\frac{\p h}{\p N} \TP^4\eta_k a^{3k}a^{3i}\VV_i\dS=\ig\frac{\p h}{\p N} \TP^4\eta_k a^{3k}a^{3i}(\TP^4 \p_t\eta_i-\TP^4 \eta_ra^{lr}\p_l v_i)\dS\\
=&\frac12\frac{d}{dt}\ig \frac{\p h}{\p N}\left|a^{3i}\TP^4\eta_i\right|^2\dS+\cdots\\
&-\ig\frac{\p h}{\p N}a^{3k}\TP^4 \eta_k \p_ta^{3i}\TP^4\eta_i\dS -\ig\frac{\p h}{\p N} \TP^4\eta_k a^{3k}a^{3i}\TP^4 \eta_ra^{lr}\p_l v_i\dS.
\end{aligned}
\end{equation}Invoking the Taylor sign condition \eqref{sign} we get the boundary energy $\left|a^{3i}\TP^4\eta_i\right|_0^2$. Then plugging $\p_ta^{3i}=-a^{3r}\p_l v_r a^{li}$ into the second term yields the cancellation structure: The last two terms in \eqref{tgB00} can be controlled by the boundary energy when $l=3$, and exactly cancel each other when $l=1,2$. 

\paragraph*{Key observation: Structure of wave equation of $h$}

From \eqref{elastol} we can deduce that $\dive_a v=-\wt\p_t h$ and $\dive_a\FP\eta=-e'(h)\FP h$. This motivates us to consider the wave equation of $h$ derived by taking $\dive_a$ in the second equation of \eqref{elastol}
\begin{equation}\label{h00wave}
\begin{aligned}
\wt \p_t^2 h-\Delta_a h=\wt\sum_{j=1}^3\FP^2 h+\text{quadratic terms of first order derivative}.
\end{aligned}
\end{equation}

Although the RHS of \eqref{h00wave} has a potential to lose regularity due to the same order of derivatives, we are still able to do the energy estimates thanks to special structure of $\FP^2 h$. If we take $L^2$ inner product of \eqref{h00wave} and $\p_t h$, the LHS gives the energy terms $\frac12\frac{d}{dt}\io \wt|\p_t h|^2+|\pa h|^2\dy$ containing $\dive_a v$. The RHS now becomes
\[
\io \wt \FP^2 h\cdot\p_t h=-\frac12\frac{d}{dt}\io\wt|\FP h|^2\dy+\cdots,
\]which exactly gives the control of $\dive_a \FP\eta$. However, the $H^3$-control of $\dive_a v$ and $\dive_a \FP\eta$ is not a direct result of the $L^2$-control of $\p^3$-differentiated wave equation \eqref{h00wave} because integrating the Laplacian term by parts yields a boundary integral which cannot be controlled due to the presence of normal derivatives.

Our idea to overcome this difficulty is to use Christodoulou-Lindblad type elliptic estimates in Lagrangian coordinates (Lemma \ref{GLL}) and invoke \eqref{h00wave} to replace two normal derivatives of $h$ by two tangential derivatives. Repeatedly, the estimates of $h$ can be reduced to tangential derivatives of $h,\p h$ which can be controlled by tangentially-differentiated wave equation \eqref{h00wave}. Specifically, we start with 
\[
\|h\|_4\approx \|\pa h\|_3\lesssim P(\|\eta\|_3)(\|\Delta_a h\|_2+\|\TP\eta\|_3\|h\|_3).
\]Then invoking \eqref{h00wave} we can reduce $\|\Delta_a h\|_2$ to $\|\wt\p_t^2 h\|_2$ and $\|\wt\FP^2 h\|_2$ plus lower order terms. Note that $\p_t$ and $\FP$ are tangential derivatives, so we gain 2 tangential derivatives once we proceed such a step. Then we can do the same thing for  $\|\wt\p_t^2 h\|_2$ and $\|\wt\FP^2 h\|_2$. For example, 
\[
\|\wt \p_t^2 h\|_2\approx\|\wt\pa \p_t^2 h\|_1\lesssim P(\|\eta\|_2)(\|\wt\Delta_a \p_t^2 h\|_0+\|\TP\eta\|_2\|\wt\p_t^2 h\|_0),
\]then plugging the $\p_t^2$-differentiated wave equations \eqref{h00wave} reduces $\|\wt\Delta_a \p_t^2 h\|_0$ to $\|(\wt)^2\p_t^4 h\|_0$ and $\|(\wt)^2\p_t^2\FP^2 h\|_0$ which are purely tangential derivatives of $h$.

In summary, we have the following reduction procedure
{\small\begin{equation}\label{reduceh}
\begin{tikzcd}
& \|\wt\p_t^2 h\|_2\arrow[r]\arrow[dr]&\|(\wt)^2\p_t^4 h\|_0\\
\|h\|_4\arrow[ur]\arrow[dr]& &\sum\limits_{j=1}^3\|(\wt)^2\FP^2\p_t^2 h\|_0\\
&\sum\limits_{j=1}^3 \|\wt\FP^2 h\|_2\arrow[r]\arrow[ur]&\sum\limits_{j,k=1}^3\|(\wt)^2(\F_k\cdot\p)^2\FP^2 h\|_0,
\end{tikzcd}
\end{equation}} and
{\small\begin{equation}\label{reducehtF}
\begin{tikzcd}
& \|\wt\p_t^3 h\|_1& &\sum\limits_{j=1}^3\|\wt\FP \p_t^2 h\|_1\\
\|\p_t h\|_3\arrow[ur]\arrow[dr]& &\sum\limits_{j=1}^3\|\FP h\|_3\arrow[ur]\arrow[dr] &\\
&\sum\limits_{j=1}^3 \|\wt\FP^2\p_t h\|_1& &\sum\limits_{j,k=1}^3\|\wt(\F_k\cdot\p)^2\FP h\|_1.
\end{tikzcd}
\end{equation}}
Therefore, \eqref{reduceh} and \eqref{reducehtF} show that it suffices to control $\p_t^3$-differentiated, $\p_t^2\FP$-differentiated, $\p_t\FP^2$-differentiated and $\sum\limits_{k=1}^3 (\F_k\cdot\p)^2\FP$-differentiated wave equation \eqref{h00wave}, whose weighted $L^2$-estimates exactly give the energy terms that we want. See Section \ref{wavediv} for detailed computation.

The above analysis gives the control of divergence and thus finalize the control of full spatial derivatives. It now remains to control the time derivatives of $v$ and $\FP\eta$. We can invoke $\p_t v=\sum\limits_{j=1}^3\FP(\FP\eta)-\pa h$ and $\p_t(\FP\eta)=\FP v$ repeatedly to replace time derivatives by spatial derivatives until there is no time derivative. 

%

\subsubsection{Tangentially-smoothed approximation problem}\label{stat2}

\paragraph*{1. Tangential smoothing: Necessity and adjusted mollification}

With the a priori estimates of \eqref{elastol}, it is straightforward to consider the standard iteration to prove the well-posedness. Specifically, if one starts with trivial solution $(\eta^{(0)},v^{(0)},h^{(0)})=(\text{Id},0,0)$ and inductively define $(\eta^{(n+1)},v^{(n+1)},h^{(n+1)})$ by
\begin{equation}\label{elastolll}
\begin{cases}
\p_t\eta^{(n+1)}=v^{(n+1)}~~~&\text{ in }\Omega, \\
\p_t v^{(n+1)}=-\nabla_{{a^{(n)}}}h^{(n+1)}+\sum\limits_{j=1}^3\FP^2\eta^{(n+1)}~~~&\text{ in }\Omega, \\
\text{div}_{a^{(n)}}v^{(n+1)}=-e'(h^{(n)})\p_t h^{(n+1)}~~~&\text{ in }\Omega, \\
h^{(n+1)}=0~~~&\text{ on }\Gamma, \\
(\eta^{(n+1)},v^{(n+1)}, h^{(n+1)})|_{t=0}=(\text{Id},v_0, h_0)
\end{cases} 
\end{equation} 
where $a^{(n)}:=[\p\eta^{(n)}]^{-1}$, then we have to control $\|v\|_4$ when solving the linearized system where $\|\nabla_{{a^{(n)}}}h^{(n+1)}\|_4$ is required. Invoking Lemma \ref{GLL} yields
\begin{equation}\label{ell5}
\|\nabla_{{a^{(n)}}}h^{(n+1)}\|_4\lesssim P(\|\eta^{(n)}\|_4)\left(\|\Delta_{a^{(n)}} h^{(n+1)}\|_3+\|\TP\eta^{(n+1)}\|_4\|h^{(n+1)}\|_4\right),
\end{equation}where $\|\TP\eta^{(n+1)}\|_4$ loses one tangential derivative. Another difficulty appears in the uniform-in-$n$ estimates of the linearized system \eqref{elastolll} which is necessary for iteration. The boundary integral \eqref{tgB00} appearing in the tangential estimates now becomes
\begin{equation}\label{b45}
\ig\frac{\p h^{(n+1)}}{\p N} \underbrace{\tpl\eta_k^{(n)}}_{n\text{-th solution}} a^{(n)3k}a^{(n)3i}(\underbrace{\p_t\tpl\eta^{(n+1)}_i}_{(n+1)\text{-th solution}}-\tpl\eta_k^{(n)}a^{(n)lk}\p_l v^{(n+1)})\dS
\end{equation}which no longer exhibits the cancellation structure or produces the boundary energy term $\left|a^{3i}\TP^4\eta_i\right|_0^2$.

Note that such derivative losses are both tangential and even appear in the study of Euler equations \cite{coutand2007LWP,coutand2010LWP,coutand2012LWP,GLL2019LWP,luozhangCWWLWP}. This motivates people to tangentially mollify the coefficients $a$ to compensate the loss of regularity. Such tangential smoothing method is first introduced by Coutand-Shkoller \cite{coutand2007LWP,coutand2010LWP}. Define $\lkk$ to be the standard mollifier with parametre $\kk>0$ on $\R^2$ as in \eqref{lkk0}. In \cite{coutand2007LWP,coutand2010LWP}, they defined $\ek:=\lkk^2 \eta$ and $\ak=[\p\ek]^{-1}$ and then defined the nonlinear $\kk$-approximation problem by replacing $a$ with $\ak$. However, such construction is not applicable to the elastodynamics system because we have to control the term $\|[\FP,\lkk^2]\eta\|_4$ in which there is a normal derivative $\F_{3j}\p_3$. Motivated by Gu-Wang \cite{gu2016construction}, we can mollify $\eta$ on the boundary first, then do the harmonic extensin to the interior
\begin{equation}
\begin{cases}
-\Delta \ek=-\Delta\eta~~~&\text{ in }\Omega,\\
\ek=\lkk^2\eta~~~&\text{ on }\Gamma.
\end{cases}
\end{equation} Denote $\ak:=[\p\ek]^{-1}$ to be the cofactor matrix, $\Jk:=\det[\p\ek]$ to be the Jacobian determinant and $\Ak:=\Jk\ak$. Then we define the nonlinear $\kk$-approximation system of \eqref{elastol} to be
\begin{equation}\label{elastolkk00}
\begin{cases}
\p_t\eta=v+\psi~~~&\text{in }\Omega,\\
\p_t v=-\pak h+\sum\limits_{j=1}^3\FP^2\eta~~~&\text{in }\Omega,\\
\diva v=-e'(h)\p_t h~~~&\text{in }\Omega,\\
\dive({\F}^{\top})_j:=\p_k\F_{kj}=-e'(\h_0)\FP\h_0~~~&\text{in }\Omega,\\
\p_t|_{\Gamma}\in\mathcal{T}([0,T]\times\Gamma)~~~&\text{on }\Gamma,\\
h=0,\F_j\cdot N=0~~~&\text{on }\Gamma,\\
-\frac{\p h}{\p N}\geq c_0>0~~~&\text{on }\Gamma,\\
(\eta,v,h)|_{t=0}=(\text{Id},v_0,\h_0).
\end{cases}
\end{equation}Here the term $\psi=\psi(\eta,v)$ is a correction term which solves the following Laplacian equation
\begin{equation}\label{psi00}
\begin{cases}
\Delta \psi=0  &~~~\text{in }\Omega, \\
\psi=\TL^{-1}\mathbb{P}_{\neq 0}\sum\limits_{L=1}^2\left(\TL\eta_l\ak^{Lk}\TP_L\lkk^2 v-\TL\lkk^2\eta_k\ak^{Lk}\TP_L v\right) &~~~\text{on }\Gamma,
\end{cases}
\end{equation}where $\mathbb{P}_{\neq 0} $ denotes the standard Littlewood-Paley projection in $\T^2$  which removes the zero-frequency part. $\TL:=\p_1^2+\p_2^2$ denotes the tangential Laplacian operator. The special structure of $\psi$ also indicates that it is more convenient to replace $\TP^4$ by $\tpl$ in the tangential estimates.

We emphasize that it seems necessary to add a correction term $\psi$ defined in \eqref{psi00} to the flow map equation. The reason is that the mollification breaks the cancellation structure in the boundary integral \eqref{tgB00}, which now becomes 
\begin{equation}\label{tgB01}
\begin{aligned}
&\ig\frac{\p h}{\p N} \tpl\lkk^2\eta_k \ak^{3k}\ak^{3i}(\tpl \p_t\eta_i-\tpl \ek_r\ak^{lr}\p_l v_i)\dS\\
=&\frac12\frac{d}{dt}\ig \frac{\p h}{\p N}\left|\ak^{3i}\tpl\lkk\eta_i\right|^2\dS+\cdots\\
+\sum_{L=1}^2&\left(\ig\frac{\p h}{\p N}\ak^{3k}\tpl\lkk\eta_k\ak^{3r}\TP_L\lkk^2v_r\ak^{Li}\tpl\lkk\eta_i\dS -\ig\frac{\p h}{\p N}  \ak^{3k}\tpl\lkk^2\eta_k\ak^{3i}\tpl\lkk^2\eta_ra^{Lr}\p_L v_i\dS\right),
\end{aligned}
\end{equation}of which the last line is zero when $\kk=0$ according to the cancellation structure \eqref{tgB00} but cannot be controlled when $\kk>0$. In Ginsberg-Lindblad-Luo \cite{GLL2019LWP} they introduced $\kk^2$-weighted higher order terms in the energy functional which made the proof quite complicated. Here, such term can be exactly cancelled with the contribution of $\psi$ on $\Gamma$
\[
\ig\frac{\p h}{\p N}\ak^{3i}\ak^{3k}\tpl\ek_k\tpl\psi_i\dS.
\] 
\begin{rmk}
After introducing the correction term $\psi$, we no longer have $\p_t\eta=v$. We note that the contribution of $\psi$ is controllable and is analyzed in Lemma \ref{etapsi} and Section \ref{hFFFdiv}. Therefore, we are able to follow the strategy in Section \ref{stat1} to prove the uniform-in-$\kk$ a priori estimates of \eqref{elastolkk00}.
\end{rmk}

\subsubsection{Construction of solutions to the approximation system}\label{stat3}

To prove the local well-posedness of \eqref{elastol}, it now remains to construct the unique strong solution to the nonlinear $\kk$-approximation problem \eqref{elastolkk00} for each \textbf{fixed} $\kk>0$. According to the analysis above, it now seems hopeful to proceed the linearization and fixed-point argument because the coefficient $a$ has been mollified to be a $C^{\infty}$-function $\ak$. Let us consider the linearized $\kk$-approximation system \eqref{elastollr}
\begin{equation}\label{elastollr00}
\begin{cases}
\p_t\eta=v+\psir~~~&\text{in }\Omega,\\
\p_t v=-\park h+\sum\limits_{j=1}^3\FP^2\eta~~~&\text{in }\Omega,\\
\divr v=-e'(\hr)\p_t h~~~&\text{in }\Omega,\\
\dive \F_j:=\p_k\F_{kj}=-e'(\h_0)\FP\h_0~~~&\text{in }\Omega,\\
h=0,~\F_j\cdot N=0~~~&\text{on }\Gamma,\\
(\eta,v,h)|_{t=0}=(\text{Id},v_0,\h_0),
\end{cases}
\end{equation}where the variables marked with ``ring" on top denotes the $n$-th approximation solution in the iteration scheme and the others correspond to the $(n+1)$-th approximation solution.

\paragraph*{1. Failure of fixed-point argument due to compressibility and deformation tensor}

Unfortunately, the fixed-point argument seems not applicable to solve the linearized system \eqref{elastollr00} even if we could obtain the a priori estimates without loss of regularity. The presence of Cauchy-Green tensor $\sum\limits_{j=1}^3\FP^2\eta$ makes the second equation of \eqref{elastollr00} lose one derivative in the fixed-point argument. This tells an essential difficulty in elastodynamics that the Euler equations never have. In Gu-Wang \cite{gu2016construction} and Zhang \cite{ZhangYH}, they introduced a directional viscosity term in flow map to compensate the derivative loss: $$\p_t\eta-\mu\sum\limits_{j=1}^3\FP^2\eta=v+\psir.$$  But the estimate for vanishing viscosity process strongly relies on the incompressibility assumptions. In other words, the compressibility stands for another essential difficulty compared with the study of Euler equations and incompressible elastodynamics. See detailed analysis in Section \ref{fail}. Therefore, it is not possible for us to deal with the compressible elastodynamic system by merely technical modifying the methods in the study of compressible Euler equations \cite{lindblad2018priori,GLL2019LWP,luozhangCWWLWP} and incompressible MHD \cite{gu2016construction} or elastodynamics \cite{ZhangYH}.

\paragraph*{2. Hyperbolic approach to solve the linearized $\kk$-approximation system}
We observe that if we replace the first equation of \eqref{elastollr00} by $\p_t\FP\eta=\FP v+\FP\psir$ and consider $\FFr_j:=\FP\eta$ instead of $\eta$ itself as an unknown of the linearized system, then the linearized system \eqref{elastollr00} becomes a first-order symmetric hyperbolic system with characteristic boundary conditions \cite{lax60,MetivierNotes} and thus can be solved by the arguments in Lax-Phillips \cite{lax60}. Note that once $\FFr,v,h$ are solved, the flow map $\eta$ is automatically solved by $\p_t\eta=v+\psir$. After this, we can construct the solution of the nonlinear $\kk$-approximation system \eqref{elastolkk00} for each fixed $\kk>0$ via standard Picard iteration and we refer to Section \ref{hyperbolic} for details.  Finally, the uniform-in-$\kk$ estimate for \eqref{elastolkk00} yields the local well-posedness and energy estimates of the original system \eqref{elastol}. Theorem \ref{lwp} is proven.

\subsubsection{Incompressible limit and Construction of initial data}\label{stat4}

The second goal of the presenting manuscript is to establish the incompressible limit as stated in Theorem \ref{limit}. This directly follows from the energy estimates uniform in sound speed, for which we need to set the energy functional $\EE(T)$ to be $\wt$-weighted by suitable choices of the power of weight functions. We note that our choices of weight functions in $\EE(T)$ come from the reduction procedures \eqref{reduceh}-\eqref{reducehtF}, and also help us avoid analyzing the higher order wave equation of $h$ which are different from the related previous works Lindblad-Luo \cite{lindblad2018priori}, Luo \cite{luo2018ww} and the author \cite{ZhangCRMHD1}. Finally, we need to prove the existence of the initial data satisfying the compatibility conditions \eqref{ccd} up to 5-th order. See Section \ref{data} for detailed description and construction. Once the initial datum are constructed, all the proofs of well-posedness and incompressible limit are finished.

\subsubsection{Comparison with previous works on Euler equations}\label{stat5}

In the previous works of incompressible limit of compressible free-surface inviscid fluid \cite{lindblad2018priori,luo2018ww,ZhangCRMHD1}, the weighted $H^1$-norm of full time derivative of $h$ is also required, i.e., $\|\p_t^4 h\|_1$ and $\|\p_t^5 h\|_0$. The reason is that the Christodoulou-Lindblad type energy functionals in \cite{lindblad2018priori,luo2018ww,ZhangCRMHD1} contain the boundary energy of time derivatives. Our work shows that such boundary energies are not necessary even if we require the energy to be uniform in sound speed. In other words, we show that the contribution of free boundary is presented by $\left|a^{3i}\TP^4\eta_i\right|_0^2$, exactly the tangential derivatives of the second fundamental form of the free boundary. This coincides with the core conclusion in \cite{christodoulou2000motion,lindblad2018priori,luo2018ww}: The regularity and geometry of the free surface enters to the highest order. Hence, our proof is completely applicable to compressible Euler equations just by mathematically setting $\FF=\mathbf{O}$ and avoid analyzing unnecessary terms. Moreover, we are also able to recover the enhanced regularity of full time derivatives in \cite{lindblad2018priori,luo2018ww,ZhangCRMHD1} for a slightly compressible elastic medium. The proof method is completely different from \cite{lindblad2018priori,luo2018ww,ZhangCRMHD1} because we do not allow more spatial derivatives appearing in the estimates due to the presence of deformation tensor. Such difficulty can be overcome by further delicate analysis of Alinhac good unknowns and elliptic estimates. See Section \ref{enhance0} for details. 

\subsection{Organisation of the paper}

In Section \ref{prelemma} we list preliminary lemmas which will be repeatedly used in the manuscript. In Section \ref{apriorinkk} we introduce the nonlinear $\kk$-approximation system \eqref{elastolkk} and prove the uniform-in-$\kk$ a priori estimates. Then we construct the solution to \eqref{elastolkk} by linearization, hyperbolic approach and Picard iteration in Section \ref{linearize}. These two results together with the uniqueness prove the local well-posedness of the free-boundary compressible elastodynamics system \eqref{elastol} in Section \ref{lwp0}. Then we establish the incompressible limit in Section \ref{ilimit}, and recover the enhaced regularity of full time derivatives  in Section \ref{enhance0}. Finally, the construction of initial data satisfying the compatibility conditions is shown in Section \ref{data}.

\bigskip

\noindent\textbf{List of Notations: }
\begin{itemize}
\item $\Omega:=\T^2\times(-1,1)$ and $\Gamma:=\T^2\times(\{\pm 1\})$.
\item $\|\cdot\|_{s}$:  We denote $\|f\|_{s}: = \|f(t,\cdot)\|_{H^s(\Omega)}$ for any function $f(t,y)\text{ on }[0,T]\times\Omega$.
\item $|\cdot|_{s}$:  We denote $|f|_{s}: = |f(t,\cdot)|_{H^s(\Gamma)}$ for any function $f(t,y)\text{ on }[0,T]\times\Gamma$.
\item $\|\cdot\|_{\dot{H}^s}$, $|\cdot|_{\dot{H}^s}$: Homogeneous Sobolev norm, replacing $H^s$ above by $\dot{H}^s.$
\item $P(\cdots)$:  A generic polynomial in its arguments;
\item $\PP_0$:  $\PP_0=P(\|\F\|_4,\|v_0\|_4,\|h_0\|_4)$;
\item $[T,f]g:=T(fg)-T(f) g$, and $[T,f,g]:=T(fg)-T(f)g-fT(g)$, where $T$ denotes a differential operator or the mollifier and $f,g$ are arbitrary functions.
\item $\TP,\TL$: $\TP=\p_1,\p_2$ denotes the tangential derivative and $\TL:=\p_1^2+\p_2^2$ denotes the tangential Laplacian.
\item $\nabla^{i}_a f:=a^{li}\p_{l} f$, $\dive_a\mathbf{f}:=a^{li}\p_l \mathbf{f}_{i}$ and $(\curl_a\mathbf{f})_k:=\epsilon_{kli}a^{ml}\p_m \mathbf{f}^{i}$, where $\epsilon_{kli}$ is the sign of the 3-permutation $(kli)\in S_3.$
\end{itemize}

\section{Preliminary lemmas}\label{prelemma}

We need the following lemmas in this manuscript.

\subsection{Sobolev inequalities}

\begin{lem}[\textbf{Kato-Ponce type inequalities}]  Let $J=(I-\Delta)^{1/2},~s\geq 0$. Then the following estimates hold:

(1) $\forall s\geq 0$, we have 
\begin{equation}\label{product}
\begin{aligned}
\|J^s(fg)\|_{L^2}&\lesssim \|f\|_{W^{s,p_1}}\|g\|_{L^{p_2}}+\|f\|_{L^{q_1}}\|g\|_{W^{s,q_2}},\\
\|\p^s(fg)\|_{L^2}&\lesssim \|f\|_{\dot{W}^{s,p_1}}\|g\|_{L^{p_2}}+\|f\|_{L^{q_1}}\|g\|_{\dot{W}^{s,q_2}},
\end{aligned}
\end{equation}with $1/2=1/p_1+1/p_2=1/q_1+1/q_2$ and $2\leq p_1,q_2<\infty$;

(2) $\forall s\geq 1$, we have
\begin{equation}\label{kato3}
\|J^s(fg)-(J^sf)g-f(J^sg)\|_{L^p}\lesssim\|f\|_{W^{1,p_1}}\|g\|_{W^{s-1,q_2}}+\|f\|_{W^{s-1,q_1}}\|g\|_{W^{1,q_2}}
\end{equation} for all the $1<p<p_1,p_2,q_1,q_2<\infty$ with $1/p_1+1/p_2=1/q_1+1/q_2=1/p$.
\end{lem}

\begin{proof}
See Kato-Ponce \cite{kato1988commutator}.
\end{proof}

\begin{lem}[\textbf{Trace theorem for harmonic functions}]\label{harmonictrace}
Suppose that $s\geq 0.5$ and $u$ solves the boundary-valued problem
\[
\begin{cases}
\Delta u=0~~~&\text{ in }\Omega,\\
u=g~~~&\text{ on }\Gamma
\end{cases}
\] where $g\in H^{s}(\Gamma)$. Then it holds that 
\[
|g|_{s}\lesssim\|u\|_{s+0.5}\lesssim|g|_{s}
\]
\end{lem}
\begin{proof}
The LHS follows from the standard Sobolev trace lemma, while the RHS is the property of Poisson integral, which can be found in \cite[Proposition 5.1.7]{taylorPDE1}.
\end{proof}

\begin{lem}[\textbf{Normal trace theorem}]\label{normaltrace}
It holds that for a vector field $X$
\begin{equation}\label{ntr}
\left|\TP X\cdot N\right|_{-0.5}\lesssim\|\TP X\|_0+\|\dive X\|_0
\end{equation}
\end{lem}
\begin{proof}
The proof directly follows from testing by a $H^{0.5}(\Gamma)$ function and divergence theorem. See \cite[Lemma 3.4]{gu2016construction}.
\end{proof}

\subsection{Properties of tangential smoothing operator}

Let $\zeta=\zeta(y_1,y_2)\in C_c^{\infty}(\R^2)$ be a standard cut-off function such that $\text{Spt }\zeta=\overline{B(0,1)}\subseteq\R^2,~~0\leq\zeta\leq 1$ and $\int_{\R^2}\zeta=1$. The corresponding dilation is $$\zeta_{\kk}(y_1,y_2)=\frac{1}{\kk^2}\zeta\left(\frac{y_1}{\kk},\frac{y_2}{\kk}\right),~~\kk>0.$$ Now we define
\begin{equation}\label{lkk0}
\lkk f(y_1,y_2,y_3):=\int_{\R^2}\zeta_{\kk}(y_1-z_1,y_2-z_2)f(z_1,z_2,z_3)\dz_1\dz_2.
\end{equation}

The following lemma records the basic properties of tangential smoothing.
\begin{lem}[\textbf{Regularity and Commutator estimates}]\label{tgsmooth} For $\kk>0$, we have

(1) The following regularity estimates:
\begin{align}
\label{lkk11} \|\lkk f\|_s&\lesssim \|f\|_s,~~\forall s\geq 0;\\
\label{lkk1} |\lkk f|_s&\lesssim |f|_s,~~\forall s\geq -0.5;\\ 
\label{lkk2} |\TP\lkk f|_0&\lesssim \kk^{-s}|f|_{1-s}, ~~\forall s\in [0,1];\\  
\label{lkk3} |f-\lkk f|_{L^{\infty}}&\lesssim \sqrt{\kk}|\TP f|_{0.5}.  
\end{align}

(2) Commutator estimates: Define the commutator $[\lkk,f]g:=\lkk(fg)-f\lkk(g)$. Then it satisfies
\begin{align}
\label{lkk4} |[\lkk,f]g|_0 &\lesssim|f|_{L^{\infty}}|g|_0,\\ 
\label{lkk5} |[\lkk,f]\TP g|_0 &\lesssim |f|_{W^{1,\infty}}|g|_0, \\ 
\label{lkk6} |[\lkk,f]\TP g|_{0.5}&\lesssim |f|_{W^{1,\infty}}|g|_{0.5}.
\end{align}
\end{lem}
\begin{proof}
See \cite{coutand2007LWP,gu2016construction,luozhangCWWLWP,ZhangCRMHD2}.
\end{proof}

\subsection{Elliptic estimates}
\begin{lem} [\textbf{Hodge-type decomposition}]\label{hodge}
Let $X$ be a smooth vector field and $s\geq 1$, then it holds that
\begin{equation}
\|X\|_s\lesssim\|X\|_0+\|\curl X\|_{s-1}+\|\dive X\|_{s-1}+|\TP X\cdot N|_{s-1.5}.
\end{equation}
\end{lem}
\begin{proof}
This follows from the well-known identity $-\Delta X=\curl\curl X-\nabla\dive X$ and integrating by parts.
\end{proof}

\begin{lem}[\textbf{Christodoulou-Lindblad elliptic estimate}]\label{GLL}
If $f|_{\p\Omega}=0$, then the following elliptic estimate holds for $r\geq 2$.
\begin{align}
\|\pa f\|_{r} \leq  P(\|\eta\|_{r})\Big(\|\lap_{a} f\|_{r-1}+\| \TP \eta\|_{r}\|f\|_{r}\Big).
\end{align}
\end{lem}
\begin{proof}
See Ginsberg-Lindblad-Luo \cite[Appendix B]{GLL2019LWP}. When $r=1$, $\|\eta\|_r,\|\TP\eta\|_r$ should be replaced by $H^2$-norm. The proof of this version can be found in Luo-Zhang \cite[Lemma 2.7]{luozhangCWWLWP}.
\end{proof}

\section{A priori estimates of the nonlinear approximation system}\label{apriorinkk}

We define $\ek$ to be the smoothed version of $\eta$ by 
\begin{equation}\label{lkklkk}
\begin{cases}
-\Delta \ek=-\Delta\eta~~~&\text{ in }\Omega,\\
\ek=\lkk^2\eta~~~&\text{ on }\Gamma.
\end{cases}
\end{equation} Denote $\ak:=[\p\ek]^{-1}$ to be the cofactor matrix, $\Jk:=\det[\p\ek]$ to be the Jacobian determinant and $\Ak:=\Jk\ak$. Then we define the nonlinear $\kk$-approximation system of \eqref{elastol} by 
\begin{equation}\label{elastolkk}
\begin{cases}
\p_t\eta=v+\psi~~~&\text{in }\Omega,\\
\p_t v=-\pak h+\sum\limits_{j=1}^3\FP^2\eta~~~&\text{in }\Omega,\\
\diva v=-e'(h)\p_t h~~~&\text{in }\Omega,\\
\dive({\F}^{\top})_j:=\p_k\F_{kj}=-e'(\h_0)\FP\h_0~~~&\text{in }\Omega,\\
\p_t|_{\Gamma}\in\mathcal{T}([0,T]\times\Gamma)~~~&\text{on }\Gamma,\\
h=0,\F_j\cdot N=0~~~&\text{on }\Gamma,\\
-\frac{\p h}{\p N}\geq c_0>0~~~&\text{on }\Gamma,\\
(\eta,v,h)|_{t=0}=(\text{Id},v_0,\h_0).
\end{cases}
\end{equation}Here the term $\psi=\psi(\eta,v)$ is a correction term which solves the following Laplacian equation
\begin{equation}\label{psi}
\begin{cases}
\Delta \psi=0  &~~~\text{in }\Omega, \\
\psi=\TL^{-1}\mathbb{P}_{\neq 0}\sum_{L=1}^2\left(\TL\eta_l\ak^{Lk}\TP_L\lkk^2 v-\TL\lkk^2\eta_k\ak^{Lk}\TP_L v\right) &~~~\text{on }\Gamma,
\end{cases}
\end{equation}where $\mathbb{P}_{\neq 0} $ denotes the standard Littlewood-Paley projection in $\T^2$  which removes the zero-frequency part. $\TL:=\p_1^2+\p_2^2$ denotes the tangential Laplacian operator.

\begin{rmk}
~

\begin{enumerate}
\item The correction term $\psi\to 0$ as $\kk\to 0$. We introduce it to eliminate the higher order boundary terms which appears in the tangential estimates of $v$. These higher order boundary terms are zero when $\kk=0$ but cannot be controlled when $\kk>0$.
\item The Littlewood-Paley projection is necessary here because we will repeatedly use the following inequality, otherwise the zero frequency part cannot be controlled $$|\TL^{-1}\mathbb{P}_{\neq 0}f|_{s}\approx |\mathbb{P}_{\neq 0}f|_{{H}^{s-2}}\approx |f|_{\dot{H}^{s-2}}.$$
\item The initial data is the same of origin system \eqref{elastol} because the compatibility conditions stay unchanged after mollification due to $\ak(0)=a(0)=I_3$ (identity matrix).
\end{enumerate}
\end{rmk}

In this section, we are going to prove the uniform-in-$\kk$ a priori estimates for the nonlinear $\kk$-approximation system \eqref{elastolkk}. 

For each $\kk>0$, we define the energy functional to be
\begin{equation}\label{energykk}
\begin{aligned}
\EE_{\kk}(T)&:=\left\|\eta\right\|_4^2+\left\|v\right\|_4^2+\sum\limits_{j=1}^3\left\|\FP\eta\right\|_4^2+\left\|h\right\|_4^2+\left|\ak^{3i}\tpl\lkk\eta_i\right|_0^2\\
&+\left\|\p_t v\right\|_3^2+\sum\limits_{j=1}^3\left\|\FP\p_t\eta\right\|_3^2+\left\|\p_t h\right\|_3^2\\
&+\left\|\p_t^2v\right\|_2^2+\sum\limits_{j=1}^3\left\|\FP\p_t^2\eta\right\|_2^2+\left\|\swt\p_t^2 h\right\|_2^2\\
&+\left\|\swt\p_t^3 v\right\|_1^2+\sum\limits_{j=1}^3\left\|\swt\FP\p_t^3 \eta\right\|_1^2+\left\|\wt\p_t^3 h\right\|_1^2\\
&+\left\|\wt\p_t^4 v\right\|_0^2+\sum\limits_{j=1}^3\left\|\wt\FP\p_t^4 \eta\right\|_0^2+\left\|(\wt)^{\frac32}\p_t^4 h\right\|_0^2
\end{aligned}
\end{equation}

Our conclusion is 

\begin{prop}\label{apriorikk}
There exists some $T>0$ independent of $\kk$, such that the energy functional $\EE_{\kk}$ satisfies
\begin{equation}\label{EEkk1}
\sup_{0\leq t\leq T} \EE_\kk(t)\leq P(\|v_0\|_{4},\|\F\|_4,\|\h_0\|_4),
\end{equation} provided the following assumptions hold for all $t\in[0,T]$
\begin{align}
\label{taylor1} (-\p h\cdot N)(t)\geq c_0/2~~~&\text{on }\Gamma,\\
\label{small1} \|\Jk(t)-1\|_3+\|\text{Id}-\ak(t)\|_3\leq \delta~~~&\text{in }\Omega.
\end{align}
\end{prop}

\begin{rmk}
The first a priori assumption \eqref{taylor1} will be justified by verifying $(-\p h\cdot N)$ is a $C_{t,x}^{0,1/4}$ function on the boundary, and thus the positivity of Taylor sign condition must propagate for a positive time. The second a priori assumption \eqref{small1} can be easily justified once the energy bounds are established by using $\ak(T)-\text{Id}=\int_0^T \p_t\ak=\int_0^T \ak:\p_t\p\ek:\ak\dt$ and the smallness of $T$.
\end{rmk}

In Section \ref{linearize}, we will prove the local well-posedness of \eqref{elastolkk} in an $\kk$-dependent time interval $[0,T_{\kk}]$.  Therefore, the uniform-in-$\kk$ a priori estimate guarantees that the solution $(\eta(\kk),v(\kk),h(\kk))$ to \eqref{elastolkk} converges to the solution to the original system \eqref{elastol} in a $\kk$-independent time interval $[0,T]$ as $\kk\to 0_+$, i.e., local existence of the solution to free-boundary compressible elastodynamic system is established.  For simplicity, we omit the $\kk$ and only write $(\eta,v,h)$ in this manuscript.

Before going to the proof, we present the estimates of the correction term $\psi$.

\begin{lem}\label{etapsi}
We have the following estimates for $(v,\psi,\eta)$ in \eqref{elastolkk}. 
\begin{align}
\label{eta4} \|\ek\|_4&\lesssim \|\eta\|_4, \\
\label{psi4} \|\psi\|_4+\|\FP\psi\|_4&\lesssim P(\|\F\|_4,\|\FP\eta\|_4,\|\eta\|_4,\|v\|_4), \\
\label{psit4}\|\p_t \psi\|_4+\|\p_t\FP\psi\|_3&\lesssim P(\|\F\|_4,\|\FP\eta\|_4,\|\eta\|_4,\|v\|_4,\|\p_t v\|_3),\\
\label{psitt3}\|\p_t^2 \psi\|_3+\|\p_t^2\FP\psi\|_2&\lesssim P(\|\F\|_4,\|\FP\eta\|_4,\|\eta\|_4,\|v\|_4, \|\p_tv\|_3, \|\p_t^2 v\|_2), 
\end{align}
\begin{equation}\label{psittt2}
\begin{aligned}
\|\swt \p_t^3\psi\|_2+\|\swt\p_t^3\FP\psi\|_1\lesssim P(\|\F\|_4,\|\FP\eta\|_4,\|\eta\|_4,\|v\|_4, \|\p_tv\|_3, \|\p_t^2 v\|_2, \|\swt\p_t^3v\|_{1}),
\end{aligned}
\end{equation}
\begin{equation}\label{psitttt1}
\begin{aligned}
\|\wt\p_t^4\psi\|_1+\|\wt\p_t^4\FP\psi\|_0\lesssim P(\|\F\|_4,\|\FP\eta\|_4,\|\eta\|_4,\|v\|_4, \|\p_tv\|_3, \|\p_t^2 v\|_2, \|\swt\p_t^3v\|_{1},\|\wt\p_t^4 v\|_0).
\end{aligned}
\end{equation}
and
\begin{align}
\label{fpeta4} \|\FP\ek\|_4&\lesssim P(\|\F\|_4,\|\FP\eta\|_4,\|\eta\|_4) \\
\label{etat4} \|\p_t\ek\|_4+\|\p_t\FP\ek\|_3&\lesssim P(\|\F\|_4,\|\FP\eta\|_4,\|\eta\|_4,\|v\|_4) , \\
\label{etatt3}\|\p_t^2\ek\|_3+\|\p_t^2\FP\ek\|_2&\lesssim P(\|\F\|_4,\|\FP\eta\|_4,\|\eta\|_4,\|v\|_4,\|\p_t v\|_3) ,
\end{align}
\begin{equation}
\label{etattt2}\|\p_t^3\ek\|_2+\|\p_t^3\FP\ek\|_1\lesssim P(\|\F\|_4,\|\FP\eta\|_4,\|\eta\|_4,\|v\|_4,\|\p_t v\|_3,\|\p_t^2 v\|_2),
\end{equation}
\begin{align}
\label{etatttt1} \|\swt\p_t^4\ek\|_1+\|\swt\p_t^4\FP\ek\|_0\lesssim&  P(\|\F\|_4,\|\FP\eta\|_4,\|\eta\|_4,\|v\|_4, \|\p_tv\|_3, \|\p_t^2 v\|_2, \|\swt\p_t^3v\|_{1})\\
\label{etattttt0} \|\wt\p_t^5\ek\|_0\lesssim\|\wt\p_t^5\eta\|_0\lesssim & P(\|\eta\|_4,\|v\|_4, \|\p_tv\|_3, \|\p_t^2 v\|_2, \|\swt\p_t^3v\|_{1},\|\wt\p_t^4 v\|_0),
\end{align}
\end{lem}
\begin{proof}
 The estimates of $\ek,\psi$ and their derivatives are exactly the same as \cite[Lemma 3.2]{ZhangCRMHD2}, which can be proved by using the properties of mollifier and elliptic estimates, so we omit the details here. We just show the estimates of $\|\FP\psi\|_4$ and $\|\FP\ek\|_4$.

We recall the definition of $\ek$ in \eqref{lkklkk} and take $\FP$ in that equation to get
\begin{equation}
\begin{cases}
-\Delta(\FP\ek)=-\Delta(\FP\eta)-[\FP,\Delta]\eta+[\FP,\Delta]\ek~~~&\text{ in }\Omega,\\
\FP\ek=\FP\lkk^2\eta~~~&\text{ on }\Gamma.
\end{cases}
\end{equation}

By standard elliptic estimates we have
\begin{equation}
\begin{aligned}
\|\FP\ek\|_4\lesssim&\left\|-\Delta(\FP\eta)-[\FP,\Delta]\eta+[\FP,\Delta]\ek\right\|_2+\left|\FP\lkk^2\eta\right|_{7/2}\\
\lesssim&\|\FP\eta\|_4+\|\F\|_4\|\eta\|_4+\left|\lkk^2(\FP\eta)\right|_{7/2}+\left|\left[\FP,\lkk^2\right]\right|_{7/2}\\
\lesssim&\|\FP\eta\|_4+\|\F\|_4\|\eta\|_4\\
&+\sum_{L=1}^2\left|\left[\lkk^2,\F_{Lj}\right]\TP_L\eta\right|_{1/2}+\left|\left[\lkk^2,\F_{Lj}\right]\TP_L\TP^3\eta\right|_{1/2}+\left|\left[\TP^3,[\lkk^2,\F_{Lj}]\TP_L\right]\eta\right|_{1/2}\\
\lesssim&\|\FP\eta\|_4+\|\F\|_4\|\eta\|_4.
\end{aligned}
\end{equation}

The time derivative of $\FP\ek$ directly follows from $\|\FP\ek\|_4$ and the estimates of $\|\p_t^r\ek\|_{5-r}$. For example, $\|\p_t^r\FP\ek\|_{4-r}\lesssim\|\F\|_3\|\p_t^r\ek\|_{5-r}$ for $1\leq r\leq 4$.

The estimates of $\FP\psi$ follows in the same way as $\FP\ek$, i.e., taking $\FP$ in the elliptic system \eqref{psi} and using standard elliptic estimates, so we omit the proof. Its time derivative can also be controlled by using the bounds of $\p_t^r\psi$. 
\end{proof}

Now we start to analyze the energy $\EE_\kk(T)$. The first step is to control the full spatial derivative of $v$ and $\FP\eta$.

\subsection{Div-Curl estimates of full spatial derivatives}\label{divcurl}

Let $s=4$ and $X=v$ and $\FP\eta$ in Lemma \ref{hodge} respectively, we have
\begin{align}
\label{v4} \|v\|_4&\lesssim\|v\|_0+\|\dive v\|_{3}+\|\curl v\|_{3}+|\TP v\cdot N|_{5/2},\\
\label{F4} \|\FP\eta\|_4&\lesssim\|\FP\eta\|_0+\|\dive (\FP\eta)\|_{3}+\|\curl (\FP\eta)\|_{3}+|\TP(\FP\eta)\cdot N|_{5/2}.
\end{align}

\paragraph*{$L^2$-estimates: Energy conservation}

The proof is nearly the same as in \eqref{econserve}. Taking $L^2$ inner product of the second equation in \eqref{elastolkk} and $v$ yields
\[
\frac{1}{2}\frac{d}{dt}\io |v|^2=-\io v\cdot(\pak h)+\sum_{j=1}^3\FP^2\eta\cdot v\dy.
\] Integrating by parts gives
\[
-\io v\cdot(\pak h)=\io(\diva v) h+\io\p_l\ak^{li} v_i h\dy\lesssim-\frac12\io \wt|h|^2\dy+\frac12\io e''(h)\p_t h|h|^2+\io\p_l\ak^{li} v_i h\dy,
\]and 
\[
\io\sum_{j=1}^3\FP^2\eta\cdot v\dy=\sum_{j=1}^3-\frac12\frac{d}{dt}\io|\FP\eta|^2\dy+\io(\FP\eta)\cdot(\FP\psi)\dy-\io(\dive \F_j)\FP\eta\cdot v\dy.
\]
Therefore
\begin{equation}\label{vFhl2}
\frac{d}{dt}\left(\|v\|_0^2+\sum_{j=1}^3\left\|\FP\eta\right\|_0^2+\|\swt h\|_0^2\right)\lesssim P(E_\kk(T)).
\end{equation}

As for $h$, one has $\|h(T)\|_0\lesssim \|\h_0\|_0+\int_0^T\|\p_t h(t)\|_0\dt$ to control its $L^2$ norm.

\paragraph*{Boundary control: Reduce to divergence and interior tangential estimates}

The boundary part can be reduced to interior tangential estimates and divergence by Lemma \ref{normaltrace}.
\begin{equation}\label{vbd}
\left|\TP^4v\cdot N\right|_{-1/2}\lesssim\left\|\TP^4 v\right\|_0+\left\|\TP^3\dive v\right\|_{0}.
\end{equation} Similarly,
\begin{equation}\label{Fbd}
\left|\TP^4(\FP\eta)\cdot N\right|_{-1/2}\lesssim \left\|\TP^4 \left(\FP\eta\right)\right\|_0+\left\|\TP^3\dive\left(\FP\eta\right)\right\|_{0}.
\end{equation}

\paragraph*{Curl control: Direct computation}

First, it suffices to control the Eulerian curl due to the a priori assumption \eqref{small1}
\[
\|\curl X\|_3\leq\|\curla X\|_3+\|\text{Id}-\ak\|_{2}\| X\|_4\lesssim\|\curla X\|_3+\delta\| X\|_4,
\] where the $\delta$-term can be absorbed by $\|X\|_4$ in the energy functional.

Recall in \eqref{elastolkk} that we have $\p_t\eta=v+\psi$ and $\p_t v-\sum\limits_{j=1}^3\FP^2\eta=-\pak h$. Taking (Eulerian curl) $\curla$, we get the evolution equation of $\curla v$.
\begin{equation}\label{curl0}
\p_t\left(\curla v\right)-\sum\limits_{j=1}^3\FP\left(\curla(\FP\eta)\right)=\underbrace{\curl_{\ak_t} v+\sum\limits_{j=1}^3\left[\curla,\FP\right]\FP\eta}_{J_0}.
\end{equation}

Then we take 3 spatial derivatives to get
\begin{equation}\label{curl3}
\p_t\left(\p^3\curla v\right)-\sum\limits_{j=1}^3\FP\p^3\left(\curla (\FP\eta)\right)=\p^3 J_0+\underbrace{\sum\limits_{j=1}^3\left[\p^3,\FP\right]\left(\curla \FP\eta\right)}_{J_1}.
\end{equation}

Now we compute the $L^2$-inner product of \eqref{curl3} with $\p^3\curla v$ to get
\begin{equation}\label{curl31}
\io\left(\p_t\p^3\curla v\right)\left(\p^3\curla v\right)\dy=\frac12\frac{d}{dt}\io \left|\p^3\curla v\right|^2\dy.
\end{equation}

Then we integrate $\FP$ by parts in the second term of LHS in \eqref{curl3}. Note that the boundary term vanishes because $\F_j\cdot N=0$ for $j=1,2,3$.

\begin{equation}\label{curl32}
\begin{aligned}
&-\sum\limits_{j=1}^3\io\FP\p^3\left(\curla (\FP\eta)\right)\cdot\left(\p^3\curla v\right)\dy \\
=&\sum\limits_{j=1}^3\io\p^3\left(\curla (\FP\eta)\right)\cdot\left(\p^3\curla\FP v\right)\dy\\
&+\underbrace{\sum\limits_{j=1}^3\io(\dive \F_j)\p^3\left(\curla (\FP\eta)\right)\cdot\left(\p^3\curla v\right)\dy}_{J_2}\\
&+\underbrace{\sum\limits_{j=1}^3\io\p^3\left(\curla (\FP\eta)\right)\cdot\left[\FP,\p^3\curla\right] v\dy}_{J_3}\\
=&\sum\limits_{j=1}^3\frac12\frac{d}{dt}\io\left|\p^3\curla \FP\eta\right|^2\dy-\sum_{j=1}^3\io\p^3\curla (\FP\eta)\cdot\p^3\curla\FP\psi\dy\\
&+\sum\limits_{j=1}^3 \io\p^3\curla (\FP\eta)\cdot\p^3\left(\left[\curla,\p_t\right] \FP\eta\right)\dy+J_2+J_3
\end{aligned}
\end{equation}

It remains to control the terms $J_0\sim J_5$ where 
\begin{align}
J_4:=&-\sum_{j=1}^3\io\p^3\curla (\FP\eta)\cdot\p^3\curla\FP\psi\dy,\\
J_5:=&\sum\limits_{j=1}^3 \io\p^3\curla (\FP\eta)\cdot\p^3\left(\left[\curla,\p_t\right] \FP\eta\right)\dy
\end{align}

These commutators can be controlled either by direct computation or using Lemma \ref{etapsi}. We have
\begin{align}\label{J01}
\io (\p^3J_0+J_1)\cdot\left(\p^3\curla v\right)\dy \lesssim \sum_{j=1}^3 P(\|\F\|_4,\|\FP\eta\|_4.\|\p v\|_4,\|\FP\ak\|_3,\|\ak\|_3)\lesssim P(\EE_\kk(T)).
\end{align}
For $J_2$, we have
\begin{align}\label{J2}
J_2\lesssim \|\p\F\|_{L^{\infty}}\|\FP\eta\|_4\|v\|_4.
\end{align}
For $J_3$, direct computation shows that
\begin{align}\label{J3}
J_3\lesssim P\left(\|\FP\eta\|_4,\|\F\|_4,\|v\|_4,\|\FP\ak\|_3,\|\ak\|_3\right)\lesssim P(\EE_\kk(T)).
\end{align}

By using Lemma \ref{etapsi}, we have that
\begin{align}\label{J45}
J_4+J_5\lesssim P\left(\|\F\|_4,\|\FP\eta\|_4,\|\ak\|_3,\|\p_t\ak\|_3,\|\TP\psi\|_4\right)\lesssim P(\EE_\kk(T)).
\end{align}

Summing up \eqref{curl31}, \eqref{curl32}, \eqref{J01}-\eqref{J45}, we get the Eulerian curl estimates by
\begin{equation}\label{curlavF}
\frac12\frac{d}{dt}\left(\io \left|\p^3\curla v\right|^2\dy+\sum\limits_{j=1}^3\io\left|\p^3\curla \FP\eta\right|^2\dy\right)\lesssim P(\EE_\kk(T)),
\end{equation}and thus for sufficiently small $\delta>0$ we have
\begin{equation}
\left\|\curl v(T)\right\|_3^2+\left\|\curl \FP\eta(T)\right\|_3^2\lesssim\delta^2\left(\|v\|_4^2+\|\FP\eta\|_4^2\right)+\PP_0+\int_0^T P(\EE_\kk(t))\dt.
\end{equation}

\paragraph*{Divergence estimates: Reduce to the elliptic estimates of $h$}

By the third equation of \eqref{elastolkk} and the a priori assumption \eqref{small1}, we know
\begin{equation}\label{divv0}
\left\|\dive v\right\|_3^2\lesssim\delta^2\|v\|_4^2+\left\|\wt \p_t h\right\|_3^2.
\end{equation}

As for the divergence of the deformation tensor $\FF$, we recall that the original system \eqref{elastoL} reads $\dive_a\FF=-\wt\FP h$. Now we should re-produce a silimar formula of $\diva\FP\eta$ for the smoothed-out problem \eqref{elastolkk} up to some error terms. Direct computation together with $\p_t\eta=v+\psi$ yields that
\begin{equation}
\diva\FP\eta(T)-\dive\F_j=\int_0^T\diva\FP v\dt+\int_0^T\diva\FP\psi+\dive_{\ak_t}\FP\eta\dt
\end{equation}
Recall that $\diva v=-\p_t e(h)$, so we commute $\diva$ with $\FP$ to get
\begin{equation}
\begin{aligned}
\int_0^T\diva\FP v\dt=&-\int_0^T \p_t\left(\FP e(h)\right)\dt+\int_0^T\left[\diva,\FP\right] v\dt\\
=&-e'(h)\FP h\bigg|_0^T+\int_0^T\left[\diva,\FP\right] v\dt.
\end{aligned}
\end{equation}
Summing up these two equalities, we get
\begin{equation}\label{divkkF0}
\diva\FP\eta(T)=-e'(h)\FP h(T)+\int_0^T\diva\FP\psi+\dive_{\ak_t}\FP\eta+\left[\diva,\FP\right] v\dt.
\end{equation}

Direct computation shows that the commutator is only one term with first order derivative. One can also equivalently compute in Eulerian coordinate.
\begin{equation}\label{commdivF}
\begin{aligned}
\left[\diva,\FP\right]f=&\ak^{li}\p_l\left(\FF_0^{kj}\p_kf_i\right)-\FF_0^{kj}\p_k\left(\ak^{li}\p_l f_i\right)\\
=&\ak^{li}\p_l\FF_0^{kj}\p_k f_i-\FF_0^{kj}\p_k\ak^{li}\p_l f_i=\ak^{li}\p_l\FF_0^{kj}\p_k f_i+\FF_0^{kj}\ak^{lr}\p_k\p_m\ek_r\ak^{mi}\p_l f_i\\
=&\ak^{mi}\p_m\left(\FP\ek_r\right)\ak^{lr}\p_l f_i-\ak^{li}\p_l\FF_0^{kj}\p_k f_i+\ak^{li}\p_l\FF_0^{kj}\p_k f_i\\
=&\pak^i\left(\FP\ek_r\right)\pak^r f_i.
\end{aligned}
\end{equation}

On the other hand, we compute that
\begin{align*}
\dive_{\ak_t}\FP\eta=-\ak^{lr}\p_m\p_t\ek_r\ak^{mi}\p_l(\FP\eta_i)=-\pak^r\left(\FP\eta_i\right)\pak^i\p_t\ek_r.
\end{align*}

Therefore we get
\begin{equation}\label{divkkF}
\begin{aligned}
\diva\FP\eta(T)=&-e'(h)\FP h(T)+\int_0^T\diva\FP\psi-\pak^r\left(\FP\eta_i\right)\pak^i\p_t\ek_r+\pak^i\left(\FP\ek_r\right)\pak^r v_i\dt\\
=&-e'(h)\FP h(T)+\int_0^T\diva\FP\psi-\pak^r\left(\FP\eta_i\right)(\pak^i\psi_r)\dt
\\&+\int_0^T\pak^i\left(\FP\ek_r\right)\pak^r\left(\p_t\eta_i-\p_t\ek_i\right)+\pak^i\left(\FP\ek_r-\FP\eta_r\right)\pak^rv_i\dt
\end{aligned}
\end{equation}

\begin{rmk}
We find that, if $\kk=0$ (and thus $\psi=0,~\ek=\eta$), then the formula \eqref{divkkF} for the smoothed-Eulerian divergence of $\FF$ exactly reproduces the divergence constraint for the original equation.
\end{rmk}

By Jensen's inequality, we know
\begin{equation}\label{divF0}
\begin{aligned}
&\left\|\dive\FP\eta(T)\right\|_3^2\lesssim \delta^2\left\|\FP\eta\right\|_4^2+\left\|\diva\FP\eta(T)\right\|_3^2\\
\lesssim& \delta^2\left\|\FP\eta\right\|_4^2+\left\|e'(h)\FP h(T)\right\|_3^2+\int_0^T\left\|\diva\FP\psi\right\|_3^2+\left\|\dive_{\ak_t}\FP\eta\right\|_3^2+\left\|\left[\diva,\FP\right] v\right\|_3^2\dt\\
\lesssim&\delta^2\left\|\FP\eta\right\|_4^2+\left\|e'(h)\FP h(T)\right\|_3^2+\int_0^T P\left(\|\FP\psi\|_4,\|\ak\|_3,\|\p_t\ak\|_3,\|\FP\eta\|_4,\|\p v\|_3\right)\dt\\
\lesssim&\delta^2\left\|\FP\eta\right\|_4^2+\left\|e'(h)\FP h(T)\right\|_3^2+\int_0^T P(\EE_\kk(t))\dt.
\end{aligned}
\end{equation}

Therefore, the divergence control of $v$ and $\FP\eta$ are reduced to weight norm $$\|\wt\p_t h\|_3\text{ and }\|\wt\FP h\|_3$$ which will be further reduced to tangentially-differentiated wave equation of $h$ by using Christodoulou-Lindblad elliptic estimates Lemma \ref{GLL}. We postpone the proof to Section \ref{divreduce} and \ref{wavediv}.

\subsection{Tangential estimates of full spatial derivatives: Alinhac good unknown}\label{tgspace}

The boundary part in the Hodge-type decomposition is reduced to the interior tangential estimates and divergence control. We already reduce divergence control to the elliptic estimates of $h$. This part we deal with the interior tangential derivatives $\|\TP^4v\|_0$ and $\|\TP^4(\FP\eta)\|_0$.  Due to the special structure of the correction term $\psi$, we replace $\TP^4$ by $\tpl$. Note that we cannot directly commute $\tpl$ in the equation $\p_t v+\sum\limits_{j=1}^3\FP^2\eta=-\pak h$ because one of the leading order terms in the commutator $[\tpl,\pak]h$ is $\tpl\ak\cdot\p h\approx \tpl\p\ek\times\p\ek\cdot\p h$ which cannot be controlled in $L^2$. The reason is that the essential highest order term in the standard derivatives of a covariant derivative is actually the covariant derivative of Alinhac good unknown, instead of simply commuting $\tpl$ with $\pak$.

For a function $f$, define its Alinhac good unknown to be $\mathbf{f}:=\tpl f-\tpl\ek\cdot\pak f$. Then the following property holds
\begin{align*}
\tpl(\pak^{i}f)&=\pak^{\alpha}(\tpl f)+(\tpl\ak^{li})\p_l f+[\tpl,\ak^{li},\p_l f] \\
&=\pak^{i}(\tpl f)-\TP\TL(\ak^{lr}\TP\p_{m}\ek_{r}\ak^{mi})\p_l f+[\tpl,\ak^{li},\p_{l} f] \\
&=\underbrace{\pak^{i}(\tpl f-\tpl \eta_{r}\ak^{lr}\p_l f)}_{=\pak^{i}\mathbf{f}}+\underbrace{\tpl \eta_{r}\pak^{i}(\pak^{r} f)-([\TP\TL,\ak^{lr}\ak^{mi}]\TP\p_{m}\ek_{r})\p_l f+[\tpl,\ak^{li},\p_{l} f] }_{=:C^{i}(f)},
\end{align*} where $[\tpl,g,h]:=\tpl(gh)-\tpl(g)h-g\tpl(h)$.
Direct computation yields that
\begin{align*}
\|\tpl \eta_{r}\pak^{i}(\pak^{r} f)\|_0&\lesssim\|\ek\|_4\|\pak^{i}(\pak^{r} f)\|_{L^{\infty}}\;\\
\|([\TP\TL,\ak^{lr}\ak^{mi}]\TP\p_{m}\ek_{r})\p_l f\|_0&\lesssim\|[\TP\TL,\ak^{lr}\ak^{mi}]\TP\p_{m}\ek_{r}\|_0\|f\|_{W^{1,\infty}}\lesssim P(\|\ek\|_4)\|f\|_3 \\
\|[\tpl,\ak^{li},\p_{l} f]\|_0&\lesssim P(\|\ek\|_4)\|f\|_4.
\end{align*}
Therefore, Alinhac good unknown enjoys the following important properties:
\begin{equation}\label{goodid}
\tpl(\pak^{i}f)=\pak^{i}\mathbf{f}+C^{\alpha}(f)
\end{equation}
with 
\begin{equation}\label{goodco}
\|C^{i}(f)\|\lesssim P(\|\ek\|_4)\|f\|_4.
\end{equation}

We introduce the Alinhac good unknowns of $v$ and $h$ with respect to $\tpl$ by
\[
\VV:=\tpl v-\tpl\ek\cdot\pak v,~~\HH:=\tpl h-\tpl\ek\cdot\pak h.
\]

Taking $\tpl$ in the second equation of \eqref{elastolkk}, we get the evolution of Alinhac good unknowns 
\begin{equation}\label{goodeq}
\p_t\VV+\pak\HH-\sum\limits_{j=1}^{3} \FP\left(\tpl\FP\eta\right)=\underbrace{\p_t(\tpl\ek\cdot\pak v)-C(h)+\sum\limits_{j=1}^{3}\left[\tpl,\FP\right](\FP\eta)}_{K_0},
\end{equation}subjected to
\begin{align}
\label{vvdiv} \pak\cdot\VV=\tpl(\diva v)-C^i(v_i)~~~&\text{in }\Omega \\
\label{hhbdry} \HH=(-\p_3h)\ak^{3k}\tpl\ek_k~~~&\text{on }\Gamma.
\end{align}

Now we take $L^2$-inner product of \eqref{goodeq} and $\VV$ to get
\begin{equation}\label{good0}
\frac12\frac{d}{dt}\io\left|\VV\right|^2\dy=-\io\pak\HH\cdot\VV\dy+\sum_{j=1}^3\io\left(\FP\left(\tpl\FP\eta\right)\right)\VV\dy+\io K_0\cdot\VV\dy
\end{equation}

For the first term on the RHS of \eqref{good0}, we integrate by parts and invoke \eqref{vvdiv}-\eqref{hhbdry} to get
\begin{equation}\label{good1}
\begin{aligned}
&-\io \pak \HH\cdot \VV=-\io \ak^{li}\p_l \HH\cdot \VV_{i}\dy \\
=&-\ig \HH(\ak^{li}N_{l}\VV_{i})dS+\io \HH(\pak\cdot\VV)\dy+\io(\p_\mu \ak^{li})\HH\VV_{i}\dy\\
=&\ig\p_3h\tpl\ek_{k}\ak^{3k}\ak^{li}N_{l}\VV_{i} dS+\io \HH\tpl (\diva v)\dy-\io \HH C^{i}(v_{i})+\io(\p_l \ak^{li})\HH\VV_{\alpha}\dy\\
=&-\ig\left(-\frac{\p h}{\p N}\right)\tpl\ek_{k}\ak^{3k}\ak^{3i}\VV_{i} dS+\io \HH\tpl (\diva v)\dy-\io \HH C^{i}(v_{i})+\io(\p_l \ak^{li})\HH\VV_{i}\dy\\
=&:B+K_1+K_2+K_3.
\end{aligned}
\end{equation}

We postpone the estimates of $B$ to the end of this section because it is the most difficult part and contributes to the boundary energy in $\EE_\kk(T)$. The term $K_2,K_3$ can be controlled directly
\begin{align}
\label{K2} K_2=-\io \HH C^{i}(v_{i})\lesssim&\|\HH\|_0\|C(v)\|_0\lesssim P\left(\|h\|_4,\|\eta\|_4,\|v\|_4,\|\ak\|_3\right).\\
\label{K3} K_3=\io(\p_l \ak^{li})\HH\VV_{i}\dy\lesssim& \|\p\ak\|_{L^{\infty}}\|\HH\|_0\|\VV\|_0\lesssim P\left(\|h\|_4,\|\eta\|_4,\|v\|_4,\|\ak\|_3\right).
\end{align}

In $K_1$, we invoke $\diva v=-e'(h)\p_t h$ to get
\begin{equation}\label{K10}
\begin{aligned}
K_1=&\io \HH\tpl (\diva v)\dy=\io\left(\tpl h-\tpl\ek\cdot\pak h\right)\tpl(e'(h)\p_t h)\dy\\
=&-\frac12\frac{d}{dt}\io e'(h)\left|\tpl h\right|^2\dy+\frac12\io e''(h)\p_t h\left|\tpl h\right|^2\dy\\
&+\io\HH\left(\left[\tpl,e'(h)\right]\p_t h\right)\dy\underbrace{-\io e'(h)\tpl\p_t h\tpl\ek\cdot\pak h\dy}_{K_{11}}\\
\lesssim&-\frac12\frac{d}{dt}\io e'(h)\left|\tpl h\right|^2\dy+K_{11}+P\left(\|\eta\|_4,\|v\|_4,\|h\|_4,\|\wt \p_th\|_3\right)
\end{aligned}
\end{equation}

 The term $K_{11}$ should be controlled by integrating $\p_t$ by parts under the time integral.
\begin{equation}\label{K11}
\begin{aligned}
\int_0^T K_{11}(t)\dt=-&\int_0^T\io e'(h)\tpl\p_t h(\tpl\ek\cdot\pak h)\dy\dt\\
\overset{\p_t}{=}&\io e'(h)\tpl h(\tpl\ek\cdot\pak h)\dy\bigg|^{t=T}_{t=0}+\int_0^T\io e'(h)\tpl h\p_t\left(\tpl\ek\cdot\pak h\right)\dy \\
\lesssim& \left\|e'(h)\tpl h\right\|_0\|\TP^4\eta\|_0\|\pak h\|_{L^{\infty}}\bigg|_{t=T}+P\left(\|\h_0\|_4,\|v_0\|_4\right)+\int_0^T P\left(\|h\|_4,\|\eta\|_4,\|v\|_4\right)\dt.
\end{aligned}
\end{equation}
Invoking that $\p^2\eta(0)=0$ (recall $\eta(0,y)=y$ is identity map initially), we have
\begin{align*}
\|e'(h)\tpl h\|_0\|\TP^4\eta\|_0\|\pak h\|_{L^{\infty}}\bigg|_{t=T}\lesssim P(\EE_\kk(T))\int_0^T\left\|\TP^4 \p_t\eta(t)\right\|_0^2\dt \\
\end{align*}
\begin{rmk}
One can also use Young's inequality to control $\int_0^T K_{11}$ by $$\delta \|\swt\tpl h(T)\|_0^2+\left(\PP_0+\int_0^TP(\|v\|_4,\|\eta\|_{4},\|h\|_3,\|\swt\p_t h\|_3)\dt\right).$$
\end{rmk}
Therefore we get the estimate of $K_{1}$:
\begin{equation}\label{K1}
\begin{aligned}
\int_0^T K_{1}(t)\dt\lesssim -\left\|\swt\tpl h(T)\right\|_0^2 +\PP_0+P(\EE_{\kk}(T))\int_0^TP(\EE_\kk(t))\dt.
\end{aligned}
\end{equation}

Next we estimate the tangential derivative of $\FP\eta$ produced in the second term on RHS of \eqref{good0}. Integrating $\FP$ by parts yields that
\[
\sum_{j=1}^3\io\left(\FP\left(\tpl\FP\eta\right)\right)\VV\dy=-\sum_{j=1}^{3}\io\left(\tpl\FP\eta\right)\FP\VV\dy\underbrace{-\io (\dive\F_j)\left(\tpl\FP\eta\right)\VV}_{K_4}
\]

$K_4$ has direct control
\begin{equation}\label{K4}
K_4\lesssim\sum_{j=1}^3\|\p\F\|_{L^{\infty}}\left\|\FP\eta\right\|_0\left\|\VV\right\|_0\lesssim\sum_{j=1}^3 P\left(\|\FP\eta\|_4,\|v\|_4,\|\ak\|_3,\|\p\F\|_{L^{\infty}}\right).
\end{equation}

In the remaining term, we invoke $\VV=\tpl v-\tpl\ek\cdot\pak v$ and $v=\p_t\eta-\psi$ to get 
\begin{equation}
\begin{aligned}
&-\sum_{j=1}^{3}\io\left(\tpl\FP\eta\right)\FP\VV\dy\\
=&-\sum_{j=1}^{3}\io\left(\tpl\FP\eta^i\right)\FP\left(\tpl v_i-\tpl\ek\cdot\pak v_i\right)\dy \\
=&-\sum_{j=1}^{3}\io\left(\tpl\FP\eta^i\right)\left(\tpl\FP\p_t\eta_i\right)\dy+\sum_{j=1}^{3}\io\left(\tpl\FP\eta\right)\left(\tpl\FP\psi\right)\dy\\
&+\sum_{j=1}^{3}\io\left(\tpl\FP\eta^i\right)\left(\left[\tpl,\FP\right]v_i+\FP\left(\tpl\ek\cdot\pak v_i\right)\right)\dy\\
=:&-\sum_{j=1}^3\frac12\frac{d}{dt}\io\left|\tpl\FP\eta\right|^2\dy+K_5,
\end{aligned}
\end{equation}where the commutators $K_5$ can be directly controlled with the help of Lemma \ref{etapsi}
\begin{equation}\label{K5}
K_5\lesssim \sum_{j=1}^3\left\|\FP\eta\right\|_4 P\left(\|\FP\psi\|_4,\|\F\|_4,\|v\|_4,\|\ek\|_4,\|\FP\ek\|_4\right)\lesssim P\left(\EE_\kk(T)\right).
\end{equation}

Summing up \eqref{K4}-\eqref{K5}, we get
\begin{equation}\label{tgF}
\sum_{j=1}^3\io\left(\FP\left(\tpl\FP\eta\right)\right)\VV\dy\lesssim-\sum_{j=1}^3\frac12\frac{d}{dt}\io\left|\tpl\FP\eta\right|^2\dy+\int_0^T P(\EE_{\kk}(t))\dt.
\end{equation}

Also, the last term in \eqref{good0} can also be directly controlled
\begin{equation}\label{tgK0}
\begin{aligned}
\io K_0\cdot\VV\lesssim& \left(\|\p_t(\tpl\ek\cdot\pak v)\|_0+\|C(h)\|_0+\|[\tpl,\FP]\FP\eta\|_0\right)\left\|\VV\right\|_0 \\
\lesssim& P\left(\|\eta\|_4,\|v\|_4,\|\p_t v\|_3,\|h\|_4,\|\F\|_4,\sum_{j=1}^3\|\FP\eta\|_4\right)\lesssim P(\EE_\kk(T)).
\end{aligned}
\end{equation}

It remains to deal with the boundary term $B$ produced in \eqref{good1}. In the proof we will see that the boundary integral together with the Raylrigh-Taylor sign condition contributes to the boundary energy, which enters to the highest order and exactly gives the control of the second fundamental form of the free surface. The correction term, first introduced by Gu-Wang \cite{gu2016construction}, plays the key role in eliminating the higher order terms brought by the tangential smoothing. First, we have

\begin{equation}\label{B0}
\begin{aligned}
B=&-\ig\left(-\frac{\p h}{\p N}\right)\tpl\ek_{k}\ak^{3k}\ak^{3i}\VV_{i}~dS\\
=&\ig\left(\frac{\p h}{\p N}\right)\tpl\ek_{k}\ak^{3k}\ak^{3i}\tpl\p_t\eta_i~dS\\
&-\ig\left(\frac{\p h}{\p N}\right)\tpl\ek_{k}\ak^{3k}\ak^{3i}\tpl\psi_i\dS-\ig\left(\frac{\p h}{\p N}\right)\tpl\ek_{k}\ak^{3k}\ak^{3i}\tpl\ek\cdot\pak v_i\dS\\
=:&B_1+B_2+B_3.
\end{aligned}
\end{equation}

The term $B_1$ is expected to give the boundary part of the energy functional $\EE_\kk(T)$ after moving one mollifier from $\ek_k$ to $\eta_i$.

\begin{equation}\label{B10}
\begin{aligned}
&\ig\left(\frac{\p h}{\p N}\right)\tpl\ek_{k}\ak^{3k}\ak^{3i}\tpl\p_t\eta_i~dS\\
=&\ig\left(\frac{\p h}{\p N}\right)\tpl\lkk\eta_{k}\ak^{3k}\ak^{3i}\tpl\p_t\lkk\eta_i~dS\underbrace{+\ig\tpl\lkk\eta_k\left(\left[\lkk,\frac{\p h}{\p N}\ak^{3k}\ak^{3i}\right]\tpl\p_t\eta_i\right)\dS}_{B_{11}}\\
=&\frac12\frac{d}{dt}\ig\frac{\p h}{\p N}\left|\ak^{3i}\tpl\lkk\eta_i\right|^2\dS+B_{11}\\
&\underbrace{-\frac12\ig\p_t\left(\frac{\p h}{\p N}\right)\left|\ak^{3i}\tpl\lkk\eta_i\right|^2\dS}_{B_{12}}\underbrace{-\ig\left(\frac{\p h}{\p N}\right)\ak^{3k}\tpl\lkk\eta_k\p_t\ak^{3i}\tpl\lkk\eta_i\dS}_{B_{13}}.
\end{aligned}
\end{equation}

The time integral of the main term gives the boundary part of $\EE_\kk(T)$ with the help of Taylor sign condition \eqref{sign}, \eqref{taylor1}.
\begin{equation}\label{B100}
\int_0^T\left(\frac12\frac{d}{dt}\ig\frac{\p h}{\p N}\left|\ak^{3i}\tpl\lkk\eta_i\right|^2\dS\right)\dt\leq -\frac{c_0}{4}\left|\ak^{3i}\TP^4\lkk\p_t\eta_i\right|_0^2\bigg|^T_0.
\end{equation}

$B_{11}$ and $B_{12}$ can be controlled with the help of mollifier property (Lemma \ref{tgsmooth}).
\begin{equation}\label{B11}
\begin{aligned}
B_{11}=&\ig\TP^{1.5}\TL\lkk\eta_{k}\TP^{0.5}\left(\left[\lkk,\frac{\p h}{\p N}\ak^{3k}\ak^{3i}\right]\tpl\p_t\eta_i\right)\dS\\
\lesssim&\|\eta_k\|_4\left|\p_3 h\ak^{3i}\ak^{3k}\right|_{W^{1,\infty}}\|\p_t\eta_i\|_4\lesssim P\left(\|\eta\|_4,\|v\|_3,\|h\|_4\right),
\end{aligned}
\end{equation}
\begin{equation}\label{B12}
B_{12}\lesssim\left|\p_t\p_3 h\right|_{L^{\infty}}\left|\ak^{3i}\tpl\lkk\eta_i\right|^2\lesssim\|\p_t h\|_3\left|\ak^{3i}\tpl\lkk\eta_i\right|^2.
\end{equation}

For $B_{13}$, we invoke $\p_t\ak^{3i}=-\ak^{3r}\p_m\p_t\ek_r\ak^{mi}$ to get
\begin{equation}\label{B130}
\begin{aligned}
B_{13}=&\ig\left(\frac{\p h}{\p N}\right)\ak^{3k}\tpl\lkk\eta_k\ak^{3r}\p_m\p_t\ek_r\ak^{mi}\tpl\lkk\eta_i\dS\\
=&\ig\left(\frac{\p h}{\p N}\right)\ak^{3k}\tpl\lkk\eta_k\ak^{3r}\p_3\p_t\ek_r\ak^{3i}\tpl\lkk\eta_i\dS \\
&+\sum_{L=1}^2\ig\left(\frac{\p h}{\p N}\right)\ak^{3k}\tpl\lkk\eta_k\ak^{3r}\TP_L\lkk^2 v_r\ak^{Li}\tpl\lkk\eta_i\dS\\
&+\sum_{L=1}^2\ig\left(\frac{\p h}{\p N}\right)\ak^{3k}\tpl\lkk\eta_k\ak^{3r}\TP_L\lkk^2\psi_r\ak^{Li}\tpl\lkk\eta_i\dS\\
=:&B_{131}+B_{132}+B_{133}.
\end{aligned}
\end{equation} Among these three terms, $B_{132}$ cannot be directly controlled. Instead, it will cancel with the main term in $B_{3}$ and the correction term $B_2$. First we give the control of $B_{131}$
\begin{equation}\label{B131}
B_{131}\lesssim\left|\ak^{3i}\TP^4\lkk\eta_i\right|_0^2\left|\p_3 h\ak^{3r}\p_3\p_t\ek_{r}\right|_{L^{\infty}}\lesssim P(\EE_\kk(T)).
\end{equation} In $B_{133}$, we need to use the mollifier property Lemma \ref{tgsmooth} to control $\tpl\lkk\eta$ by sacrificing a $1/\sqrt{\kk}$. Luckily, the $\psi$ term compensates this factor such that the energy bound is still uniform in $\kk$.
\begin{equation}
B_{133}\lesssim\sum_{L=1}^2\frac{1}{\sqrt{\kk}}\left|\p_3 h\ak^{3r}\ak^{Li}\right|_{L^{\infty}}\left|\ak^{3i}\tpl\lkk\eta_i\right|_0\left|\TP\lkk\psi\right|_{L^{\infty}}\left|\eta\right|_{7/2}.
\end{equation} Then we use Sobolev embedding $W^{1,4}(\T^2)\hookrightarrow L^{\infty}(\T^2)$ to get
\[
\left|\TP\psi\right|_{L^{\infty}}\lesssim\left|\TL\psi\right|_{L^4}=\sum_{L=1}^2\left|\TL\left(\eta_{k}-\ek_k\right)\ak^{Lk}\TP_L\lkk v-\TL\ek_k\ak^{Lk}\TP_L\left(v-\lkk^2 v\right)\right|_{L^4}.
\]By the mollifier property \eqref{lkk3}, we know
\[
\left|\TP\psi\right|_{L^{\infty}}\lesssim\sqrt{\kk} P\left(\|\eta\|_4,\|v\|_3\right),
\] and thus
\begin{equation}\label{B133}
B_{133}\lesssim P(\EE_\kk(T)).
\end{equation}

Now we start to analyze $B_3:=-\ig\left(\frac{\p h}{\p N}\right)(\tpl\ek_{k})\ak^{3k}\ak^{3i}(\tpl\ek\cdot\pak v_i)\dS$. Again we separate normal derivative from tangential derivative.
\begin{equation}\label{B30}
\begin{aligned}
B_3=&-\ig\left(\frac{\p h}{\p N}\right)\tpl\ek_{k}\ak^{3k}\ak^{3i}\tpl\ek_r\ak^{3r}\p_3v_i\dS\\
&\underbrace{-\sum_{L=1}^2\ig\left(\frac{\p h}{\p N}\right)\tpl\ek_{k}\ak^{3k}\ak^{3i}\tpl\ek_r\ak^{Lr}\p_Lv_i\dS}_{B_{31}}\\
\lesssim&\left|\ak^{3i}\tpl\ek_i\right|_0^2P\left(\|v\|_3,\|h\|_3,\|\eta\|_4\right)+B_{31}.
\end{aligned}
\end{equation}Here we note that $\ak^{3i}\tpl\ek_i$ can be controlled by commuting one mollifier with the help of Lemma \ref{tgsmooth}
\[
\left|\ak^{3i}\tpl\ek_i\right|_0\leq\left|\lkk\left(\ak^{3i}\tpl\lkk\eta_i\right)\right|_0+\left|\left[\lkk,\ak^{3k}\right]\tpl\lkk\eta_k\right|_0\lesssim P(\|\eta\|_4).
\]

There are still 3 terms remaining to be controlled
\begin{align}
\label{B20} B_2=&-\ig\left(\frac{\p h}{\p N}\right)\tpl\ek_{k}\ak^{3k}\ak^{3i}\tpl\psi_i\dS\\
\label{B1320} B_{132}=&\sum_{L=1}^2\ig\left(\frac{\p h}{\p N}\right)\ak^{3k}\tpl\lkk\eta_k\ak^{3r}\TP_L\lkk^2 v_r\ak^{Li}\tpl\lkk\eta_i\dS\\
\label{B310} B_{31}=&-\sum_{L=1}^2\ig\left(\frac{\p h}{\p N}\right)\tpl\ek_{k}\ak^{3k}\ak^{3i}\tpl\ek_r\ak^{Lr}\TP_Lv_i\dS
\end{align}\

We plug the expression of $\TL\psi$ in $B_2$ to get
\begin{align}
\label{B201} B_2=&-\sum_{L=1}^2\ig \left(\frac{\p h}{\p N}\right)\ak^{3i}\ak^{3k}\tpl \ek_{k}\TP^2(\TL\eta_{r}\ak^{Lr}\TP_L\lkk^2 v_{i})\dS\\
\label{B202} &+\sum_{L=1}^2\ig \left(\frac{\p h}{\p N}\right)\ak^{3i}\ak^{3k}\tpl \ek_{k}\TP^2\ek_{r}\ak^{Lr} \TP_Lv_{i}\dS\\
\label{B203} &+\sum_{L=1}^2\ig \left(\frac{\p h}{\p N}\right)\ak^{3i}\ak^{3k}\tpl \ek_{k}([\TP^2,\ak^{Lr}\TP_L v_{i}]\TL\ek_{r})\dS\\
\label{B204} &+\sum_{L=1}^2\ig \left(\frac{\p h}{\p N}\right)\ak^{3i}\ak^{3k}\tpl \ek_{k}\TP^2\mathbb{P}_{=0}\left(\TL\eta_{k}\ak^{Lk}\TP_L\lkk^2 v-\TL\lkk^2\eta_{k}\ak^{Lk}\TP_L v\right).
\end{align} 

First we notice that $\eqref{B202}+B_{31}=0$, and \eqref{B203}-\eqref{B204} can be controlled by direct computation. Notice that the zero frequency part is always of lower order by Bernstein's inequality.
\begin{align}
\eqref{B203}\lesssim& P\left(\|h\|_3,\|\eta\|_4,\|v\|_4,\left|\ak^{3i}\tpl\lkk\eta_i\right|_0\right)\lesssim P(\EE_\kk(T)),\\
\eqref{B204}\lesssim& \sum_{L=1}^2\left|\p_3 h\ak^{3i}\right|_{L^{\infty}}\left|\ak^{3i}\tpl\lkk\eta_i\right|_0\left|\TL\eta_{k}\ak^{Lk}\TP_L\lkk^2 v-\TL\lkk^2\eta_{k}\ak^{Lk}\TP_L v\right|_0\lesssim P(\EE_\kk(T)).
\end{align}

In order for cancelling $B_{132}$, we just need to commute one mollifier in $B_{201}$ from $\eta_k$ to $\eta_i$. We find that \eqref{B2011} exactly canceles with $B_{132}$. All the commutators \eqref{B2012}-\eqref{B2013} can be controlled directly so we omit the details.
\begin{align}
\label{B2011} \eqref{B201}&=-\sum_{L=1}^2\ig \left(\frac{\p h}{\p N}\right) \ak^{3k}\tpl\lkk\eta_{k}\left(\ak^{3i}\TP_L\lkk^2v_{i}\right)\left(\ak^{Lr}\tpl\lkk\eta_{r}\right)\dS \\
\label{B2012} &~~~~-\sum_{L=1}^2 \ig \left(\frac{\p h}{\p N}\right) \ak^{3k}\tpl\lkk\eta_{k}\left(\left[\lkk,\ak^{3i}\ak^{3k}\ak^{Lr}\TP_i\lkk^2 v_{i}\right]\tpl\eta_{r}\right)\dS \\
\label{B2013} &~~~~-\sum_{L=1}^2 \ig \left(\frac{\p h}{\p N}\right) \ak^{3i}\ak^{3k}\tpl \ek_{k} \left(\left[\TP^2,\ak^{Lr}\TP_i\lkk^2v_{i}\right]\TL\eta_{r}\right)\dS.
\end{align}

Summarising \eqref{B0}-\eqref{B2013}, we get the control of the boundary integral $B$ by
\begin{equation}\label{B}
\int_0^T B(t)\dt\lesssim-\frac{c_0}{4}\left|\ak^{3i}\tpl\lkk\eta_i\right|_0^2+\int_0^T P(\EE_\kk(t))\dt.
\end{equation}

Now, summing up \eqref{K2}, \eqref{K3}, \eqref{K1}, \eqref{tgF}, \eqref{tgK0}, \eqref{B0} and \eqref{B}, we get the estimates of Alinhac good unknowns
\begin{equation}
\left\|\VV\right\|_0^2+\sum_{j=1}^3\left\|\tpl\FP\eta\right\|_0^2+\left\|\swt\tpl h\right\|_0^2+\frac{c_0}{4}\left|\ak^{3i}\tpl\lkk\eta_i\right|_0^2\lesssim\PP_0+P(\EE_\kk(T))\int_0^T P(\EE_\kk(t))\dt.
\end{equation}

Finally by the definition of Alinhac good unknown, we know
\[
\|\TP^4 v\|_0\lesssim\|\VV\|_0+\|\tpl\eta\|_0\|\ak\p v\|_{L^{\infty}}\lesssim\|\VV\|_0+P(\EE_\kk(T))\int_0^T P(\EE_\kk(t))\dt,
\]and thus we finalize the tangential estimates
\begin{equation}\label{tgenergy}
\left\|\TP^4 v\right\|_0^2+\sum_{j=1}^3\left\|\TP^4\FP\eta\right\|_0^2+\left\|\swt\TP^4 h\right\|_0^2+\frac{c_0}{4}\left|\ak^{3i}\TP^4\lkk\eta_i\right|_0^2\lesssim\PP_0+P(\EE_\kk(T))\int_0^T P(\EE_\kk(t))\dt.
\end{equation}

\subsection{Elliptic estimates and reduction to tangentially-differentiated wave equations}\label{divreduce}

The divergence of $v$ and $\FP\eta$ are reduced to $\|\wt \p_t h\|_3$ and $\sum\limits_{j=1}^3\|\wt\FP h\|_3$ in div-curl estimates. In this section we invoke Christodoulou-Lindblad type elliptic estimates (Lemma \ref{GLL}) to further reduce the estimates of $h$ to tangential derivatives.

Let us derive the wave equation of $h$ first. Taking smoothed-divergence $\diva$ in the second equation of \eqref{elastolkk}, we get
\begin{equation}\label{hwave00}
\wt\p_t^2h-\lapak h=-\p_t\ak^{li}\p_l v_i-e''(h)(\p_th)^2-\sum_{j=1}^3\left(\FP\left(\diva\FP\eta\right)-\left[\diva,\FP\right]\FP\eta\right).
\end{equation}

Invoking the formula \eqref{divkkF} of $\diva \FP\eta$, we know the wave equation of $h$ becomes
\begin{equation}\label{hwave0}
\begin{aligned}
\wt\p_t^2h-\lapak h=&-\p_t\ak^{li}\p_l v_i-\pak^i\left(\FP\ek_l\right)\pak^l\left(\FP\ek_i\right)+\sum_{j=1}^3e'(h)\FP^2 h\\
&-e''(h)(\p_th)^2+\sum_{j=1}^3e''(h)\left(\FP h\right)^2\\
&-\sum_{j=1}^3\FP\int_0^T\diva\FP\psi-\pak^r\left(\FP\eta_i\right)\pak^i\p_t\ek_r+\pak^i\left(\FP\ek_r\right)\pak^r v_i\dt.
\end{aligned}
\end{equation}

\begin{rmk}
Among all terms in the source part of \eqref{hwave0}, the first two lines show the terms that also appears in the original system \eqref{elastol}. In Eulerian coordinates, they can be written as $$(\nabla v)\cdot(\nabla v)-(\nabla \FF)\cdot(\nabla \FF)+\sum_{j=1}^3 e'(\h)(\FF^j\cdot\nabla)^2\h-e''(\h)(D_t\h)^2+e''(\h)\sum_{j=1}^3\left((\FF^j\cdot\nabla) \h\right)^2.$$
The last two terms are of higher order weight function and thus much are smaller. Therefore the main terms are the first three terms. 
\end{rmk}

Now we start to control $h$. From the a priori assumption \eqref{small1} and Poincar\'e inequality, we have 
\[
\|h\|_4\lesssim\|\p^4 h\|_0\leq \|\pak h\|_3+\|a-I_3\|_{3}\|h\|_4\lesssim \|\pak h\|_3+\delta\|h\|_4,
\] and the $\delta$-term can be absorbed by LHS if we choose $\delta>0$ to be sufficiently small. From now on, we will use the notation $\|f\|_r\approx \|\pak f\|_{r-1}$ to record similar inequality for $1\leq r\leq 4$. 

Invoking Lemma \ref{GLL}, we have
\begin{equation}\label{h40}
\|h\|_4\approx\|\pak h\|_3\lesssim P(\|\ek\|_3)\left(\|\lapak h\|_2+\|\TP\ek\|_3\|h\|_3\right).
\end{equation}

Then inserting the wave equation, we know
\begin{equation}
\begin{aligned}
\|\lapak h\|_2\lesssim& \left\|\wt\p_t^2 h\right\|_2+\sum_{j=1}^3\left\|\wt\FP^2 h\right\|_2\\
&+\left\|\p_t\ak\cdot\p v\right\|_2+\sum_{j=1}^3\left\|\pak(\FP\eta)\right\|_2^2+\left\|e''(h)(\p_th)^2\right\|_2+\sum_{j=1}^3\left\|e''(h)\left(\FP h\right)^2\right\|_2\\
&+\sum_{j=1}^3\int_0^T\left\|\FP\left(\diva\FP\psi-\pak^r\left(\FP\eta_i\right)\pak^i\p_t\ek_r+\pak^i\left(\FP\ek_r\right)\pak^r v_i\right)\right\|_2\dt\\
\lesssim&\left\|\wt\p_t^2 h\right\|_2+\sum_{j=1}^3\left\|\wt\FP^2 h\right\|_2\\
&+\|\p v\|_2^2\|\p\ek\|_{L^{\infty}}+\sum_{j=1}^3\|\ak\|_2^2\left\|\FP\eta\right\|_3^2+\left\|\left(e'(h)\p_t h\right)\right\|_2^2+\sum_{j=1}^3\left\|\left(e'(h)\FP h\right)\right\|_2^2\\
&+\int_0^T P\left(\|\ak\|_3,\sum\limits_{j=1}^3\|\FP\psi\|_4+\|\FP\ek\|_4+\|\FP\eta\|_4,\|\p_t\ek\|_4\right)\dt\\
\lesssim&\left\|\wt\p_t^2 h\right\|_2+\sum_{j=1}^3\left\|\wt\FP^2 h\right\|_2+\PP_0+\int_0^TP(\EE_\kk(t))\dt.
\end{aligned}
\end{equation}Therefore, we reduce the estimates of $\|h\|_4$ to $\|\wt\p_t^2 h\|_2$ and $\sum\limits_{j=1}^3\|\wt\FP^2 h\|_2$. 
\begin{rmk}
~

\begin{enumerate}
\item In \eqref{h40}, the essential terms are $\|\wt\p_t^2 h\|_2$ and $\sum\limits_{j=1}^3\|\wt\FP^2 h\|_2$ obtained from $\lapak h$. The term $P(\|\ek\|_3)\|\TP\ek\|_3\|h\|_3$ is actually a lower order term. One can repeatedly apply Lemma \ref{GLL} to $h$ in order to reduce to the estimates of $\|h\|_1$ which can be directly controlled by the $L^2$-estimate of the wave equation \eqref{hwave0}. Due to this reason, we will omit the processes of controlling these lower order terms appearing in the elliptic estimates throughout this manuscript.

\item Note that $h|_{\Gamma}=0$ and $\p_t,\FP$ are tangential derivatives on the boundary, so $\p_t^2 h$ and $\FP^2 h$ are still vanishing on the boundary and thus we can use Lemma \ref{GLL} to further reduce them to full tangentially-differentiated quantites and use the tangentially-differentiated wave equation to finalize the control. We postpone the proof after the elliptic estimates of $\p_t h$.
\end{enumerate}
\end{rmk}

Next we control $\|\p_t h\|_3$ in the same way by inserting the $\p_t$-differentiated wave equation \eqref{hwave0}
\begin{equation}
\|\p_t h\|_3\approx\|\pak \p_t h\|_2\lesssim P(\|\ek\|_2)\left(\|\lapak \p_t h\|_1+\|\TP\ek\|_2\|\p_t h\|_2\right),
\end{equation}and then
\begin{equation}
\begin{aligned}
\|\lapak \p_t h\|_1\lesssim&\|\p_t\lapak h\|_1+\|[\lapak,\p_t] h\|_1\\
\lesssim&\left\|\wt\p_t^3 h\right\|_1+\sum_{j=1}^3\left\|\wt \FP^2\p_t h\right\|_1+\left\|e''(h)\p_t h\p_t^2 h\right\|_1+\sum_{j=1}^3\left\|e''(h)\p_t h\FP^2 h\right\|_1\\
&+\|\p_t(\p_t\ak \p v)\|_1+\sum_{j=1}^3\left\|\p_t\left(\pak(\FP\eta)\right)\right\|_1^2+\|[\lapak,\p_t] h\|_1\\
&+\|\p_t(e''(h)(\p_t h)^2)\|_1+\sum_{j=1}^3\left\|\p_t\left(e''(h)(\FP h)^2\right)\right\|_1\\
&+\sum_{j=1}^3\left\|\FP\left(\diva\FP\psi-\pak^r\left(\FP\eta_i\right)\pak^i\p_t\ek_r+\pak^i\left(\FP\ek_r\right)\pak^r v_i\right)\right\|_1.
\end{aligned}
\end{equation}
Notice that $|e^{(k)}(h)|\lesssim|e'(h)|^k$, so the power of weight function $\wt$ is always enough. For example,
\[
\left|\p_t(e''(h)(\p_t h)^2)\right|\leq\left|2(e''(h)\p_t^2 h)\p_t h\right|+\left|e'''(h)(\p_t h)^3\right|\lesssim\left|((\wt)^2\p_t^2 h)\p_t h\right|+\left|(\wt\p_t h)^3\right|.
\] Therefore, we know
\begin{equation}\label{ht30}
\|\lapak \p_t h\|_1\lesssim\left\|\wt\p_t^3 h\right\|_1+\sum_{j=1}^3\left\|\wt \FP^2\p_t h\right\|_1+\PP_0+\int_0^T P(\EE_\kk(t))\dt.
\end{equation}
As one can see, the remaining terms in \eqref{ht30} are exactly from the energy functionalsof $\p_t^3$-differentiated and $\sum\limits_{j=1}^3\FP^2\p_t$-differentiated wave equation \eqref{hwave0}. 

Similarly, $\|\diva \FP\eta\|_3\approx\|\wt\FP h\|_3$ can be controlled by Lemma \ref{GLL} and inserting $\FP$-differentiated wave equation. We omit the details and get the following conclusion
\begin{equation}\label{hF30}
\|\dive \FP\eta\|_3\lesssim\left\|(\wt)^2\p_t^2\FP h\right\|_1+\sum_{k=1}^3\left\|(\wt)^2 (\F_k\cdot\p)^2\FP h\right\|_1+\PP_0+\int_0^T P(\EE_\kk(t))\dt,
\end{equation}where the remaining terms will be controlled by analyzing $\p_t^2\FP$-differentiated and $\sum\limits_{j,k=1}^3(\F_k\cdot\p)^2\FP$-differentiated wave equation \eqref{hwave0}.

Now we come to further reduce the remaining terms in \eqref{h40} to full tangential derivatives.

We have
\begin{equation}\label{htt20}
\left\|\swt \p_t^2 h\right\|_2\approx\left\|\pak\left(\swt\p_t^2 h\right)\right\|_1\leq \left\|\swt\pak\p_t^2 h\right\|_1+\left\|(\swt)^{-1/2}e''(h)\pak h\p_t^2 h\right\|_1,
\end{equation} and then
\begin{equation}
\begin{aligned}
\left\|\swt\pak\p_t^2 h\right\|_1\lesssim& P(\|\ek\|_2)\left\|\swt\lapak \p_t^2 h\right\|_0+\PP_0+\int_0^TP(\EE_{\kk}(t))\dt \\
\lesssim& P(\|\ek\|_2)\left(\left\|(\wt)^{\frac32}\p_t^4 h\right\|_0+\sum_{j=1}^3\left\|(\wt)^{\frac32}\FP^2 \p_t^2h\right\|_0\right)+\PP_0+\int_0^TP(\EE_{\kk}(t))\dt.
\end{aligned}
\end{equation}

Similarly we have
\begin{equation}\label{hFF20}
\begin{aligned}
\|\wt\FP^2 h\|_2\approx& \|\wt\pak\FP^2 h\|_1+\|e''(h)\pak h\FP^2 h\|_1\\
\lesssim& P(\|\ek\|_2)\|\wt\lapak \FP^2 h\|_0+\PP_0+\int_0^TP(\EE_{\kk}(t))\dt \\
\lesssim& P(\|\ek\|_2)\left(\|(\wt)^2\p_t^2\FP^2 h\|_0+\sum_{k=1}^3\left\|(\wt)^2(\F_k\cdot\p)^2\FP^2h\right\|_0\right)\\
&+\PP_0+\int_0^TP(\EE_{\kk}(t))\dt.
\end{aligned}
\end{equation}

Summarizing \eqref{ht30}, \eqref{hF30}, \eqref{htt20} and \eqref{hFF20}, we still need to control the following quantities.
\begin{equation}\label{remainh}
\begin{aligned}
&\left\|\wt\p_t^3 h\right\|_1+\|(\wt)^{\frac32}\p_t^4 h\|_0\\
&+\sum_{j=1}^3\left\|(\wt)^2\p_t^2\FP h\right\|_1+\sum_{j=1}^3\|(\wt)^{\frac32}\p_t^2\FP^2 h\|_0\\
&+\sum_{j=1}^3\left\|(\wt)^2\FP^2 \p_t^2h\right\|_0+\sum_{j=1}^3\left\|\wt \FP^2\p_t h\right\|_1\\
&+\sum_{j,k=1}^3\left\|(\wt)^2(\F_k\cdot\p)^2\FP^2h\right\|_0+\sum_{k=1}^3\left\|(\wt)^2 (\F_k\cdot\p)^2\FP h\right\|_1.
\end{aligned}
\end{equation}

\begin{rmk}
The quantites in \eqref{remainh} exactly correspond to the weighted energy functionals of $\p_t^3$-differentiated, $\sum\limits_{j=1}^3\p_t^2\FP$-differentiated, $\sum\limits_{j=1}^3\p_t\FP^2$-differentiated and $\sum\limits_{j,k=1}^3(\F_k\cdot\p)^2\FP$-differentiated wave equation of $h$ \eqref{hwave0}. Keep in mind that these derivatives are all tangential, and thus there is no boundary integral appearing in the energy estimates.
\end{rmk}

\paragraph*{Reduction of the remaining terms in the energy functional $\EE_\kk(T)$}

Comparing with the energy functional $\EE_\kk(T)$, we still need to control
\begin{equation}\label{remainvFvF}
\begin{aligned}
&\|\p_t v\|_3,\|\FP\p_t\eta\|_3,\|\p_t^2 v\|_2,\|\FP\p_t^2\eta\|_2,\\
&\|\wt\p_t^3 v\|_1,\|\wt\FP\p_t^3\eta\|_1,\|\wt\p_t^4 v\|_0,\|\wt\FP\p_t^4\eta\|_0.
\end{aligned}
\end{equation}

Let us start with full time derivatives. Recall $\p_t v=-\pak h+\sum\limits_{j=1}^3\FP^2\eta$ gives 
\[
\p_t^4 v=-\pak\p_t^3 h+[\pak,\p_t^3]h+\sum\limits_{j=1}^3\FP^2\p_t^3\eta, 
\]and thus
\begin{equation}\label{remainvtttt}
\|\wt\p_t^4 v\|_0\lesssim\|\wt \p_t^3 h\|_1+\sum_{j=1}^3\|\F_j\|_{L^{\infty}}\left\|\wt\FP\p_t^3\eta\right\|_1+\text{lower order terms}.
\end{equation} Then
\begin{equation}\label{remainFtttt}
\|\wt\FP\p_t^4 \eta\|_0\lesssim\|\F_j\|_{L^{\infty}}\|\wt\p_t^3 v\|_1+\|\wt\FP\p_t^3\psi\|_0.
\end{equation}

Recall that in \eqref{psittt2} we only need at most 3 time derivatives of $v$ to control $\wt\FP\p_t^3\psi$. Therefore, the control of full time derivatives of $v$ and $\FP\eta$ has been reduced to less time derivatives. Similarly, we have
\begin{equation}\label{remainvF}
\begin{aligned}
\|\swt\p_t^3 v\|_1\lesssim&\|\swt\pak \p_t^2 h\|_1+\sum_{j=1}^3\|\F_j\|_{2}\|\swt\FP\p_t^2 \eta\|_2+\text{lower order terms}.\\
\|\swt\FP\p_t^3 \eta\|_1\lesssim&\|\F_j\|_2\|\swt\p_t^2 v\|_2+\|\swt\FP\p_t^2\psi\|_1.\\
\|\p_t^2 v\|_2\lesssim&\|\pak\p_t h\|_2+\sum_{j=1}^3\|\F_j\|_2\|\FP\p_t \eta\|_3+\text{lower order terms}.\\
\|\FP\p_t^2\eta\|_2\lesssim&\|\F_j\|_2\|\p_t v\|_3+\|\FP\p_t\psi\|_2.\\
\|\p_t v\|_3\lesssim&\|\pak h\|_3+\sum_{j=1}^3\|\F_j\|_3\|\FP\eta\|_4.\\
\|\FP\p_t\eta\|_3\lesssim&\|\F_j\|_3\|v\|_4+\|\FP\psi\|_4.
\end{aligned}
\end{equation} Invoking Lemma \ref{etapsi}, all these time derivatives are reduced to full spatial derivatives and the estimates of $h$. Since the estimates of $\|v\|_4$ and $\|\FP\eta\|_4$ are reduced to the elliptic estimates of $h$, we know it remains to control the quantities in \eqref{remainh}, which will be controlled via tangentially-differentiated wave equation of $h$, and thus can be controlled by $\PP_0+\int_0^T P(\EE_\kk(t))\dt$.

\subsection{Special Structure of the wave equation}\label{wavediv}

This part is the key to the whole proof to establish the energy estimates. Let us start with the wave equation of $h$ \eqref{hwave0}.

\begin{equation}
\begin{aligned}
\wt\p_t^2h-\lapak h=&\underbrace{-\p_t\ak^{li}\p_l v_i-\pak^i\left(\FP\ek_l\right)\pak^l\left(\FP\ek_i\right)}_{L_1}+\sum_{j=1}^3e'(h)\FP^2 h\\
&\underbrace{-e''(h)(\p_th)^2+\sum_{j=1}^3e''(h)\left(\FP h\right)^2}_{L_2}\\
&-\sum_{j=1}^3\FP\int_0^T\underbrace{\diva\FP\psi-\pak^r\left(\FP\eta_i\right)\pak^i\p_t\ek_r+\pak^i\left(\FP\ek_r\right)\pak^r v_i}_{L_3^j}\dt.
\end{aligned}
\end{equation}

We will show the full details in the case of $\p_t^3$-differentiated wave equation and $\sum\limits_{j,k=1}^3(\F_k\cdot\p)^2\FP$-differentiated wave equation, which are the most complicated cases. The other mixed cases follow in the same way so we will omit those details.

\subsubsection{Fully time-differentiated wave equation}

First, we differentiate $\p_t^3$ in the wave equation \eqref{hwave0} to get
\begin{equation}\label{htttwave}
\begin{aligned}
\wt\p_t^5h-\lapak\p_t^3 h=&\sum_{j=1}^3e'(h)\FP^2\p_t^3 h\\
&+\underbrace{\left[e'(h),\p_t^3\right]\p_t^2 h+\left[\p_t^3,\lapak\right]h+\sum_{j=1}^3\left[\p_t^3,\wt\right]\FP^2 h}_{L_{41}}\\
&+\p_t^3(L_1+L_2)-\sum_{j=1}^3\FP\int_0^T\p_t^3 L_3^j\dt.
\end{aligned}
\end{equation}

Then we take $L^2$-inner product of \eqref{htttwave} with $(e'(h))^2\p_t^4 h$ to get
\begin{equation}\label{ttt1}
\begin{aligned}
&\io (\wt)^3\p_t^5 h\p_t^4 h\dy-\io(\wt)^2\diva(\pak \p_t^3 h)\p_t^4 h\dy\\
=&\sum_{j=1}^3\io(\wt)^3\FP^2\p_t^3 h\p_t^4 h\dy \\
&+\io (\wt)^2 \left(L_{41}+\p_t^3(L_1+L_2)\right)\p_t^4 h\dy\\
&-\sum_{j=1}^3\io(\wt)^2\p_t^4 h\cdot\left(\FP\int_0^T\p_t^3 L_3^j\dt\right)\dy.
\end{aligned}
\end{equation}

The LHS of \eqref{ttt1} gives the weighted energy of $\p_t^4 h$ and $\p_t^3 h$. We integrate by parts in the second term to get
\begin{equation}\label{ttt2}
\begin{aligned}
\text{LHS of }\eqref{ttt1}=&\frac12\frac{d}{dt}\io\left|(\wt)^{\frac32}\p_t^4 h\right|^2\dt+\frac12\frac{d}{dt}\io\left|\wt\pak \p_t^3 h\right|^2\dy \\
&-\frac32\io(\wt)^2e''(h)\p_t h\left|\p_t^4 h\right|^2\dy-\io e'(h)e''(h)\p_t h\left|\pak\p_t^3 h\right|^2\dy\\
&+2\io \wt e''(h)(\pak h)\cdot(\pak\p_t^3 h)\p_t^4 h\dy+\io(\wt)^2\pak\p_t^3 h\cdot\left([\pak,\p_t]\p_t^3 h\right)\dy\\
\end{aligned}
\end{equation}

The first term on RHS of \eqref{ttt1} gives the weighted energy of $\sum\limits_{j=1}^3 \FP\p_t^3 h$ after integrating $\FP$ by parts
\begin{equation}\label{ttt3}
\begin{aligned}
&\io(\wt)^3\FP^2\p_t^3 h\p_t^4 h\dy \\
=&-\frac12\frac{d}{dt}\io\left|(\wt)^{\frac32}\FP\p_t^3 h\right|^2\dy\\
&+\frac32\io(\wt)^2e''(h)\p_t h\left|\FP\p_t^3 h\right|^2\dy-3\io(\wt)^2e''(h)\left(\FP h\right)\left(\FP\p_t^3 h\right)\p_t^4 h\dy\\
&-\io\p_l\F_{lj} (\wt)^3\left(\FP\p_t^3 h\right)\p_t^4 h\dy
\end{aligned}
\end{equation}

The remaining terms can all be directly controlled by invoking the physical conditions $|e^{(k)}(h)|\lesssim|e'(h)|^k$. We list the details as follow:

In \eqref{ttt2}:
\begin{equation}\label{ttt4}
\begin{aligned}
&-\frac32\io(\wt)^2e''(h)\p_t h\left|\p_t^4 h\right|^2\dy-\io e'(h)e''(h)\p_t h\left|\pak\p_t^3 h\right|^2\dy\\
&+2\io \wt e''(h)(\pak h)\cdot(\pak\p_t^3 h)\p_t^4 h\dy+\io(\wt)^2\pak\p_t^3 h\cdot\left([\pak,\p_t]\p_t^3 h\right)\dy\\
\lesssim&\|\swt\|_{L^{\infty}}\left(\left\|(\wt)^{\frac32}\p_t^4 h\right\|_0^2+\left\|\wt\pak\p_t^3 h\right\|_0^2\right)\\
&+\|\swt\|_{L^{\infty}}\|\ak\|_{L^{\infty}}\|\p h\|_{L^{\infty}}\|\wt\pak\p_t^3 h\|_0\|(\wt)^{\frac32}\p_t^4 h\|_0+\|\pak \p_t \eta\|_{L^{\infty}}\left\|\wt\pak\p_t^3 h\right\|_0^2.
\end{aligned}
\end{equation}

In \eqref{ttt3}:
\begin{equation}\label{ttt5}
\begin{aligned}
&\frac32\io(\wt)^2e''(h)\p_t h\left|\FP\p_t^3 h\right|^2\dy-3\io(\wt)^2e''(h)\left(\FP h\right)\left(\FP\p_t^3 h\right)\p_t^4 h\dy\\
&-\io\p_l\F_{lj} (\wt)^3\left(\FP\p_t^3 h\right)\p_t^4 h\dy\\
\lesssim&\|\swt\|_{L^{\infty}}\left\|(\wt)^{\frac32}\FP\p_t^3 h\right\|_0^2\\
&+\left(\|\p\F\|_{L^{\infty}}+\|e'(h)\FP h\|_{L^{\infty}}\right)\left\|(\wt)^{\frac32}\FP\p_t^3 h\right\|_0\left\|(\wt)^{\frac32}\p_t^4 h\right\|_0
\end{aligned}
\end{equation}

In \eqref{ttt1},
\[
\io (\wt)^2 \left(L_{41}+\p_t^3(L_1+L_2)\right)\p_t^4 h\dy\lesssim\left\|(\wt)^{\frac32}\p_t^4 h\right\|_0\left\|\swt\left(L_{41}+\p_t^3(L_1+L_2)\right)\right\|_0.
\] We compute that
\begin{align*}
&\|\swt \p_t^3L_1\|_0\lesssim\|\swt\p_t^3 v\|_0\|\p v\|_{L^{\infty}}+\|\swt\|_{L^{\infty}}\|\p_t^2 v\|_0\|\p_t v\|_{L^{\infty}}\\
+&\|\ak\|_{L^{\infty}}^2\left(\|\swt\FP\p_t^3 \ek\|_0\|\p \FP\ek\|_{L^{\infty}}+\|\swt\|_{L^{\infty}}\|\FP\p_t^2 \ek\|_0\|\FP\p_t \ek\|_{L^{\infty}}\right),
\end{align*}and
\begin{align*}
\p_t^3 L_2=&e^{(5)}(h)(\p_t h)^5+8e^{(4)}(h)(\p_t h)^3\p_t^2 h+12e^{(3)}(h)\p_t h(\p_t^2 h)^2+4e^{(3)}(h)(\p_t h)^2\p_t^3 h\\
&+6e''(h)\p_t^2 h\p_t^3 h+2e''(h)\p_t h\p_t^4 h\\
&+e''(h)\left(2(\FP h)(\FP\p_t^3 h)+6(\FP\p_t^2 h)(\FP\p_t h)\right)\\
&+e^{(3)}(h)\big(\p_t^3 h(\FP h)^2+\p_t^2 h(\FP h)(\FP\p_t h)\\
&~~~~+\p_th(2(\FP\p_t h)^2+2(\FP h)(\FP\p_t^2 h))\big)\\
&+e^{(4)}(h)\left(2(\p_t h)^2(\FP h)(\FP\p_t h)+2(\p_t h)(\p_t^2 h)(\FP h)^2\right)\\
&+e^{(5)}(h)(\p_t h)^3(\FP h)^2\\
\Rightarrow \|\swt\p_t^3 L_2\|_0\lesssim& P\bigg(\|\swt\|_{L^{\infty}},\|\p_t h\|_{L^{\infty}},\|\swt\p_t^2 h\|_2,\|\wt\p_t^3 h\|_1,\|(\wt)^{\frac32}\p_t^4 h\|_0\\
&\|\FP h\|_{L^{\infty}},\|\swt\p_t \FP h\|_2,\|\wt\p_t^2\FP h\|_1,\|(\wt)^{\frac32}\p_t^3\FP h\|_0\bigg),
\end{align*}
and $L_{41}$ consists of quadratic and cubic terms and each term containing 5 derivatives in total with highest order equal to 4. Here we list the precise form, but analogous computations will be omitted in the remaining of the paper.
\begin{align*}
L_{41}=&\left[e'(h),\p_t^3\right]\p_t^2 h+\left[\p_t^3,\lapak\right]h+\sum_{j=1}^3\left[\p_t^3,\wt\right]\FP^2 h\\
=&e'(h)(3\p_t h\p_t^4 h+4\p_t^2 h\p_t^3 h)+3e''(h)((\p_t h)^2\p_t^3 h+\p_t h(\p_t^2 h)^2)+e'''(h)(\p_th)^3\\
&+(\lapak \p_t^2v)(\pak h)+(\lapak \p_tv)(\pak\p_t h)+(\lapak v)(\pak \p_t^2 h)\\
&+(\pak \p_t^2v)(\pak\pak h)+(\pak \p_tv)(\pak\pak\p_t h)+(\pak v)(\pak\pak \p_t^2 h)\\
&+\sum_{l_1+l_2+l_3=1}(\pak\p_t^{l_1} v)(\lapak\p_t^{l_2} v)(\pak\p_t^{l_3} h)+(\pak\p_t^{l_1} v)(\pak\p_t^{l_2} v)(\pak\pak\p_t^{l_3} h)\\
&+\sum_{j=1}^3 3e'(h)(\p_t h\p_t^2\FP^2 h+2\p_t^2 h\FP^2\p_t h+\p_t^3 h\FP^2 h)\\
&~~~~~~~~+2e''(h)(3(\p_t h)^2\FP^2\p_t h+3\p_t^2 h\p_t h\FP^2 h)+e'''(h)(\p_th)^3\FP^2 h,
\end{align*}which gives
\begin{equation}\label{ttt6}
\begin{aligned}
&\io (\wt)^2 \left(L_{41}+\p_t^3(L_1+L_2)\right)\p_t^4 h\dy\\
\lesssim &\sum_{j=1}^3 P\bigg(\left\|(\wt)^{\frac32}\left(\p_t^4 h,\FP^2\p_t^2 h\right)\right\|_0,\left\|\wt\left(\p_t^3 h,\FP^2\p_t h\right)\right\|_1,\\
&~~~~~~\left\|\swt\left(\p_t^2 h,\FP^2 h\right)\right\|_2,\left\|\p_t h\right\|_2,\left\|v_t\right\|_3,\left\|v\right\|_4,\left\|\eta\right\|_4\bigg).
\end{aligned}
\end{equation} Finally, 
\begin{equation}\label{ttt7}
\begin{aligned}
&-\sum_{j=1}^3\io(\wt)^2\p_t^4 h\cdot\left(\FP\int_0^T\p_t^3 L_3^j\dt\right)\dy\\
\lesssim&\sum_{j=1}^3\left\|(\wt)^{\frac32}\p_t^4 h\right\|_0\left\|\swt\FP \p_t^2 L_3^j\right\|_0\\
\lesssim&\sum_{j=1}^3\left\|(\wt)^{\frac32}\p_t^4 h\right\|_0P\left(\|\ak\|_2,\|\p_t^2(\FP\eta,\FP\ek)\|_1,\|v\|_3,\|v_t\|_2,\|\FP\p_t\psi\|_2,\|\FP\psi\|_3\right).
\end{aligned}
\end{equation}

Summarizing \eqref{ttt1}-\eqref{ttt7}, we get the energy estimates of $\p_t^3$-differentiated wave equation
\begin{equation}\label{htttenergy}
\frac{d}{dt}\left(\left\|(\wt)^{\frac32}\p_t^4 h\right\|_0^2\dt+\left\|\wt \p_t^3 h\right\|_1^2+\sum_{j=1}^3\left\|(\wt)^{\frac32}\FP\p_t^3 h\right\|_0^2\right)\lesssim P(\EE_\kk(t)).
\end{equation}

\subsubsection{Fully $(\F\cdot\p)$-differentiated wave equation}\label{hFFFdiv}

We differentiate $(\F_k\cdot\p)^2(\F_l\cdot\p)$ (no summation on $k,l$!) in the wave equation \eqref{hwave0} to get
\begin{equation}\label{hFFFwave}
\begin{aligned}
&\wt(\F_k\cdot\p)^2(\F_l\cdot\p)\p_t^2 h-\lapak\left((\F_k\cdot\p)^2(\F_l\cdot\p)\right) h\\
=&\sum_{j=1}^3 e'(h)(\F_k\cdot\p)^2(\F_l\cdot\p)\FP^2 h\\
&+\underbrace{\left[e'(h),(\F_k\cdot\p)^2(\F_l\cdot\p)\right]\p_t^2 h+\left[(\F_k\cdot\p)^2(\F_l\cdot\p),\lapak\right]h+\sum_{j=1}^3\left[(\F_k\cdot\p)^2(\F_l\cdot\p),\wt\right]\FP^2 h}_{L_{42}}\\
&+(\F_k\cdot\p)^2(\F_l\cdot\p)(L_1+L_2)-\sum_{j=1}^3\FP\int_0^T(\F_k\cdot\p)^2(\F_l\cdot\p) L_3^j\dt
\end{aligned}
\end{equation}

Then we take the $L^2$-inner product of \eqref{hFFFwave} and $(\wt)^2(\F_k\cdot\p)^2(\F_l\cdot\p)\p_t h$ to get
\begin{equation}\label{FFF0}
\begin{aligned}
&\io(\wt)^3\left((\F_k\cdot\p)^2(\F_l\cdot\p)\p_t h\right)\p_t\left((\F_k\cdot\p)^2(\F_l\cdot\p)\p_t h\right)\dy\\
&-\io(\wt)^2\diva\left(\pak(\F_k\cdot\p)^2(\F_l\cdot\p) h\right) \p_t \left((\F_k\cdot\p)^2(\F_l\cdot\p)h\right)\dy \\
=&\sum_{j=1}^3\io(\wt)^3\left((\F_k\cdot\p)^2(\F_l\cdot\p)\FP^2 h\right)\cdot \left((\F_k\cdot\p)^2(\F_l\cdot\p)\p_t h\right)\dy\\
&+\io(\wt)^2\left(L_{42}+(\F_k\cdot\p)^2(\F_l\cdot\p)(L_1+L_2)\right)\cdot \left((\F_k\cdot\p)^2(\F_l\cdot\p)\p_t h\right)\dy\\
&-\sum_{j=1}^3\io(\wt)^2(\F_k\cdot\p)^2(\F_l\cdot\p)\p_t h\cdot\left((\F_k\cdot\p)^2(\F_l\cdot\p)\FP\int_0^T L_3^j\right)\dy
\end{aligned}
\end{equation}

For LHS of \eqref{FFF0}, we integrate by part in the second term to get the energy term plus some commutators
\begin{equation}\label{FFF1}
\begin{aligned}
\text{LHS of }\eqref{FFF0}=&\frac12\frac{d}{dt}\io \left|(\wt)^{\frac32}(\F_k\cdot\p)^2(\F_l\cdot\p)\p_t h\right|^2+\left|\wt\pak(\F_k\cdot\p)^2(\F_l\cdot\p) h\right|^2\dy\\
+&\io\frac32(\wt)^2e''(h)\p_t h\left|(\F_k\cdot\p)^2(\F_l\cdot\p)h\right|^2+e'(h)e''(h)\p_t h\left|\pak(\F_k\cdot\p)^2(\F_l\cdot\p)h\right|^2\dy\\
+&\io 2e'(h)e''(h)\left(\pak(\F_k\cdot\p)^2(\F_l\cdot\p) h\right)\p_t\left((\F_k\cdot\p)^2(\F_l\cdot\p)h\right)\dy\\
+&\io(\wt)^2\left(\pak(\F_k\cdot\p)^2(\F_l\cdot\p)h\right)\cdot\left(\left[\diva,\p_t\right](\F_k\cdot\p)^2(\F_l\cdot\p)h\right)\dy.
\end{aligned}
\end{equation}

The first term on the RHS of \eqref{FFF0} gives the weight energy of $h$ after we integrate $\FP$ by parts for each $j$
\begin{equation}\label{FFF2}
\begin{aligned}
&\sum_{j=1}^3\io(\wt)^3\left((\F_k\cdot\p)^2(\F_l\cdot\p)\FP^2 h\right)\cdot \left((\F_k\cdot\p)^2(\F_l\cdot\p)\p_t h\right)\dy\\
=&-\sum_{j=1}^3\frac12\frac{d}{dt}\io\left|(\wt)^{\frac32}(\F_k\cdot\p)^2(\F_l\cdot\p)\FP h\right|^2\dy\\
&+\sum_{j=1}^3\frac32\io(\wt)^2e''(h)\p_t h\left|(\F_k\cdot\p)^2(\F_l\cdot\p)\FP h\right|^2\dy\\
&-\sum_{j=1}^3 3\io(\wt)^2e''(h)(\FP h)\left((\F_k\cdot\p)^2(\F_l\cdot\p)\FP h\right)\left((\F_k\cdot\p)^2(\F_l\cdot\p)\p_t h\right)\dy\\
&-\sum_{j=1}^3\io(\wt)^3\p_m\F_{mj}\left((\F_k\cdot\p)^2(\F_l\cdot\p)\FP h\right)\left((\F_k\cdot\p)^2(\F_l\cdot\p)\p_t h\right)\dy.
\end{aligned}
\end{equation}

Now we analyze the remaining terms. Keep in mind the physical constraints \eqref{sound}: $|e^{(k)}(h)|\lesssim|e'(h)|^k$. 

In \eqref{FFF1}, we have
\begin{equation}\label{FFF3}
\begin{aligned}
&\io 2e'(h)e''(h)\left(\pak(\F_k\cdot\p)^2(\F_l\cdot\p) h\right)\p_t\left((\F_k\cdot\p)^2(\F_l\cdot\p)h\right)\dy\\
&+\io(\wt)^2\left(\pak(\F_k\cdot\p)^2(\F_l\cdot\p)h\right)\cdot\left(\left[\diva,\p_t\right](\F_k\cdot\p)^2(\F_l\cdot\p)h\right)\dy\\
\lesssim& 2\|\swt\|_{L^{\infty}}\left\|\wt(\F_k\cdot\p)^2(\F_l\cdot\p)h \right\|_1\left\|(\wt)^{\frac32}(\F_k\cdot\p)^2(\F_l\cdot\p)\p_t h\right\|_0\\
&+\|\pak v\|_{L^{\infty}}\left\|\wt(\F_k\cdot\p)^2(\F_l\cdot\p)h \right\|_1^2.
\end{aligned}
\end{equation}

In \eqref{FFF2}, we have
\begin{equation}\label{FFF4}
\begin{aligned}
&~~~~\sum_{j=1}^3\frac32\io(\wt)^2e''(h)\p_t h\left|(\F_k\cdot\p)^2(\F_l\cdot\p)\FP h\right|^2\dy\\
&-\sum_{j=1}^3 3\io(\wt)^2e''(h)(\FP h)\left((\F_k\cdot\p)^2(\F_l\cdot\p)\FP h\right)\left((\F_k\cdot\p)^2(\F_l\cdot\p)\p_t h\right)\dy\\
&-\sum_{j=1}^3\io(\wt)^3\p_m\F_{mj}\left((\F_k\cdot\p)^2(\F_l\cdot\p)\FP h\right)\left((\F_k\cdot\p)^2(\F_l\cdot\p)\p_t h\right)\dy\\
\lesssim\sum_{j=1}^3&\|\wt\p_t h\|_{L^{\infty}}\left\|(\wt)^{\frac32}(\F_k\cdot\p)^2(\F_l\cdot\p)\FP h\right\|_0^2\\
+&3\|\wt\FP h\|_{L^{\infty}}\left\|(\wt)^{\frac32}(\F_k\cdot\p)^2(\F_l\cdot\p)\FP h\right\|_0\left\|(\wt)^{\frac32}(\F_k\cdot\p)^2(\F_l\cdot\p)\p_t h\right\|_0\\
+&P\left(\|\wt\|_{L^{\infty}},\left\|(\wt)^{\frac32}(\F_k\cdot\p)^2(\F_l\cdot\p)\FP h\right\|_0,\|\F\|_4,\|\p_t h\|_3\right).
\end{aligned}
\end{equation}

It remains to analyze the last two lines in \eqref{FFF1}. The most difficult term is the last line. We should first integrate $\p_t$ by parts and use the time integral outside $L_3^j$ to eliminate that time derivative, and then integrate $\FP$ by parts to get
\begin{equation}\label{FFF5}
\begin{aligned}
&-\int_0^T\io(\wt)^2(\F_k\cdot\p)^2(\F_l\cdot\p)\p_t h\cdot\left((\F_k\cdot\p)^2(\F_l\cdot\p)\FP\int_0^t L_3^j(\tau)\right)\dy d\tau\\
\overset{\p_t}{=}&\int_0^T\io(\wt)^2\left((\F_k\cdot\p)^2(\F_l\cdot\p) h\right) \cdot\left((\F_k\cdot\p)^2(\F_l\cdot\p)\FP L_3^j(t)\right)\dy\dt\\
&+2\int_0^T\wt e''(h)\p_t h(\F_k\cdot\p)^2(\F_l\cdot\p)h\cdot\left((\F_k\cdot\p)^2(\F_l\cdot\p)\FP\int_0^t L_3^j\right)\dy\dt\\
&-\io(\wt)^2(\F_k\cdot\p)^2(\F_l\cdot\p)h\cdot\left((\F_k\cdot\p)^2(\F_l\cdot\p)\FP\int_0^TL_3^j\right)\dy\\
\overset{\FP}{=}&-\int_0^T\io(\wt)^2\left((\F_k\cdot\p)^2(\F_l\cdot\p)\FP h\right)\cdot\left((\F_k\cdot\p)^2(\F_l\cdot\p)L_3^j(t)\right)\dy\dt\\
&-2\int_0^T\io\wt e''(h)\p_t h\left((\F_k\cdot\p)^2(\F_l\cdot\p)\FP h\right)\cdot\left((\F_k\cdot\p)^2(\F_l\cdot\p)\int_0^t L_3^j\right)\dy\dt \\
&+\io(\wt)^2\left((\F_k\cdot\p)^2(\F_l\cdot\p)\FP h\right)\cdot\left((\F_k\cdot\p)^2(\F_l\cdot\p)\int_0^T L_3^j\right)\dy+\cdots\\
\lesssim& P(\EE_\kk(T))\int_0^TP\left(\left\|(\wt)^{\frac32}(\F_k\cdot\p)^2(\F_l\cdot\p)\FP h\right\|_0,\left\|\left(\FP\psi,\FP\eta,\FP\ek,\F,v\right)\right\|_4\right)\dt,
\end{aligned}
\end{equation}where in the omitted terms the derivative falls on either $\F_j$ (and thus gives a lower order term) or the weight function (and thus gives more weight functions). 

Finally we analyze the second last line in \eqref{FFF1}. It suffices to control $\swt$-weighted $L^2$-norm of $L_{42}+(\F_k\cdot\p)^2(\F_l\cdot\p)(L_1+L_2)$.

The $L_{42}$-term is easier to control because at least one derivative falls on the weight function, and all the other terms only contain spatial derivative and thus do not require any weight function.
\begin{equation}\label{FFF6}
\begin{aligned}
&\left\|\swt L_{42}\right\|_0\\
=&\left\|\swt\left(\left[e'(h),(\F_k\cdot\p)^2(\F_l\cdot\p)\right]\p_t^2 h+\left[(\F_k\cdot\p)^2(\F_l\cdot\p),\lapak\right]h+\sum_{j=1}^3\left[(\F_k\cdot\p)^2(\F_l\cdot\p),\wt\right]\FP^2 h\right)\right\|_0\\
\lesssim& P\left(\left\|\swt\right\|_{L^{\infty}},\left\|\wt\p_t^2 h\right\|_0,\|h\|_4,\|\F\|_4\right).
\end{aligned}
\end{equation}

The term $(\F_k\cdot\p)^2(\F_l\cdot\p)(L_1+L_2)$ can be similarly controlled as in \eqref{FFF6}
\begin{equation}\label{FFF7}
\begin{aligned}
&\left\|\swt(\F_k\cdot\p)^2(\F_l\cdot\p)(L_1+L_2)\right\|_0\\
\leq&\left\|\swt(\F_k\cdot\p)^2(\F_l\cdot\p)\left(-\p_t\ak^{li}\p_l v_i-\pak^i\left(\FP\ek_l\right)\pak^l\left(\FP\ek_i\right)\right)\right\|_0\\
&+\left\|\swt(\F_k\cdot\p)^2(\F_l\cdot\p)\left(-e''(h)(\p_th)^2+\sum_{j=1}^3e''(h)\left(\FP h\right)^2\right)\right\|_0
\end{aligned}
\end{equation}

Notice that the leading order terms in $(\F_k\cdot\p)^2(\F_l\cdot\p)(L_1+L_2)$ are $$(\F_k\cdot\p)^2(\F_l\cdot\p)\left((\pak v)(\pak v)-(\pak\FP\ek)(\pak\FP\ek)-e''(h)(\p_t h)^2+e''(h)(\FP h)^2\right),$$ which are all of at most 4 derivatives with at most one time derivative. Therefore, no weight function is needed except the divergence of deformation tensor $\sum\limits_{j=1}^3\|\wt \FP h\|_3$. Luckily, $|e''(h)|\lesssim|e'(h)|^2$ has provided enough number of weight functions. Hence, we get
\begin{equation}\label{FFF8}
\left\|\swt(\F_k\cdot\p)^2(\F_l\cdot\p)(L_1+L_2)\right\|_0\lesssim P\left(\|\F\|_4,\|\eta\|_4,\|v\|_4,\|\FP\ek\|_4,\|\p_t h\|_3,\|\wt \FP h\|_3\right).
\end{equation}

Summarising \eqref{FFF0}-\eqref{FFF8}, we finally get the weighted energy estimates of $(\F_k\cdot\p)^2(\F_l\cdot\p)$-differentiated wave equation \eqref{hFFFwave} by taking summation on $k,l$,
\begin{equation}\label{hFFFenergy}
\begin{aligned}
&\sum_{k,l=1}^3\left(\left\|(\wt)^{\frac32}(\F_k\cdot\p)^2(\F_l\cdot\p)\p_t h\right\|_0^2+\left\|\wt (\F_k\cdot\p)^2(\F_l\cdot\p)h\right\|_1^2+\sum_{j=1}^3\left\|(\F_k\cdot\p)^2(\F_l\cdot\p)\FP h\right\|_0^2\right)\\
\lesssim&\PP_0+ P(\EE_\kk(T))\int_0^T P(\EE_\kk(t))\dt.
\end{aligned}
\end{equation}

\subsubsection{Mixed-tangentially-differentiated wave equations}

We still need to control the weighted energy estimates of $\sum\limits_{j=1}^3\p_t^2\FP$-differentiated and $\sum\limits_{j=1}^3\p_t\FP^2$-differentiated wave equations. The proof follows in the same way as the previously discussed cases because
\begin{enumerate}
\item \textbf{The energy estimates must be uniform in sound speed}: Compared with ``fully time-differentiated" case, there are more spatial derivatives. Therefore, less number of weight functions will be needed in the commutator estimates.
\item \textbf{No loss of derivative}: The number of derivatives in the source terms will be no greater than the previously discusses case and thus there is no loss of derivatives in the source terms. Also, since both $\FP$ and $\p_t$ are tangential derivatives on the boundary, there is no boundary intergal when integrating by parts.
\item \textbf{On the source term $L_3^j$}: The treatment of this term is the same as ``fully $(\F\cdot\p)$-differentiated" case: First integrating $\p_t$ by parts to eliminate one time derivative with the help of time integral as in \eqref{FFF5}, then integrating $\FP$ by parts to finish the control.
\end{enumerate}

We list the results here
\begin{equation}\label{htFenergy}
\begin{aligned}
&\sum_{k=1}^3\left(\left\|(\wt)^{\frac32}(\F_k\cdot\p)\p_t^3 h\right\|_0^2+\left\|\wt (\F_k\cdot\p)\p_t^2h\right\|_1^2+\sum_{j=1}^3\left\|(\F_k\cdot\p)\FP\p_t^2 h\right\|_0^2\right)\\
+&\sum_{k=1}^3\left(\left\|(\wt)^{\frac32}(\F_k\cdot\p)^2\p_t^2 h\right\|_0^2+\left\|\wt (\F_k\cdot\p)^2\p_th\right\|_1^2+\sum_{j=1}^3\left\|(\F_k\cdot\p)^2\FP \p_th\right\|_0^2\right)\\
\lesssim&\PP_0+ P(\EE_\kk(T))\int_0^T P(\EE_\kk(t))\dt.
\end{aligned}
\end{equation}

\subsection{A priori estimates of the nonlinear $\kk$-approximation system}

Now we are able to finalize the uniform-in-$\kk$ a priori estimate of the nonlinear approximation system \eqref{elastolkk}. We recall that \eqref{vFhl2} gives the $L^2$-estimates, \eqref{v4} and \eqref{F4} give the div-curl-tangential decomposition for the full spatial derivatives. Then \eqref{vbd} and \eqref{Fbd} reduce the boundary part to interior tangential estimates which are established by \eqref{tgenergy}. The inequality \eqref{curlavF} gives the common control of curl part and \eqref{divv0}, \eqref{divF0} reduce the divergence control to the elliptic estimates of $h$. With the help of Christodoulou-Lindblad \cite{christodoulou2000motion} elliptic estimates Lemma \ref{GLL}, we reduce the estimates of $h$ to the fully-tangential-differentiated derivatives in \eqref{remainh}. On the other hand, the estimates of time derivatives of $v$ and $\FP\eta$ are again reduced to full spatial derivatives and the estimates of $h$ in \eqref{remainvtttt}, \eqref{remainFtttt} and \eqref{remainvF}.

After the reductions above, it remains to control the weighted energy estimates of $h$ listed in \eqref{remainh}. This is finished by analyzing the tangentially-differentiated wave equation \eqref{hwave0} of $h$. Finally, we get the desired energy estimates (actually slightly better than our expectation because less weight functions are needed) as in \eqref{htttenergy}, \eqref{hFFFenergy} and \eqref{htFenergy}. We taking summation and thus get a nonlinear Gronwall-type inequality
\begin{equation}
\EE_\kk(T)\lesssim\delta^2\left(\|v\|_4^2+\sum_{j=1}^3\|\FP\eta\|_4^2\right)+\PP_0+P(\EE_\kk(T))\int_0^T P(\EE_\kk(t))\dt
\end{equation}under the a priori assumptions \eqref{taylor1}-\eqref{small1}.

By the argument in Tao \cite[Chapter 1.3]{tao2006nonlinear}, we can pick a suitably small $T>0$ such that the $\delta$-terms can be absorbed to LHS and get the following energy bound
\begin{equation}\label{EEkk}
\sup_{0\leq t\leq T}\EE_\kk(t)\leq P\left(\|v_0\|_4,\|\F\|_4,\|\h_0\|_4\right).
\end{equation}

\subsection{Justification of the a priori assumptions}\label{justifyn}

Finally, we need to justify the a priori assumptions \eqref{taylor1} and \eqref{small1}. The Rayleigh-Taylor sign condition can be directly justified by using Morrey's embedding. Notice that $\frac{\p h}{\p N}\in L^{\infty}([0,T];H^{\frac52}(\Gamma)),~\p_t\left(\frac{\p h}{\p N}\right)\in L^{\infty}([0,T];H^{\frac32}(\Gamma))$ and the boundary is $\T^2$ a bounded open set in 2D. By using Morrey's embedding, the taylor sign is actually a H\"older continuous function in both $t$ and $y$ variables:
$$\frac{\p h}{\p N}\in W^{1,\infty}([0,T];H^{\frac32}(\Gamma))\hookrightarrow W^{1,\infty}([0,T];W^{1,4}(\Gamma))\hookrightarrow W^{1,4}([0,T]\times\Gamma)\hookrightarrow C_{t,x}^{0,\frac14}([0,T]\times\Gamma).$$ Therefore, the a priori assumption holds is a positive time interval provided that $-\frac{\p \h_0}{\p N}\geq c_0>0$ initially. The second assumption \eqref{small1} is easily justified by $$\text{Id}-\ak=\int_0^T \p_t\ak \dt=-\int_0^T \ak:\p\ek:\ak\dt,$$ of which the $H^3$-norm is directly bounded by our energy functional. Similar estimate holds for $\Jk-1$ since $\Jk$ is a bilinear function of $\p\ek$ and $\p\ek-$Id is sufficiently small due to that of $\ak-$Id.

\section{Construction of solutions to the approximation system}\label{linearize}

In Section \ref{apriorinkk}, we have derived the uniform-in-$\kk$ a priori estimates for the nonlinear $\kk$-approximation system \eqref{elastolkk}. Now we are going to construct the strong solution to \eqref{elastolkk} for each fixed $\kk>0$. Given the a priori estimates, it is natural to consider linearization and Picard iteration to construct the solution to \eqref{elastolkk}. Specifically, we expect to start from the trivial solution $(\eta^{(0)},v^{(0)},h^{(0)})=(\eta^{(1)},v^{(1)},h^{(1)})=(\text{Id},0,0)$ and inductively define $(\eta^{(n+1)},v^{(n+1)},h^{(n+1)})$ provided that $\{(\eta^{(k)},v^{(k)},h^{(k)})\}_{0\leq k\leq n}$ are given
\begin{equation}\label{elastolln}
\begin{cases}
\p_t\eta^{(n+1)}=v^{(n+1)}+\psi^{(n)}~~~&\text{in }\Omega,\\
\p_t v^{(n+1)}=-\nabla_{\ak^{(n)}}h^{(n+1)}+\sum\limits_{j=1}^3\FP^2\eta^{(n+1)}~~~&\text{in }\Omega,\\
\dive_{\ak^{(n)}}v^{(n+1)}=-e'(h^{(n)})\p_t h^{(n+1)}~~~&\text{in }\Omega,\\
\dive \F_j:=\p_k\F_{kj}=-e'(\h_0)\FP\h_0~~~&\text{in }\Omega,\\
h^{(n+1)}=0,~\F_j\cdot N=0~~~&\text{on }\Gamma,\\
(\eta^{(n+1)},v^{(n+1)},h^{(n+1)})|_{t=0}=(\text{Id},v_0,\h_0),
\end{cases}
\end{equation}where $\ek^{(n)}$ is the smoothed version of $\eta^{(n)}$ and $\ak^{(n)}:=[\p\ek^{(n)}]^{-1}$. For simplicity on notation, we denote $(\er,\ar,\vr,\hr)=(\eta^{(n)},a^{(n)},v^{(n)},h^{(n)})$ and $(\eta,v,h):=(\eta^{(n+1)},v^{(n+1)},h^{(n+1)})$. Then the linearized system becomes
\begin{equation}\label{elastollr}
\begin{cases}
\p_t\eta=v+\psir~~~&\text{in }\Omega,\\
\p_t v=-\park h+\sum\limits_{j=1}^3\FP^2\eta~~~&\text{in }\Omega,\\
\divr v=-e'(\hr)\p_t h~~~&\text{in }\Omega,\\
\dive \F_j:=\p_k\F_{kj}=-e'(\h_0)\FP\h_0~~~&\text{in }\Omega,\\
h=0,~\F_j\cdot N=0~~~&\text{on }\Gamma,\\
(\eta,v,h)|_{t=0}=(\text{Id},v_0,\h_0).
\end{cases}
\end{equation}

What we need to do are
\begin{enumerate}
\item Construct the unique strong solution to the linearized $\kk$-approximation system \eqref{elastollr} for each fixed $\kk>0$.
\item Derive the uniform-in-$n$ estimates for the linearized $\kk$-apporximation system \eqref{elastollr} for each fixed $\kk>0$.
\item Picard iteration: Prove the sequence $\{(\eta^{(n)},v^{(n)},h^{(n)})\}_{n\in\N}$ strongly converges (subsequentially) in suitable Sobolev spaces to derive the unique strong solution to the nonlinear $\kk$-approximation system for each fixed $\kk>0$.
\end{enumerate}

We will prove the following conclusions by verifying the steps above

\begin{prop}[\textbf{Local well-posedness of the linearized approximation system \eqref{elastollr}}]\label{lllwp}
Fix $\kk>0$. There exists a time $T_\kk:=T_{\kk}(\|\p\Ark\|_{L^{\infty}},\|\p_t h\|_{L^{\infty}},\|\F\|_3,\|\FP\psir\|_1)>0$ such that the linearized $\kk$-approximation system \eqref{elastollr} has a unique strong solution $(\eta,v,h)$ in $[0,T_\kk]$ satisfying the $H^1$-estimate
\begin{equation}
\sup_{0\leq t\leq T_\kk}\|v\|_1^2+\|h\|_1^2+\sum_{j=1}^3\|\FP\eta\|_1^2\leq 2(\|v_0\|_1^2+\|\F\|_1^2+\|\h_0\|_1^2).
\end{equation}
\end{prop}
\begin{flushright}
$\square$
\end{flushright}

\begin{prop}[\textbf{Uniform-in-$n$ estimates of the linearized approximation system \eqref{elastollr}}]\label{llenergy}
Fix $\kk>0$. There exists a time $T_\kk>0$, such that the unique strong solution $(\eta,v,h)$ to the linearized $\kk$-approximation system \eqref{elastollr} satisfies the following estimates
\begin{equation}
\sup_{0\leq t\leq T_\kk}\EEr_\kk(t)\leq P\left(\|v_0\|_4,\|\F\|_4,\|\h_0\|_4\right),
\end{equation}where
\begin{equation}\label{energykkll}
\begin{aligned}
\EEr_{\kk}(T)&:=\left\|\eta\right\|_4^2+\left\|v\right\|_4^2+\sum\limits_{j=1}^3\left\|\FP\eta\right\|_4^2+\left\|h\right\|_4^2\\
&+\left\|\p_t v\right\|_3^2+\sum\limits_{j=1}^3\left\|\FP\p_t\eta\right\|_3^2+\left\|\p_t h\right\|_3^2\\
&+\left\|\p_t^2v\right\|_2^2+\sum\limits_{j=1}^3\left\|\FP\p_t^2\eta\right\|_2^2+\left\|(e'(\hr))^{\frac12}\p_t^2 h\right\|_2^2\\
&+\left\|(e'(\hr))^{\frac12}\p_t^3 v\right\|_1^2+\sum\limits_{j=1}^3\left\|(e'(\hr))^{\frac12}\FP\p_t^3 \eta\right\|_1^2+\left\|e'(\hr)\p_t^3 h\right\|_1^2\\
&+\left\|e'(\hr)\p_t^4 v\right\|_0^2+\sum\limits_{j=1}^3\left\|e'(\hr)\FP\p_t^4 \eta\right\|_0^2+\left\|(e'(\hr))^{\frac32}\p_t^4 h\right\|_0^2.
\end{aligned}
\end{equation}
\end{prop}
\begin{flushright}
$\square$
\end{flushright}

\begin{prop}[\textbf{Local well-posedness of the nonlinear $\kk$-approximation system \eqref{elastolkk}}]\label{nnlwp}
Fix $\kk>0$. There exists a time $T_\kk>0$ such that the nonlinear $\kk$-approximation system \eqref{elastolkk} has a unique strong solution $(\eta(\kk),v(\kk),h(\kk))$ in $[0,T_\kk]$ satisfying the following estimates
\begin{equation}
\sup_{0\leq t\leq T_\kk}\EE_\kk(t)\leq P\left(\|v_0\|_4,\|\F\|_4,\|\h_0\|_4\right).
\end{equation}
\end{prop}
\begin{flushright}
$\square$
\end{flushright}

\subsection{Failure of fixed-point argument to solve linearized system}\label{fail}

Before going to the proof, we would like to point out the specific reasons for the failure of fixed-point argument. In elastodynamics (resp. MHD), the presence of Cauchy-Green tensor $\sum\limits_{j=1}^3\FP^2\eta$ (resp. Lorentz force) makes the second equation of \eqref{elastollr} lose one derivative, which tells the essential difference from Euler equations. We are able to avoid such problem in the nonlinear a priori estimates because one can integrate $\FP$ by parts in curl and tangential estimates, and re-produce the divergence by $\diva \FP\eta\approx -e'(h)\FP h+\int_0^T$(controllable terms)$\dt$. But in the fixed-point argument (for both Banach fixed-point theorem and Tikhonov fixed-point theorem) to solve the linearized equation, we have to estimate $\|v\|_4$ by $\|v\|_4=\|v_0\|_4+\int_0^T\|\p_t v(t)\|_4\dt$ and thus the $H^4$-norms of $\sum\limits_{j=1}^3\FP^2\eta$ and $\park h$ are necessary to be controlled. In Gu-Wang \cite{gu2016construction} and Zhang \cite{ZhangYH} for incompressible MHD and elastodynamics, they introduced a directional viscosity term in the equation of flow map as follows which does compensate the derivative loss caused by the Cauchy-Green tensor. 
\begin{equation}\label{etamueqeq}
\p_t\eta-\mu\sum\limits_{j=1}^3\FP^2\eta=v+\psir
\end{equation} 

However, one also has to derive the estimate for passing the vanishing viscosity limit $\mu\to 0_+$ to \eqref{elastollr}. When estimating the $H^3$-norm of $\divr(\FP\eta)$, one needs to take $\p^3\diva\FP$ in \eqref{etamueqeq} and thus the RHS yields $\p^3\divr(\FP v)$ which loses one derivative. If the fluid is incompressible, then one can directly commute $\divr$ with $\FP$ to eliminate such higher order term. (See Gu-Wang \cite[(5.38)]{gu2016construction} and Zhang \cite[(3.29)]{ZhangYH}.)  In the compressible case, we have to find other ways to control the divergence. Recall that our previous a priori estimate for the nonlinear system \eqref{elastolkk} strongly depends on the fact that the difference between $\diva (\FP\eta)$ and $-e'(h)\FP h$ is controllable. But now the divergence of deformation tensor becomes
\[
\divr\FP\eta=-e'(\hr)\FP h+\int_0^T\text{controllable terms}+\mu\int_0^T\sum_{k=1}^3\divr\FP(\F_k\cdot\p)^2\eta,
\]where we have to control the $H^3$-norm of $\mu\divr\FP(\F_k\cdot\p)^2\eta$. Unfortunately, the directional viscosity only gives us the $L_t^2H_x^4$ norm of $\mu(\F_k\cdot\p)^2\eta$ and thus the expected ``error term" cannot be controlled. Therefore, the ``vanishing directional viscosity method" in \cite{gu2016construction,ZhangYH} is obviously not applicable to compressible elastodynamics, which leads to the failure of fixed-point argument.

\subsection{Hyperbolic approach to solve the linearized approximation system}\label{hyperbolic}

First we need the following bounds for the coefficients $\ar,\er,\Jr$ provided that they hold for $\{(\eta^{(k)},v^{(k)},h^{(k)})\}_{0\leq k\leq n-1}$. 
\begin{lem}\label{etapsir}
For fixed $\kk>0$, there exists $0<\delta\ll 1$ and $A>0$ such that $\forall T\in(0,T_\kk)$
\begin{align}
\psir,\p_t\psir,\FP\psir,\p_t\er,\FP\er\in& L^{\infty}([0,T];H^4(\Omega)),\\
\p_t^2\psir,\p_t\FP\psir,\p_t^2\er,\p_t\FP\er,\p_t\hr\in& L^{\infty}([0,T];H^3(\Omega)),\\
\sqrt{e'(\hr)}\p_t^3\psir,\p_t^2\FP\psi,\p_t^3\er,\p_t^2\FP\er,\sqrt{e'(\hr)}\p_t^2\hr\in& L^{\infty}([0,T];H^2(\Omega)),\\
e'(\hr)\p_t^4\psir,\sqrt{e'(\hr)}\p_t^3\FP\psir,\sqrt{e'(\hr)}\p_t^4\er,\sqrt{e'(\hr)}\p_t^3\FP\er,e'(\hr)\p_t^3\hr\in& L^{\infty}([0,T];H^1(\Omega)),\\
e'(\hr)\p_t^4\FP\psir,e'(\hr)\p_t^5\er,e'(\hr)\p_t^4\FP\er,(\wt)^{\frac32}\p_t^4\hr\in& L^{\infty}([0,T];L^2(\Omega)).\\
\|\Jr-1\|_3+\|\Jrk-1\|_3+\|\text{Id}-\ar\|_3+\|\text{Id}-\ark\|_3\leq&\delta.
\end{align}
\end{lem}
\begin{proof}
The proof follows in the same way as Lemma \ref{etapsi} and the justification of \eqref{small1} in Section \ref{justifyn}.
\end{proof}

Now let us rewrite the system \eqref{elastollr} to be the following system by denoting $\FFr^i_j=\FP\eta^i$ and $\FFr_j=(\FP\eta^1,\FP\eta^2,\FP\eta^3)$, $\FFr^i=\left((\F_1\cdot\p)\eta_i,(\F_2\cdot\p)\eta_i,(\F_3\cdot\p)\eta_i\right)$

\begin{equation}\label{elastolr}
\begin{cases}
\p_t\FFr_{j}=\FP v+\psir~~~&\text{in }\Omega,\\
\p_t v+\park h-\sum\limits_{j=1}^3\FP\FFr_j=0~~~&\text{in }\Omega,\\
e'(\hr)\p_t h+\divr v=0~~~&\text{in }\Omega,\\
\dive \F_j:=\p_k\F_{kj}=-e'(\h_0)\FP\h_0~~~&\text{in }\Omega,\\
h=0,~\F_j\cdot N=0~~~&\text{on }\Gamma,\\
(\eta,v,h)|_{t=0}=(\text{Id},v_0,\h_0).
\end{cases}
\end{equation} Note that once $\FFr,v,h$ are solved, the flow map $\eta$ is automatically solved by $\p_t\eta=v+\psir$. Therefore it remains to solve \eqref{elastolr}. We write \eqref{elastolr} into a first-order symmetric hyperbolic system of the variable $\XX\in\R^{13}$ defined by
\begin{equation}\label{XX}
\XX:=\left[h,v,\FFr^1,\FFr^2,\FFr^3\right]^{\top}
\end{equation}
\begin{equation}\label{XXeq}
A_0(t,y)\p_t \XX+\sum_{l=1}^3A_l(t,y)\p_l \XX =f.
\end{equation}Here
{\small{\begin{equation}\label{XXf}
f(t,y)=\left[\mathbf{0}_{1\times 4},(\F_1\cdot\p)\psir_1,(\F_2\cdot\p)\psir_1,(\F_3\cdot\p)\psir_1,(\F_1\cdot\p)\psir_2,(\F_2\cdot\p)\psir_2,(\F_3\cdot\p)\psir_2,(\F_1\cdot\p)\psir_3,(\F_2\cdot\p)\psir_3,(\F_3\cdot\p)\psir_3\right]^{\top},
\end{equation} 
}}
\begin{equation}\label{A0}
A_0(t,y)=\text{diag}\left[e'(\hr),1,\cdots,1\right],
\end{equation} 
and $A_l(t,y)$ equals to the following matrix 
\begin{equation}\label{Al}
\renewcommand\arraystretch{2}
\begin{bmatrix}
0&\ark^{l1}&\ark^{l2}&\ark^{l3}&0&0&0&0&0&0&0&0&0 \\
\ark^{l1}& & & &-\F_{l1}&-\F_{l2}&-\F_{l3}&0&0&0&0&0&0\\
\ark^{l2}& & \mathbf{O}_{3\times3}&&0&0&0&-\F_{l1}&-\F_{l2}&-\F_{l3}&0&0&0\\
\ark^{l3}& & & &0&0&0&0&0&0&-\F_{l1}&-\F_{l2}&-\F_{l3}\\
0&-\F_{l1}&0&0&&&&&&&&&&\\
0&-\F_{l2}&0&0&&&&&&&&&&\\
0&-\F_{l3}&0&0&&&&&&&&&&\\
0&0&-\F_{l1}&0&&&&&&&&&&\\
0&0&-\F_{l2}&0&&&&&&\mathbf{O}_{9\times 9}&&&&\\
0&0&-\F_{l3}&0&&&&&&&&&&\\
0&0&0&-\F_{l1}&&&&&&&&&&\\
0&0&0&-\F_{l2}&&&&&&&&&&\\
0&0&0&-\F_{l3}&&&&&&&&&&\\
\end{bmatrix}.
\end{equation}

Then we find that, the normal projection of the coefficient matrix 
\[ A_{N}:=\sum_{l=1}^3 A_l N_l=\pm \begin{bmatrix}
0&\ark^{l1}&\ark^{l2}&\ark^{l3}&&&\\
\ark^{l1}& & & &&&\\
\ark^{l2}& & \mathbf{O}_{3\times3}&&&&\\
\ark^{l3}& & & &&&\\
&&&&&&&\\
&&&&&\mathbf{O}_{9\times 9}&\\
&&&&&&&\\
\end{bmatrix}\text{ similar to }\pm
\begin{bmatrix}
0&0&0&1&&&\\
0& & & &&&\\
0& & \mathbf{O}_{3\times3}&&&&\\
1& & & &&&\\
&&&&&&&\\
&&&&&\mathbf{O}_{9\times 9}&\\
&&&&&&&\\
\end{bmatrix}
\]which is of constant rank and singular on the boundary. So the linearized system \eqref{elastolr} is indeed of characteristic boundary conditions. By the argument in Lax-Phillips \cite{lax60}, we need to prove the following things in order for the existence of $L^2$-solution to \eqref{elastolr}, 
\begin{enumerate}
\item Derive the $L^2$-a priori bound for \eqref{XXeq} without loss of regularity.
\item Find the dual problem of \eqref{XXeq}.
\item Derive the $L^2$-a priori bound for the dual problem without loss of regularity.
\end{enumerate}

The $L^2$-a priori bound for \eqref{XXeq} is quite straightforward. We take $L^2$ inner product of $v$ and the second equation of \eqref{elastolr} to get
\begin{equation}
\frac{1}{2}\frac{d}{dt}\io|v|^2\dy+\io\left(\park h\right)\cdot v\dy-\sum_{l=1}^3(\F_l\cdot\p)\FFr^i\cdot v_i\dy=0
\end{equation}

Integrating by parts in the second and third term, and invoking the boundary conditions $h=0,\F\cdot N=\mathbf{0}$, we know
\begin{equation}
\io\left(\park h\right)\cdot v\dy=\frac12\frac{d}{dt}\io e'(\hr)|h|^2\dy-\frac12\io e''(\hr)\p_t\hr|h|^2\dy-\io\p_l\ark^{li} v_i\cdot h\dy,
\end{equation} and
\begin{equation}
\begin{aligned}
&-\sum_{l=1}^3(\F_l\cdot\p)\FFr^i\cdot v_i\dy=\sum_{l=1}^3\io\FFr^i\cdot(\F_l\cdot\p)v_i\dy+\io\left(\dive \F_l\right)\FFr\cdot v\dy\\
=&\sum_{l=1}^3\frac12\frac{d}{dt}\io\left|(\F_l\cdot\p)\eta\right|^2\dy-\sum_{l=1}^3\io\FFr^i\cdot(\F_l\cdot\p)\psir_i\dy+\io\left(\dive \F_l\right)\FFr\cdot v\dy.
\end{aligned}
\end{equation}

Therefore, we have
\begin{equation}
\begin{aligned}
\frac12\frac{d}{dt}\left(\left\|v\right\|_0^2+\left\|\sqrt{e'(\hr)}h\right\|_0^2+\sum_{l=1}^3\left\|\FFr_l\right\|_0^2\right)\lesssim \|\p\F\|_{L^{\infty}}\|\FFr\|_0\|v\|_0+\sum_{l=1}^3\|\FFr\|_0\left\|(\F_l\cdot\p)\psi\right\|_0+\frac12\|e''(\hr)\p_t\hr\|_{L^{\infty}}\|h\|_0^2+\|\p\ark\|_{L^{\infty}}\|v\|_0\|h\|_0,
\end{aligned}
\end{equation}and thus by Gronwall inequality we are able to get the $L^2$-a priori bound
\begin{equation}\label{XXL2}
\left\|v\right\|_0^2+\left\|\sqrt{e'(\hr)}h\right\|_0^2+\sum_{l=1}^3\left\|\FFr_l\right\|_0^2\bigg|^T_0\lesssim C\left(\|\F\|_3,\|e''(\hr)\p_t\hr\|_{L^{\infty}},\|\p\ark\|_{L^{\infty}}\right)\|f\|_{0}.
\end{equation}

Next we derive the dual problem of \eqref{XXeq}. We introduce the ``test function" $\YY\in\R^{13}$ as the variable of the dual problem by 
\begin{equation}
\YY:=\left[\theta,w,\GGr^1,\GGr^2,\GGr^3\right]^{\top},
\end{equation}where the definition of $\GGr^i$ is the similar as $\FFr^i$.

Testing \eqref{XXeq} with $\YY$ under space-time intergal, we can get the following first-order hyperbolic system 
\begin{equation}\label{YYeq}
A_0(t,y)\p_t\YY+\sum_{l=1}^3A_l(t,y)\p_l\YY=f^*,
\end{equation}
where $f^*:=\left[f^*_1,\cdots,f^*_{13}\right]^{\top}\in\R^{13}$ is defined by
\begin{align*}
f^*_1=&-\p_l\ark^{li}w_i\\
[f^*_2,f^*_3,f^*_4]=&\left[\sum_{l=1}^3\dive \F_l\cdot \GGr_1^l-\p_l\ark^{l1}\theta,\sum_{l=1}^3\dive \F_l\cdot \GGr_2^l-\p_l\ark^{l2}\theta,\sum_{l=1}^3\dive \F_l\cdot \GGr_3^l-\p_l\ark^{l3}\theta\right]\\
[f^*_5,f^*_6,f^*_7]=&\left[-\dive\F_1 w^1+(\F_1\cdot\p\psi^1),-\dive\F_2 w^1+(\F_2\cdot\p\psi^1),-\dive\F_3 w^1+(\F_3\cdot\p\psi^1)\right]\\
[f^*_8,f^*_9,f^*_{10}]=&\left[-\dive\F_1 w^2+(\F_1\cdot\p\psi^2),-\dive\F_2 w^2+(\F_2\cdot\p\psi^2),-\dive\F_3 w^2+(\F_3\cdot\p\psi^2)\right]\\
[f^*_{11},f^*_{12},f^*_{13}]=&\left[-\dive\F_1 w^3+(\F_1\cdot\p\psi^3),-\dive\F_2 w^3+(\F_2\cdot\p\psi^3),-\dive\F_3 w^3+(\F_3\cdot\p\psi^3)\right].
\end{align*}

Compared with the expression of $f$ in \eqref{XXf}, the extra terms are all cause by integrating $\FP$ by parts and $\dive\F\neq 0$ and are all of the form $\dive \F\cdot w,~\p\ark\cdot w,~\p\ark\cdot\theta$ or $\dive \F\cdot \GGr$. Therefore, we can exactly mimic the proof for \eqref{XXeq} to get the $L^2$-a priori bound with no loss of regularity of \eqref{YYeq} as
\begin{equation}\label{YYL2}
\left\|w\right\|_0^2+\left\|\sqrt{e'(\hr)}\theta\right\|_0^2+\sum_{l=1}^3\left\|\GGr_l\right\|_0^2\lesssim C\left(\|\F\|_3,\|e''(\hr)\p_t\hr\|_{L^{\infty}}\right)\|f^*\|_{0}.
\end{equation}

Combining \eqref{XXL2} and \eqref{YYL2}, we know the system \eqref{elastolr} and thus \eqref{elastollr} has a $L^2$-solution. We are able to verify it is the unique strong solution by $H^1$-estimates. Here we only show the tangential control, while the normal derivative can be estimated by div-curl decomposition which is exactly the same as what will be done in Section \ref{apriorill}.

The tangential part of $H^1$-estimate does not need Alinhac good unknown method because the coefficient $\Ark$ is now $C^{\infty}$ and $\kk>0$ is fixed. Taking $\TP$ in the second equation of \eqref{elastollr}, we get
\begin{equation}\label{tgll1}
\frac{1}{2}\frac{d}{dt}\io\left|\TP v\right|^2\dy=-\io \TP v_i\TP\left(\ark^{li}\p_l h\right)\dy+\sum_{j=1}^3\io\TP\left(\FP^2\eta\right)\cdot \TP v\dy
\end{equation}

Integrating $\FP$ by parts in the second term yields that
\begin{equation}\label{tgll2}
\begin{aligned}
&\io\TP\left(\FP^2\eta\right)\cdot \TP v\dy\\
=&-\io\TP\left(\FP\eta\right)\cdot \TP\left(\FP\p_t\eta\right)\dy+\io\TP\left(\FP\eta\right)\cdot \TP\left(\FP\psir\right)\dy+\io(\dive \F_j)\TP\left(\FP\eta\right)\cdot \TP v\dy\\
&+\io\left[\TP,\FP\right]\FP\eta\cdot\TP v\dy+\io\FP\eta\cdot\left[\TP,\FP\right]\TP v\dy\\
\lesssim&-\frac12\frac{d}{dt}\left\|\TP\left(\FP\eta\right)\right\|_0^2+P\left(\|\FP\eta\|_1,\|\FP\psir\|_1,\|\p\F\|_{L^{\infty}},\|v\|_1\right)
\end{aligned}
\end{equation}

Integrating by parts in the first term and invoking Piola's identity gives that
\begin{equation}\label{tgll3}
\begin{aligned}
&-\io \TP v_i\TP\left(\ark^{li}\p_l h\right)\dy\\
=&\io(\divr v)\TP h\dy-\ig\TP v_i\ark^{li}N_l\underbrace{\TP h}_{=0} \dS\\
&-\io\TP v_i\cdot\left[\TP,\ark^{li}\right]\p_l h\dy-\io\left[\TP,\ark^{li}\right]\p_l v_i\cdot \TP h\dy+\io\TP v_i\p_l\ark^{li}\TP h\dy\\
\lesssim&-\frac12\frac{d}{dt}\left\|\sqrt{e'(\hr)}\TP h\right\|_0^2+P\left(\|\TP\Ark\|_{L^{\infty}},\|v\|_1,\|h\|_1,\|\p_t \hr\|_{L^{\infty}}\right)
\end{aligned}
\end{equation}

Combining \eqref{tgll1}-\eqref{tgll3}, we can get the $H^1$-tangential estimates of \eqref{elastollr}. The div-curl estimate follows in the same way as in the Section \ref{apriorill} so we omit the proof here. Then $\|h\|_1$ can be estimated by the linearized wave equation derived by taking divergence in the second equation of \eqref{elastollr}. Such step is a direct consequence of $L^2$-estimate of \eqref{hwavel0} so we also skip the proof. Therefore we get the $H^1$-estimates without loss of regularity for the linearized $\kk$-approximation problem \eqref{elastollr}, which demonstrates the solution we constructed above is a strong solution and also unique. Proposition \ref{lllwp} is proven.
\subsection{Uniform estimates of the linearized approximation system}\label{apriorill}

Now we inductively prove the uniform-in-$n$ estimates for the linearized $\kk$-approximation system \eqref{elastollr}. WSuppose we have already have the energy estimates for $0\leq k\leq n$ and thus Lemma \ref{etapsir} holds true for the coefficients of $(n+1)$-th linearized system.

\subsubsection{Estimates of velocity}

The estimates of velocity is still based on the div-curl-tangential estimates and elliptic estimates of $h$. First we have
\begin{equation}\label{vlhodge}
\|v\|_4\lesssim\|v\|_0+\|\dive v\|_3+\|\curl v\|_3+\left|\TP v\cdot N\right|_{5/2}.
\end{equation}

\paragraph*{Curl estimates:}

The curl estimate follows the same way as in Section \ref{divcurl}. Taking $\curlr$ in the second equation of \eqref{elastolr} we get
\begin{equation}\label{curll0}
\p_t\left(\curlr v\right)-\sum\limits_{j=1}^3\FP\left(\curlr(\FP\eta)\right)=\curl_{\ark_t} v+\sum\limits_{j=1}^3\left[\curlr,\FP\right]\FP\eta.
\end{equation} Then we take $\p^3$, multiply $\p^3\curlr v$, integrate by parts and invoke $\p_t\eta=v+\psir$ to get
\begin{equation}\label{curlrvF}
\begin{aligned}
\left\|\curl v(T)\right\|_3^2+\sum_{j=1}^3\left\|\curl \FP\eta(T)\right\|_3^2\lesssim\delta^2\left(\|v(T)\|_4^2+\sum_{j=1}^3\left\|\FP\eta\right\|_4^2\right)+\PP_0+\int_0^T P(\EEr_{\kk}(t))\dt.
\end{aligned}
\end{equation}

\paragraph*{Tangential estimates:}

The boundary term in \eqref{vlhodge} is again reduced to $\|\TP^4 v\|_0$. We still use the Alinhac good unknown method by introducing
\[
\VVr:=\tpl v-\tpl\erk\cdot\park v,~~\HHr:=\tpl h-\tpl\erk\cdot\park h.
\]
Applying $\tpl$ to the second equation in the linearization system \eqref{elastolr}, one gets
\begin{equation}
\label{goodlinearl} \p_t\VVr-\sum_{j=1}^3\FP\left(\tpl\FP\eta\right)+\park \HHr=\p_t(\tpl \erk\cdot\park v)-\mathring{C}(h)+\sum\limits_{j=1}^{3}\left[\tpl,\FP\right](\FP\eta),
\end{equation}subject to the boundary condition
\begin{equation}\label{bdryrgood}
\HHr=-\tpl \erk_{\beta}\ark^{3\beta}\p_3 h~~~\text{ on }\Gamma,
\end{equation}and the corresponding compressibility condition
\begin{equation}\label{divrgood}
\park\cdot\VVr=\tpl(\divr v)-\mathring{C}^{\alpha}(v_{\alpha}),~~~\text{in }\Omega.
\end{equation}
Here for any function $f$, the comuutator $\mathring{C}(f)$ is defined in the same way as before but replacing $\ak$ by $\ark$:
\begin{equation}\label{alinhacar}
\tpl(\park f)=\park \mathring{\mathbf{f}}+\mathring{C}(f),
\end{equation} with 
\begin{equation}\label{alihanccr}
\|\mathring{C}(f)\|_0\lesssim P(\|\er\|_4)\|\p f\|_3.
\end{equation} Here $\mathring{\mathbf{f}}$ is the Alinhac good unknown for $f$.

Similarly we have
\begin{equation}\label{goodtpll}
\|\TP^4 f(t)\|_0\lesssim\|\mathring{\mathbf{f}}\|_0+P(\|f(0)\|_3)+P(\|\er\|_4)\int_0^t P(\|\p_t f(\tau)\|_3)~d\tau.
\end{equation}

Now we take $L^2$ inner product between \eqref{goodlinearl} and $\VVr$ to get the followin analogous estimates
\begin{equation}\label{tgs1l}
\frac{1}{2}\frac{d}{dt}\io |\VVr|^2\dy+\sum_{j=1}^3\io\tpl\left(\FP\eta\right)\cdot\FP\VVr=-\io \park \HHr\cdot \VVr\dy+\io(\text{RHS of }\eqref{goodlinearl})\cdot\VVr\dy,
\end{equation}where ``RHS" can be directly controlled as in the nonlinear counterpart. As for the first term, we integrate by parts to get 
\begin{equation}\label{tgs2l}
-\io \park \HHr\cdot \VVr\dy=-\ig \ark^{3i}\VVr_{i} \HHr\dS+\io \HHr(\park\cdot\VVr)\dy+\io\p_l\ark^{li}\HHr\VVr_{i}\dy,
\end{equation}where the second and the third term can be controlled in the same way as in Section \ref{tgspace}.

For the boundary term in \eqref{tgs2l}, we no longer need to plug the precise form of $\psir$ into it and find the subtle cancelltaion as in Section \ref{tgspace} because the energy estimate is not required to be $\kk$-independent. Instead, we integrate $\TP^{1/2}$ by parts then apply Kato-Ponce inequality \eqref{product} and Sobolev embedding ${H}^{0.5}(\T^2)\hookrightarrow L^4(\T^2)$ to get
\begin{equation}\label{Bl}
\begin{aligned}
-\ig \ark^{3i}\VVr_{i} \HHr\dS&=\io \left(\frac{\p h}{\p N}\right)\tpl\left(\lkk^2\er_{k}\right)\ark^{3k}\ark^{3i}\VVr_{i}\dS \\
&\lesssim\left(\left|\p_3 h\ark^{3k}\ark^{3i}\right|_{L^{\infty}}\left|\tpl(\lkk^2\er_{k})\right|_{0.5}+\left|\p_3 h\ark^{3k}\ark^{3i}\right|_{W^{\frac12,4}}\left|\tpl(\lkk^2\er_{k})\right|_{L^4}\right)|\VVr|_{\dot{H}^{-\frac12}}\dS \\
&\lesssim \frac{1}{\kk} P\left(\|h\|_{4},\|v\|_4,\|\er\|_{4}\right) .
\end{aligned}
\end{equation}

Another difference is that our flow map now contains a directional viscosity term, and thus the second term on LHS or \eqref{tgs1l} should be treated differently. We plug the expression of Alinhac good unknown $\VVr$ and $v=\p_t\eta-\psir$ into it to get
\begin{align*}
&\io\tpl\left(\FP\eta\right)\cdot\FP\VVr\dy\\
=&\io\tpl\left(\FP\eta\right)\cdot\FP\left(\tpl v-\tpl\erk\cdot\park v\right)\dy\\
=&-\frac12\frac{d}{dt}\io\left|\tpl\left(\FP\eta\right)\right|^2\dy\\
&-\io\tpl\left(\FP\eta\right)\cdot\left(\tpl\left(\FP\psir\right)-\left[\tpl,\FP\right]v-\FP\left(\tpl\erk\cdot\park v\right)\right).
\end{align*}

Direct computation shows that
\begin{equation}\label{tgFl}
\begin{aligned}
\sum_{j=1}^3\left\|\tpl\left(\FP\eta\right)(t)\right\|_0^2\bigg|^T_0\lesssim\sum_{j=1}^3 P\left(\|\F,\eta,\erk,v\|_4,\|\FP\eta,\FP\erk,\FP\psir\|_4\right).
\end{aligned}
\end{equation}

Summing up \eqref{goodtpll}-\eqref{tgFl}, we get the tangential estimates as follows
\begin{equation}\label{tglinear}
\left\|\TP^4 v\right\|_0^2+\sum_{j=1}^3\left\|\TP^4 \left(\FP\eta\right)\right\|_0^2+\left\|\sqrt{e'(\hr)}\TP^4 h\right\|_0^2
\lesssim\PP_0+\int_0^T P(\EEr_\kk(t))\dt.
\end{equation}

\paragraph*{Divergence estimates:}

The divergence estimate still follows the similar way as in Section \ref{apriorinkk}. We know that $\divr v=-e'(\hr)\p_t h$ and thus 
\begin{equation}\label{divlv}
\|\dive v\|_3^2\lesssim\delta^2\|v\|_4^2+\|e'(\hr)\p_t h\|_3^2
\end{equation}

The formula of $\divr \FP\eta$ can be computed in the same way as \eqref{divkkF}
\begin{equation}\label{divllF}
\begin{aligned}
\divr\FP\eta(T)=&-\dive \F_j+\int_0^T\dive_{\ark_t}\FP\eta\dt+\int_0^T\divr\FP(v+\psir)\dt\\
=&-\dive \F_j+\int_0^T\FP e'(\hr)\p_t h\dt\\
&+\int_0^T\divr\FP\psir\underbrace{-\ark^{lr}\p_m\p_t\erk_r\ark^{mi}}_{=\ark_t^{li}}\p_l(\FP \eta_i)+\ark^{li}\p_l(\FP v_i)\dt\\
=&-e'(\hr)\FP h(T)\\
&+\int_0^T\underbrace{\divr(\FP\psir)-\left(\park^r(\FP\eta_i)\right)\left(\park^i\p_t\erk_r\right)+\left(\park^i(\FP\eta_r)\right)\left(\park^r v_i\right)}_{=:N_j}\dt,
\end{aligned}
\end{equation} and thus
\begin{equation}\label{divlF}
\left\|\dive \FP\eta(T)\right\|_3^2\lesssim\delta^2\left\|\FP\eta\right\|_4^2+\left\|e'(\hr)\FP h\right\|_3^2+\left\|\int_0^T N_j(t)\dt\right\|_3^2
\end{equation}

Therefore, the divergence control is again reduced to the weighted estimates of $\p_t h$ and $\FP h$.

\subsubsection{Elliptic estimates and linearized wave equations}

Note that both $\p_t$ and $\FP$ are tangential derivatives on the boundary, so $h|_{\Gamma}=0$ implies that $\p_t h$ and $\FP h$ also vanish on the boundary. Therefore we can apply the elliptic estimates to both $\p_t h$ and $\FP h$.

We take divergence in the second equation of \eqref{elastollr} and invoke the third equation to get the linearized wave equation
\begin{equation}\label{hwavel0}
\begin{aligned}
e'(\hr)\p_t^2h-\lapark h=&\sum_{j=1}^3e'(\hr)\FP^2 h-\p_t\ark^{li}\p_l v_i-\park^i\left(\FP\erk_l\right)\park^l\left(\FP\erk_i\right)\\
&-e''(\hr)(\p_th)(\p_t\hr)+\sum_{j=1}^3e''(\hr)\left(\FP h\right)\left(\FP\hr\right)-\sum_{j=1}^3\FP\int_0^TN_j\dt.
\end{aligned}
\end{equation}

Invoking Lemma \ref{GLL}, we get
\begin{equation}\label{hl40}
\|h\|_4\lesssim P(\|\erk\|_3)\left(\|\lapark h\|_2+\|\TP\erk\|_3\|h\|_3\right).
\end{equation} Then plugging the wave equation \eqref{hwavel0}, we get similar estimates as in the treatment of \eqref{h40}
\begin{equation}\label{laplh2}
\|\lapark h\|_2\lesssim \left\|e'(\hr)\p_t^2 h\right\|_2+\sum_{j=1}^3\left\|e'(\hr)\FP^2 h\right\|_2+\PP_0+\int_0^T P(\EEr_\kk(t))\dt.
\end{equation} Therefore $\|h\|_4$ is reduced to $\left\|e'(\hr)\p_t^2 h\right\|_2$ adn $\sum\limits_{j=1}^3\left\|e'(\hr)\FP^2 h\right\|_2$.

Similarly, $\|\p_t h\|_3$ can be controlled by using Lemma \ref{GLL} and $\p_t$-differentiated wave equation \eqref{hwavel0}
\begin{equation}
\|\p_t h\|_3\approx\|\park \p_t h\|_2\lesssim \|\p_t\lapark h\|_1+\|[\p_t,\lapark]h\|_1,
\end{equation} and then
\begin{equation}
\begin{aligned}
\|\p_t\lapark h\|_1\lesssim&\|e'(\hr)\p_t^3 h\|+1+\sum_{j=1}^3\left\|e'(\hr)\FP^2\p_t h\right\|_1\\
&+\|\p_t(\p_t\ark \p v)\|_1+\sum_{j=1}^3\left\|\p_t(\park(\FP \eta))\right\|_1\left\|(\park(\FP \eta))\right\|_1\\
&+\|e''(\hr)\p_t\hr\p_t^2 h\|_1+\|e''(\hr)\p_t\hr\FP^2h\|_1\\
&+\|e'''(\hr)\p_t^2\hr\p_t h\|_1+\|e'''(\hr)\p_t\FP h\cdot\FP h\|_1+\sum_{j=1}^3\|\FP N_j\|_1.
\end{aligned}
\end{equation} Mimicing the proof of \eqref{ht30} yields
\begin{equation}\label{htl30}
\|\p_t\lapark h\|_1\lesssim\|e'(\hr)\p_t^3 h\|+1+\sum_{j=1}^3\left\|e'(\hr)\FP^2\p_t h\right\|_1+\PP_0+\int_0^T P(\EEr_\kk(t))\dt.
\end{equation}

Following the same manner of \eqref{ht30}, \eqref{hF30}, \eqref{htt20} and \eqref{hFF20}, we can get the following reduction
\begin{align}
\label{hFl30} \left\| \FP h\right\|_3\rightarrow&\left\|{e'(\hr)}\p_t^2\FP h\right\|_1+\sum_{k=1}^3\left\|{e'(\hr)}\FP(\F_k\cdot\p)^2\right\|_1\\
\label{httl20} \left\|\sqrt{e'(\hr)}\p_t^2 h\right\|_2\rightarrow&\left\|{e'(\hr)}^{\frac32}\p_t^4 h\right\|_0+\sum_{j=1}^3\left\|{e'(\hr)}^{\frac32}\FP^2\p_t^2 h\right\|_0\\
\label{hFFl20} \left\|\sqrt{e'(\hr)}\FP^2 h\right\|_2\rightarrow&\left\|{e'(\hr)}^{\frac32}\p_t^2\FP^2 h\right\|_0+\sum_{k=1}^3\left\|{e'(\hr)}^{\frac32}\FP^2(\F_k\cdot\p)^2 h\right\|_0.
\end{align}

Therefore, as in \eqref{remainh}, it remains to control the following quantities

\begin{equation}\label{remainhl}
\begin{aligned}
&\left\|e'(\hr)\p_t^3 h\right\|_1+\left\|(e'(\hr))^{\frac32}\p_t^4 h\right\|_0\\
&+\sum_{j=1}^3\left\|(e'(\hr))^2\p_t^2\FP h\right\|_1+\sum_{j=1}^3\left\|(e'(\hr))^{\frac32}\p_t^2\FP^2 h\right\|_0\\
&+\sum_{j=1}^3\left\|(e'(\hr))^2\FP^2 \p_t^2h\right\|_0+\sum_{j=1}^3\left\|e'(\hr) \FP^2\p_t h\right\|_1\\
&+\sum_{j,k=1}^3\left\|(e'(\hr))^2(\F_k\cdot\p)^2\FP^2h\right\|_0+\sum_{k=1}^3\left\|(e'(\hr))^2 (\F_k\cdot\p)^2\FP h\right\|_1,
\end{aligned}
\end{equation}
which will be estimated by $\p_t^3$-differentiated, $\p_t^2\FP$-differentiated, $\p_t\FP^2$-differentiated and $\sum\limits_{k=1}^3(\F_k\cdot\p)^2\FP$-differentiated wave equation as in Section \ref{wavediv}. There is no essential difference but just replacing lower order term of $h$ and weight function $\wt$ by the counterpart of $\hr$ and replacing $\ak$ by $\ark$ due to the linearisation. So we omit the proof here.

\subsubsection{Reduction of time derivatives of $v$ and $\FP\eta$}

Finally, we follow the same manner as in \eqref{remainvtttt}, \eqref{remainFtttt} and \eqref{remainvF} to get the following reduction

\begin{equation}\label{remainvFl}
\begin{aligned}
\left\|e'(\hr)\p_t^4 v\right\|_0\lesssim&\|e'(\hr) \p_t^3 h\|_1+\sum_{j=1}^3\|\F_j\|_{L^{\infty}}\left\|e'(\hr)\FP\p_t^3\eta\right\|_1+\text{lower order terms}.\\
\left\|e'(\hr)\FP\p_t^4 \eta\right\|_0\lesssim&\|\F_j\|_{L^{\infty}}\left\|e'(\hr)\p_t^3 v\right\|_1+\left\|e'(\hr)\FP\p_t^3\psir\right\|_0.\\
\left\|(e'(\hr))^{\frac12}\p_t^3 v\right\|_1\lesssim&\left\|(e'(\hr))^{\frac12}\park \p_t^2 h\right\|_1+\sum_{j=1}^3\|\F_j\|_{2}\left\|(e'(\hr))^{\frac12}\FP\p_t^2 \eta\right\|_2+\text{lower order terms}.\\
\left\|(e'(\hr))^{\frac12}\FP\p_t^3 \eta\right\|_1\lesssim&\|\F_j\|_2\left\|(e'(\hr))^{\frac12}\p_t^2 v\right\|_2+\left\|(e'(\hr))^{\frac12}\FP\p_t^2\psir\right\|_1.\\
\left\|\p_t^2 v\right\|_2\lesssim&\|\park\p_t h\|_2+\sum_{j=1}^3\|\F_j\|_2\|\FP\p_t \eta\|_3+\text{lower order terms}.\\
\left\|\FP\p_t^2\eta\right\|_2\lesssim&\|\F_j\|_2\|\p_t v\|_3+\|\FP\p_t\psir\|_2.\\
\left\|\p_t v\right\|_3\lesssim&\|\park h\|_3+\sum_{j=1}^3\|\F_j\|_3\|\FP\eta\|_4.\\
\left\|\FP\p_t\eta\right\|_3\lesssim&\|\F_j\|_3\|v\|_4+\|\FP\psir\|_4.
\end{aligned}
\end{equation}

Combining \eqref{curlrvF}, \eqref{tglinear}, \eqref{divlv}, \eqref{divlF}, \eqref{hl40}, \eqref{htl30}-\eqref{remainvFl}, we get
\begin{equation}
\EEr_\kk(T)\lesssim_{\kk^{-1}}\PP_0+P(\EEr_\kk(T))\int_0^T P(\EEr_{\kk}(t))\dt
\end{equation}
and thus by Gronwall-type inequality, there exists $T_\kk>0$ such that the following uniform-in-$n$ estimates hold
\begin{equation}\label{EEkkll}
\sup_{0\leq t\leq T_\kk}\EEr_\kk(t)\lesssim P(\|v_0\|_4^,\|\h_0\|_4,\|\F\|_4).
\end{equation}Therefore Proposition \ref{llenergy} is proven.

\subsection{Picard iteration to the nonlinear approximation system}

The last step is to construct the solution to \eqref{elastolkk} by Picard iteration which proves Propostion \ref{nnlwp}. Define 
\begin{equation}\label{diffl1}
[v]^{(n)}:=v^{(n+1)}-v^{(n)},~~[h]^{(n)}:=h^{(n+1)}-h^{(n)},~~[\eta]^{(n)}:=\eta^{(n+1)}-\eta^{(n)}, 
\end{equation}and
\begin{equation}\label{diffl2}
[a]^{(n)}:=a^{(n)}-a^{(n-1)},~~[\psi]^{(n)}:=\psi^{(n)}-\psi^{(n-1)}. 
\end{equation} Then we have the following system of $([\eta]^{(n)},[v]^{(n)},[h]^{(n)})$
\begin{equation}\label{elastodiffl}
\begin{cases}
\p_t[\eta]^{(n)}=[v]^{(n)}+[\psi]^{(n)}~~~&\text{ in }\Omega \\
\p_t[v]^{(n)}=-\nabla_{\ak^{(n)}}[h]^{(n)}-\nabla_{[\ak]^{(n)}} h^{(n)}+\sum\limits_{j=1}^3\FP^2[\eta]^{(n)}~~~&\text{ in }\Omega \\
\text{div}_{\ak^{(n)}}[v]^{(n)}=-e'(h^{(n)})\p_t [h]^{(n)}-\text{div}_{[\ak]^{(n)}}v^{(n)}-(e'(h^{(n)})-e'(h^{(n-1)}))\p_t h^{(n)}~~~&\text{ in }\Omega \\
\dive\F_j=-e'(\h_0)\FP\h_0~~~&\text{ in }\Omega \\
[h]^{(n)}=0, \F_j\cdot N=0~~~&\text{ on }\Gamma, \\
\end{cases}
\end{equation}together with its energy functional
\begin{equation}\label{ediffl}
\begin{aligned}
[\EE]^{(n)}(T)&:=\left\|[\eta]^{(n)}\right\|_3^2+\left\|[v]^{(n)}\right\|_3^2+\sum\limits_{j=1}^3\left\|\FP[\eta]^{(n)}\right\|_3^2+\left\|[h]^{(n)}\right\|_3^2\\
&+\left\|\p_t [v]^{(n)}\right\|_2^2+\sum\limits_{j=1}^3\left\|\FP\p_t[\eta]^{(n)}\right\|_2^2+\left\|\p_t [h]^{(n)}\right\|_2^2\\
&+\left\|\p_t^2[v]^{(n)}\right\|_1^2+\sum\limits_{j=1}^3\left\|\FP\p_t^2\eta^{(n)}\right\|_1^2+\left\|\sqrt{e'(h^{(n)})}\p_t^2 [h]^{(n)}\right\|_1^2\\
&+\left\|\sqrt{e'(h^{(n)})}\p_t^3 [v]^{(n)}\right\|_0^2+\sum\limits_{j=1}^3\left\|\sqrt{e'(h^{(n)})}\FP\p_t^3 [\eta]^{(n)}\right\|_0^2+\left\|e'(h^{(n)})\p_t^3 [h]^{(n)}\right\|_0^2.
\end{aligned}
\end{equation} Here the correction term $[\psi]^{(n)}$ becomes
\begin{equation}
-\Delta [\psi]^{(n)}=0,
\end{equation} with the following boundary condition
\begin{align*}
[\psi]^{(n)}=&\sum_{L=1}^2\TP^{-1}\mathbb{P}\bigg(\TL[\eta]^{(n-1)}_{k}\ak^{(n)Lk}\TP_i\lkk^2v^{(n)}+\TP\eta^{(n-1)}_{\beta}[\ak]^{(n)Lk}\TP_L\lkk^2v^{(n)}+\TP\eta^{(n-1)}_{k}\ak^{(n-1)Lk}\TP_L\lkk^2[v]^{(n-1)} \\
&-\TL\lkk^2[\eta]^{(n-1)}_{k}\ak^{(n) Lk}\TP_L v^{(n)}-\TL\lkk^2\eta^{(n-1)}_{k}[\ak]^{(n) Lk}\TP_L v^{(n)}-\TL\lkk^2\eta^{(n-1)}_{k}\ak^{(n-1) Lk}\TP_L [v]^{(n-1)}\bigg).
\end{align*}

\subsubsection{Estimates of $[\eta],[\ak],[\psi]$}

First by definition we have
\begin{align*}
[a]^{(n)li}(T)=\int_0^T\p_t(a^{(n)li}-a^{(n-1)li})\dt=-\int_0^T [a]^{(n)lr}\p_k\p_t\eta^{(n)}_{r}a^{(n)ki}+a^{(n-1)lr}\p_k\p_t[\eta]^{(n-1)}_{r}a^{(n)ki}+a^{(n-1)lr}\p_k\p_t\eta^{(n-1)}_{r}[a]^{(n)ki},
\end{align*}and thus
\begin{equation}\label{diffa}
\|[a]^{(n)}(T)\|_2\lesssim\int_0^T\|[a]^{(n)}(T)\|_2^2(\|[v]^{(n-1)}\|_3+\|[\psi]^{(n)}\|_3)\dt.
\end{equation} The estimates of $[\psi]^{(n)}$ and $\FP[\psi]^{(n)}$ can be similarly derived by elliptic estimates
\begin{equation}\label{diffpsi}
\|[\psi]^{(n)}\|_3^2\lesssim |[\psi]^{(n)}|_{2.5}\lesssim \|[\eta]^{(n-1)}\|_3^2+\|[v]^{(n-1)}\|_2^2+\|[\ak]^{(n)}\|_1^2,
\end{equation}and
\begin{equation}\label{diffFpsi}
\|\FP[\psi]^{(n)}\|_3^2\lesssim |[\psi]^{(n)}|_{2.5}\lesssim P\left(\|\FP[\eta]^{(n-1)}\|_3,\|[\eta]^{(n-1)}\|_3,\|[v]^{(n-1)}\|_2,\|[\ak]^{(n)}\|_1\right).
\end{equation} Analogous results hold for their time derivatives so we omit the proof
\begin{align}
[\psi]^{(n)},\p_t[\psi]^{(n)},\FP[\psi]^{(n)},\p_t[\eta]^{(n)},\FP[\eta]^{(n)}\in& L^{\infty}([0,T];H^3(\Omega)),\\
\p_t^2[\psi]^{(n)},\p_t\FP[\psi]^{(n)},\p_t^2[\eta]^{(n)},\p_t\FP[\eta]^{(n)}\in& L^{\infty}([0,T];H^2(\Omega)),\\
\sqrt{e'(h^{(n)})}\p_t^3[\psi]^{(n)},\p_t^2\FP[\psi]^{(n)},\p_t^3[\eta]^{(n)},\p_t^2\FP[\eta]^{(n)}\in& L^{\infty}([0,T];H^1(\Omega)),\\
e'(h^{(n)})\p_t^4[\psi]^{(n)},\sqrt{e'(h^{(n)})}\left(\p_t^3\FP[\psi]^{(n)},\p_t^4[\eta]^{(n)},\p_t^3\FP[\eta]^{(n)}\right)\in& L^{\infty}([0,T];L^2(\Omega)).
\end{align}

\subsubsection{Curl estimates of $[v]$ and $\FP[\eta]$}

Similarly as in Section \ref{divcurl}, we get the evolution equation of $\curl_{\ak^{(n)}}[v]^{(n)}$ as
\begin{equation}
\begin{aligned}
&\p_t\left(\curl_{\ak^{(n)}}[v]^{(n)}\right)-\sum_{j=1}^3\FP\left(\curl_{\ak^{(n)}}\FP[\eta]^{(n)}\right)\\
=&\curl_{\ak_t^{(n)}}[v]^{(n)}+\curl_{[\ak_t]^{(n)}}v^{(n)}+\sum_{j=1}^3\left[\curl_{\ak^{(n)}},\FP\right](\FP[\eta]^{(n)})+\FP\left(\curl_{[\ak]^{(n)}}\FP\eta^{(n)}\right)+\left[\curl_{[\ak]^{(n)}},\FP\right](\FP\eta^{(n)}).
\end{aligned}
\end{equation} Then the curl estimates directly follows from the $L^2$-estimates of $\p^2$-differentiated evolution equation.

\subsubsection{Tangential estimates of $[v]$ and $\FP[\eta]$}

We adopt the Ainhac good unknown method as in Section \ref{tgspace}. For each $n$ we define
\[
\VV^{(n+1)}=\TP^3v^{(n+1)}-\TP^3\ek^{(n)}\cdot\nabla_{\ak^{(n)}}v^{(n+1)},~~\HH^{(n+1)}=\TP^3h^{(n+1)}-\TP^3\ek^{(n)}\cdot\nabla_{\ak^{(n)}}h^{(n+1)}.
\] Their differences are denoted by 
\[
[\VV]^{(n)}:=\VV^{(n+1)}-\VV^{(n)}, [\HH]^{(n)}:=\HH^{(n+1)}-\HH^{(n)}.
\] The evolution equation of $[\VV]$ and $[\HH]$ now becomes
\begin{equation}
\begin{aligned}
\p_t[\VV]^{(n)}+\nabla_{\ak^{(n)}}[\HH]^{(n)}-\sum_{j=1}^3\FP\TP^3(\FP[\eta]^{(n)})=-\nabla_{[\ak]^{(n)}}\HH^{(n)}+\ff^{(n)},
\end{aligned}
\end{equation}with boundary condition
\begin{equation}
[\HH]^{(n)}|_{\Gamma}=-\left(\TP^3\ek^{(n)}_{k}\ak^{(n)3k}+\TP^3[\ek]^{(n-1)}_{k}\ak^{(n)3k}+\TP^3\ek^{(n-1)}_{k}[\ak]^{(n)3k}\right),
\end{equation}and the compressibility equation
\begin{equation}
\nabla_{\ak^{(n)}}\cdot[\VV]^{(n)}=-\nabla_{[\ak]^{(n)}}\cdot\VV^{(n)}+\mathfrak{g}^{(n)}.
\end{equation}
Here \begin{align*}
\ff^{(n)i}&=\p_t\left(\TP^3[\ek]^{(n-1)}_{k}\ak^{(n)lk}\p_{l}v^{(n+1)}_{i}+\TP^3\ek^{(n-1)}_{k}[\ak]^{(n)lk}\p_{\mu}v^{(n+1)}_{i}+\TP^3\ek^{(n-1)}_{k}\ak^{(n)lk}\p_{l}[v]^{(n)}_{i}\right)\\
&~~~~+[\ak]^{(n)lk}\p_{l}(\ak^{(n)ri}\p_{r}h^{(n+1)})\TP^3\ek^{(n)}_{k}+\ak^{(n-1)lk}\p_{l}([\ak]^{(n)ri}\p_{r}h^{(n+1)})\TP^3\ek^{(n)}_{k} \\
&~~~~+\ak^{(n-1)lk}\p_{l}(\ak^{(n-1)ri}\p_{r}[h]^{(n)})\TP^3\ek^{(n)}_{k}+\ak^{(n-1)lk}\p_{l}([\ak]^{(n)ri}\p_{r}h^{(n)})\TP^3[\ek]^{(n-1)}_{k}\\
&~~~~-\left[\TP^2,[\ak]^{(n)lk}\ak^{(n)ri}\TP\right]\p_{r}\ek^{(n)}_{k}\p_{l}h^{(n+1)}-\left[\TP^2,\ak^{(n-1)lk}[\ak]^{(n)ri}\TP\right]\p_{r}\ek^{(n)}_{k}\p_{l}h^{(n+1)}\\
&~~~~-\left[\TP^2,\ak^{(n-1)lk}\ak^{(n-1)ri}\TP\right]\p_{r}[\ek]^{(n-1)}_{k}\p_{l}h^{(n+1)}-\left[\TP^2,\ak^{(n-1)lk}\ak^{(n-1)ri}\TP\right]\p_{r}\ek^{(n-1)}_{k}\p_{l}[h]^{(n)}\\
&~~~~-\left[\TP^3,[\ak]^{(n)li},\p_{l}h^{(n+1)}\right]-\left[\TP^3,\ak^{(n-1)li},\p_{l}[h]^{(n)}\right]+\sum_{j=1}^3\left[\TP^3\FP\right]\FP[\eta]^{(n)},
\end{align*}
and 
\begin{align*}
\mathfrak{g}^{(n)}&=\TP^3(\dive_{\ak^{(n)}}[v]^{(n)}-\dive_{[\ak]^{(n)}}v^{(n)}) \\
&~~~~-\left[\TP^2,[\ak]^{(n)lk}\ak^{(n)ri}\TP\right]\p_{r}\ek^{(n)}_{k}\p_{l}v^{(n+1)}_{i}-\left[\TP^2,\ak^{(n-1)lk}[\ak]^{(n)ri}\TP\right]\p_{r}\ek^{(n)}_{k}\p_{l}v^{(n+1)}_{i} \\
&~~~~-\left[\TP^2,\ak^{(n-1)lk}\ak^{(n-1)ri}\TP\right]\p_{r}[\ek]^{(n-1)}_{k}\p_{l}v^{(n+1)}_{i} -\left[\TP^2,\ak^{(n-1)lk}\ak^{(n)ri}\TP\right]\p_{r}\ek^{(n-1)}_{k}\p_{l}[v]^{(n)}_{i} \\
&~~~~-\left[\TP^3,[\ak]^{(n)li},\p_{l}v^{(n+1)}_{i}\right]-\left[\TP^3,\ak^{(n-1)li},\p_{l}[v]^{(n)}_{i}\right] \\
&~~~~+[\ak]^{(n)lk}\p_{l}(\ak^{(n)ri}\p_{r}v^{(n+1)}_{i})\TP^3\ek^{(n)}_{k}+\ak^{(n-1)lk}\p_{l}([\ak]^{(n)ri}\p_{r}v^{(n+1)}_{i})\TP^3\ek^{(n)}_{k} \\
&~~~~+\ak^{(n-1)lk}\p_{l}(\ak^{(n-1)ri}\p_{r}[v]^{(n)}_{i})\TP^3\ek^{(n)}_{k}+\ak^{(n-1)lk}\p_{l}([\ak]^{(n)ri}\p_{r}v^{(n)}_{i})\TP^3[\ek]^{(n-1)}_{k}.
\end{align*}

Then we do the $L^2$-estimates of the good unknowns to get analogous results as in \eqref{tglinear}
\begin{equation}\label{difftg}
\left\|\TP^3 [v]^{(n)}\right\|_0^2+\sum_{j=1}^3\left\|\TP^3 \left(\FP[\eta]^{(n)}\right)\right\|_0^2+\left\|\sqrt{e'(h^{(n)})}\TP^3 [h]^{(n)}\right\|_0^2
\lesssim_{\kk^{-1}}\PP_0+\int_0^T [\EE]^{(n)}(t)+\EE^{(n-1)}(t)\dt.
\end{equation}

\subsubsection{Divergence estimates and Reduction procedure}

Similarly as in \eqref{divllF}, we derive the formula of $\dive_{\ak^{(n)}}\FP[\eta]^{(n+1)}$
\begin{equation}\label{divllFdiff}
\begin{aligned}
&~~~~\dive_{\ak^{(n)}}\FP[\eta]^{(n+1)}=-e'(h^{(n)})\FP[h]^{(n)}+\dive_{[a]^{(n)}}\eta^{(n)}-(e'(h^{(n)})-e'(h^{(n-1)}))\FP h^{(n)}\\
+&\int_0^T\Bigg(\dive_{\ak^{(n)}}\FP[\psi]^{(n)}+\dive_{[\ak]^{(n-1)}}\FP\psi^{(n-1)}\\
&-\left(\nabla_{\ak^{(n)}}^r(\FP[\eta_i]^{(n)})+\nabla_{[\ak]^{(n)}}^r(\FP\eta_i^{(n)})+\nabla_{\ak^{(n-1)}}^r(\FP\eta_i^{(n)})\right)\\
&~~~~\cdot\left(\nabla_{\ak^{(n)}}^r([\p_t\eta_i]^{(n)})+\nabla_{[\ak]^{(n)}}^r(\p_t\eta_i^{(n)})+\nabla_{\ak^{(n-1)}}^r(\p_t\eta_i^{(n)})\right)\\
&-\left(\nabla_{\ak^{(n)}}^r(\FP[\eta_i]^{(n)})+\nabla_{[\ak]^{(n)}}^r(\FP\eta_i^{(n)})+\nabla_{\ak^{(n-1)}}^r(\FP\eta_i^{(n)})\right)\\
&~~~~\cdot\left(\nabla_{\ak^{(n)}}^r[v_i]^{(n)}+\nabla_{[\ak]^{(n)}}^rv_i^{(n)}+\nabla_{\ak^{(n-1)}}^rv_i^{(n)}\right)\Bigg)\dt.
\end{aligned}
\end{equation}

The wave equation of $[h]^{(n+1)}$ becomes
\begin{equation}
\begin{aligned}
&e'(h^{(n)})\p_t^2 [h]^{(n)}-\Delta_{\ak^{(n)}}[h]^{(n)}=-\sum_{j=1}^3e'(h^{(n)})\FP^2 [h]^{(n)}\\
&-\p_t\left((e'(h^{(n)})-e'(h^{(n-1)}))\p_t h^{(n)}\right)-\sum_{j=1}^3(e'(h^{(n)})-e'(h^{(n-1)}))\FP^2 h^{(n)} \\
&-(\p_t\ak^{(n)li})\p_l [v]^{(n)}_{i}-\p_t[\ak]^{(n)li}\p_l v^{(n)}_i+[\ak]^{(n)li}\p_{l}({\ak^{(n)m}}_{i}\p_m h^{(n)})+\ak^{(n-1)li}\p_{l}({[\ak]^{(n)m}}_{i}\p_m h^{(n)})\\
&-\sum_{j=1}^3\left(\nabla_{\ak^{(n)}}^i(\FP[\ek_l]^{(n)})+\nabla_{[\ak]^{(n)}}^i(\FP\ek_l^{(n)})+\nabla_{\ak^{(n-1)}}^i(\FP\ek_l^{(n)})\right)\\
&~~~~\cdot\left(\nabla_{\ak^{(n)}}^l(\FP[\ek_i]^{(n)})+\nabla_{[\ak]^{(n)}}^l(\FP\ek_i^{(n)})+\nabla_{\ak^{(n-1)}}^l(\FP\ek_i^{(n)})\right)\\
&+\FP\left(\dive_{[a]^{(n)}}\eta^{(n)}-(e'(h^{(n)})-e'(h^{(n-1)}))\FP h^{(n)}\right)+\FP\int_0^T\cdots\dt,
\end{aligned}
\end{equation}where the time integral is the same as in \eqref{divllFdiff}. Similarly as the previous reduction procedure shown in \eqref{hFl30}-\eqref{hFFl20} we are able to finalize the energy estimates. We omit the details here because there is no essential difference from the previous control for the linearized equation.

Finally, we can apply Gronwall-type inequality to prove that there exists a sufficiently small $T_{\kk}>0$ such that
\[
\forall t\in[0,T_\kk],~~~[\EE]^{(n)}\leq\frac14\left([\EE]^{(n-1)}+[\EE]^{(n-2)}\right),
\]and thus
\[
[\EE]^{(n)}\lesssim_{\kk^{-1}}\PP_0/2^n.
\] Therefore, the sequence of approximation solutions $\{(\eta^{(n)},v^{(n)},h^{(n)})\}_{n\in\N}$ strongly converges (subsequentially) to $(\eta(\kk),v(\kk),h(\kk))$ of \eqref{elastolkk} as $n\to\infty$. The local well-posedness of the nonlinear $\kk$-approximation system is established.

\section{Local well-posedness of the original system}\label{lwp0}

Now we can finalize the proof of local well-posedness of the free-boundary compressible elastodynamic equations \eqref{elastol}. In Proposition \ref{nnlwp}, we proved that for each fixed $\kk>0$, the nonlinear $\kk$-approximation problem admits a unique strong solution $(\eta(\kk),v(\kk),h(\kk))$ in time interval $[0,T_\kk]$. In Proposition \ref{apriorikk}, we derive the uniform-in-$\kk$ a priori estimates for $(\eta(\kk),v(\kk),h(\kk))$. Therefore, there exists a $T>0$ independent of $\kk$, such that  $(\eta(\kk),v(\kk),h(\kk))$ exists in $[0,T]$ for each $\kk>0$. Also, such uniform energy bound yields the strong convergence (subsequentially) of $(\eta(\kk),v(\kk),h(\kk))$ to a limit $(\eta,v,h)$ in $[0,T]$ which solves the original system \eqref{elastol} and the corresponding energy functional $\EE (T)$ defined in \eqref{energy} satisfies the energy bound \eqref{energy1}. 

It remains to prove the uniqueness of the solution. We suppose $(\eta^1,v^1,h^1),(\eta^2,v^2,h^2)$ to be two solutions to \eqref{elastol} which satisfies the energy estimate \eqref{energy1} in Theorem \ref{lwp}. Denote the difference by $([\eta],[v],[h]):=(\eta^1-\eta^2,v^1-v^2,h^1-h^2)$ and $a^L:=[\p\eta^L]^{-1}$ ($L=1,2$) with $[a]:=a^2-a^1$. Then $([\eta],[v],[h])$ solves the following system \textbf{with ZERO initial data}:

\begin{equation}\label{diff}
\begin{cases}
\p_t[\eta]=[v]~~~&\text{ in }\Omega, \\
\p_t[v]+\sum\limits_{j=1}^3\FP^2[\eta]=-\nabla_{a^1}[h]+\nabla_{[a]}h^2~~~&\text{in }\Omega, \\
\dive_{a^1}[v]=\dive_{[a]}v^2-e'(h^2)\p_t[h]-(e'(h^1)-e'(h^2))\p_t h^2~~~&\text{in }\Omega \\
[h]=0~~~&\text{ on }\Gamma.
\end{cases}
\end{equation} 
We define the energy functional of \eqref{diff} by 
\begin{equation}\label{Ediff}
\begin{aligned}
[\EE]=&\left\|[\eta]\right\|_2^2+\sum_{k=0}^2\left\|\p_t^{2-k}[v]\right\|_k^2+\sum_{j=1}^3\left\|\p_t^{2-k}\FP[\eta]\right\|_k^2+\left|(a^{1})^{3i}\TP^2[i]_{i}\right|_0^2\\
&+\|[h]\|_2+\|\p_t [h]\|_2+\left\|\sqrt{e'(h^1)}\p_t^2[h]\right\|_0^2
\end{aligned}
\end{equation}

The only essential difference in the energy estimates is the boundary integral $$\ig [\HH](a^{1})^{3i} [\VV]_{i}\dS.$$.

We define the Alinhac good unknowns of $v^L,h^L$ for $L=1,2$
\[
\VV^L=\TP^2v^L-\TP^2\eta^L\cdot\nabla_{a^L} v^L,~~\HH^L=\TP^2h^L-\TP^2\eta^L\cdot\nabla_{a^L} h^L, 
\]and
\[
[\VV]:=\VV^1-\VV^2,~~[\HH]:=\HH^1-\HH^2.
\]
 The boundary integral then becomes
\begin{align*}
\ig [\HH](a^{1})^{3i} [\VV]_{i}=&-\ig\p_3[h]\TP^2\eta^2_{k}(a^2)^{3k}(a^2)^{3i}[\VV]_{i}\dS-\ig \p_3 h^1(\TP^2[\eta]_{k}(a^1)^{3k}+\TP^2\eta^2_{k}[a]^{3k})(a^1)^{3i}[\VV]_{i}\dS \\
\lesssim& -\frac{1}{2}\frac{d}{dt}\ig \p_3 h^1|(a^{1})^{3i}\TP^2[\eta]_{i}|_0^2\dS\\
&-\ig\p_3 h^1(a^1)^{3r}\TP^2[\eta]_{r}(\TP^2\eta^2_{k}[a]^{lk}\p_l v^1_{i}-\TP^2\eta^2_{k}(a^2)^{lk}\p_\mu[v]_{i})(a^1)^{3i}\dS \\
&-\ig \p_3 h^1(\TP^2[\eta]_{k}(a^1)^{3k}+\TP^2\eta^2_{k}[a]^{3k})(a^1)^{3i}[\VV]_{i}\dS\\
&\lesssim-\frac{c_0}{2}\frac{d}{dt}\ig |(a^{1})^{3i}\TP^2[\eta]_{i}|_0^2\dS+P(\text{initial data})P([\EE](t)).
\end{align*}

Here we use the precise formula of $[\VV]$, and in the third step we apply the physical sign condition for $h^1$. Therefore we have
\[
\sup_{t\in[0,T_0]}[\EE](t)\leq P(\text{initial data})+\int_0^{T_0}P([\EE](t))\dt.
\] 

Since the initial data of \eqref{diff} is 0, then we know $[\EE](t)=0$ for all $t\in[0,T]$ which gives the uniqueness of the solution to the compressible elastodynamics system \eqref{elastol}. Theorem \ref{lwp} is proven.

\section{Incompressible limit}\label{ilimit}

In this section we will prove Theorem \ref{limit}, i.e., the incompressible limit. This requires our energy estimate to be uniform in the sound speed. In physics, the sound speed of a compressible fluid is defined by $c^2(t,x):=p'(\rho)$. We parametrize the sound speed by $\eps>0$ such that
\[
{p'}_{\eps}(\rho)|_{\rho=1}=\eps.
\] Recall that the enthalpy derivative ${\h '}_{\eps}(\rho):=\frac{{p'}_{\eps}(\rho)}{\rho}>0$. We can write $\rho$ as a function of $\h$ depending on $\eps$. 

In the proof of Theorem \ref{lwp}, the estimate of the weighted energy functional $\EE(T)$ is uniform in sound speed $\eps$, i.e., it does not rely on $1/e'(h)$. Once we have this uniformly bounded energy, we are able to prove that the solution $(v^{\eps},\FF^{\eps},h^{\eps})$ of \eqref{elastoL} converges to $(V,G,Q)$ of incompressible elastodynamics system \eqref{elastoi} provided that the initial data $(v_0^{\eps},\FF_0^{\eps},\h_0^{\eps})$ converges to the initial data $(\vvv,\GG_0,Q_0)$ as $\eps\to\infty$. 

Specifically, let $\vvv$ be a divergence-free vector field, $\G$ be a divergence-free matrix and $Q_0$ is defined by the elliptic system with constraints $-\frac{\p Q}{\p N}|_{\Gamma}\geq c_0>0$
\[
\begin{cases}
-\Delta Q_0=\p_i\vvv^k\p_k\vvv^i-\p_i\G_{kj}\p_k\G_{ij}~~~&\text{in }\Omega,\\
Q_0=0~~~&\text{on }\Gamma.\\
\end{cases}
\] Let $(V,G,Q)$ be the solution to the free-boundary incompressible elastodynamic system \eqref{elastoi} with initial data $(\vvv,\G,Q_0)\in H^4\times H^4\times H^4$:
\[
\begin{cases}
\p_t\eta=V~~~&\text{in }\Omega,\\
\p_t V=-\pa Q+\sum\limits_{j=1}^3(G\cdot\pa)G~~~&\text{in }\Omega,\\
\dive_a V=0~~~&\text{in }\Omega,\\
(\dive G^{\top})_j:=\p_k G_{kj}=0~~~&\text{in }\Omega,\\
\p_t|_{\Gamma}\in\mathcal{T}([0,T]\times\Gamma)~~~&\text{on }\Gamma,\\
Q=0~~~&\text{on }\Gamma,\\
G_j\cdot N=0~~~&\text{on }\Gamma,\\
-\frac{\p Q}{\p N}\geq c_0>0~~~&\text{on }\Gamma,\\
(\eta,V,G,Q)|_{t=0}=(\text{Id},\vvv,\G,Q_0).
\end{cases}
\] 

Suppose that there exists initial data $(v_0^{\eps},\FF_0^{\eps},\h_0^{\eps})$ of compressible elastodynamic system \eqref{elastoL} converging to $(\vvv,\GG_0,Q_0)$ in $C^1$-norm (proved in the Section \ref{data}). Then the uniform-in-$\eps$ energy estimates \eqref{energy1} show that $v^{\eps},h^{\eps},F^{\eps}$ are all $C_{t,x}^1$ are uniformly bounded and also that $\p_t v,\p_t F,\p_t h\in C^{0,\frac12}$ by Morrey's embedding. Therefore, we actually show that $v^{\eps},h^{\eps},F^{\eps}$ are equi-continuous in $C_{t,x}^1$, and thus $(v^{\eps},F^{\eps},h^{\eps})$ has a convergent subsequence by Arzel\`a-Ascoli Lemma. As $\eps\to+\infty$, we have $e'(h_{\eps})\p_t h_{\eps},~e'(h_{\eps})\FP h_{\eps}\to 0$ which implies $$(v^{\eps},F^{\eps},h^{\eps})\to (V,G,Q)$$ provided the convergence of initial data. Finally, we have
\begin{equation}\label{soundl1}
\rho_{\eps}(h)\to 1\text{ and }e_{\eps}(h)\to 0,\text{ as }\eps\to\infty.
\end{equation}Therefore the incompressible limit is established. Theorem \ref{limit} is proven.

\section{Enhanced regularity of full time derivatives}\label{enhance0}

In Lindblad-Luo \cite{lindblad2018priori} and Luo \cite{luo2018ww} concerned with compressible Euler equations, their energy functionals contain the $H^1$-norm of $\p_t^4$-derivative to close the energy estimates and also the incompressible limit. This is because they had to include the full time derivatives in the boundary energy. In the presenting manuscript we no longer need those terms because we exactly have proven that the boundary energy is contributed by the full spatial tangential derivatives. Our result is also applicable to compressible Euler equations. Yet we are still able to prove the higher regularity of full time derivatives of $(v,F,h)$ to recover the previous incompressible limit results for free-boundary Euler equations when the elastic medium is slightly compressible, i.e., $|e'(h)|\ll 1$ is sufficiently small.

However, we cannot mimic the proof in \cite{lindblad2018priori,luo2018ww} to derive analogous result for compressible elastodynamic equations. The essential reason is that the source term of the following wave equation loses one derivative in $\sum\limits_{j=1}^3e'(h)\FP^2 h$ due to the appearance of the deformation tensor $F$
\begin{equation}\label{hwave}
\begin{aligned}
e'(h)\p_t^2 h-\Delta_a h=&\sum_{j=1}^3e'(h)\FP^2 h-\p_ta^{li}\p_lv_i-\sum_{j=1}^3(\pa^iF_{lj})\cdot(\pa^lF_{ij})-e''(h)(\p_t h)^2+\sum_{j=1}^3e''(h)(\FP h)^2.
\end{aligned}
\end{equation} If we analyze the $\p_t^4$-differentiated wave equation \eqref{hwave}, then we will have to control $\p\p_t^4 v$ coming from the source term. In \cite{lindblad2018priori,luo2018ww}, one can invoke the Euler equations $\p_t v=-\pa h$ to replace $\p\p_t^4 v$ by $\p^2\p_t^3 h$ whose $L^2$-norm can be reduced to $\|\p_t^5 h\|_0$ plus lower order terms by Lemma \ref{GLL}. However, the presence of deformation tensor produces another term $\p(\FP)\p_t^3 F$. Since $F\neq\mathbf{O}$ on the boundary, one cannot apply Lemma \ref{GLL} to $F$ and thus has to analyze higher order spatial derivatives which will ruin the whole proof. 

We notice that, for elastodynamic equations \eqref{elastoL}, we can do further div-curl-tangential control and delicate analysis of Alinhac good unknowns and the $\p_t^4$-differentiated wave equation to estimate the weighted norms of $\p_t^4v,\p_t^4F$ and $\p_t^4h,\p_t^5 h$. The loss of derivative caused by $e'(h)\FP^2 h$ can be absorbed if the compressibility of the elastic medium is sufficiently slight. Thus, our proof is totally different from \cite{lindblad2018priori, luo2018ww, ZhangCRMHD1}.

\subsection{Higher order wave equations and divergence control}\label{div5}

From now on, we again have $\FF_j=\FP\eta$. Taking $\p_t^4$ in \eqref{hwave} yields
\begin{equation}\label{hwave5}
\begin{aligned}
e'(h)\p_t^6h-\dive_a(\p_t^4 \pa h)=&\sum_{j=1}^3e'(h)\FP^2 \p_t^4 h+\underbrace{\p_t^4\left((\pa^i v_l)(\pa^l v_i)-\sum_{j=1}^3(\pa^iF_{lj})\cdot(\pa^lF_{ij})\right)}_{Z_1}\\
&\underbrace{-\p_t^4\left(e''(h)(\p_th)^2\right)+\sum_{j=1}^3\p_t^4\left(e''(h)(\FP h)^2\right)}_{Z_2}\\
&\underbrace{-\left[\p_t^4,e'(h)\right]\p_t^2 h+\left[\p_t^4,\dive_a\right]\pa h}_{Z_3}.
\end{aligned}
\end{equation}

We compute the $L^2$ inner product of \eqref{hwave5} and $(\wt)^2\p_t^5 h$ to get
\begin{align}
\label{h5e1} &\frac12\frac{d}{dt}\io\left|(\wt)^{\frac32}\p_t^5 h\right|^2+\left|\wt \p_t^4\pa h\right|^2\dy\\
\label{h5e2} =&\sum_{j=1}^3\io (\wt)^3\FP^2\p_t^4 h\cdot\p_t^5 h\dy+\io(\wt)^2(Z_1+Z_2+Z_3)\p_t^5 h\dy\\
\label{h5z4} &+\frac32\io(\wt)^2e''(h)\p_t h\left|\p_t^5 h\right|^2\dy+\io\wt e''(h)\p_t h\left|\p_t^4\pa h\right|^2\dy\\
\label{h5z5} &+\io(\wt)^2\left[\pa,\p_t^4\right]\p_t h\dy+2\wt e''(h)(\pa h)\p_t^5 h(\p_t^4\pa h)\dy
\end{align}

Let us analyze each term in details. In the first term of \eqref{h5e2}, we can integrate $\FP$ by parts to get the divergence control of $\FF$
\begin{equation}\label{h5e21}
\begin{aligned}
&\io (\wt)^3\FP^2\p_t^4 h\cdot\p_t^5 h\dy\\
=&-\frac12\frac{d}{dt}\io\left|(\wt)^{\frac32}\FP\p_t^4 h\right|^2\dy+\frac32\io(\wt)^2e''(h)\p_t h\left|\FP\p_t^4 h\right|^2\dy\\
&-3\io(\wt)^2e''(h)(\FP h)\left(\FP\p_t^4 h\right)\p_t^5 h\dy-\io(\wt)^3(\dive \F_j)\left(\FP\p_t^4 h\right)\p_t^5 h\dy\\
\lesssim&-\frac12\frac{d}{dt}\io\left|(\wt)^{\frac32}\FP\p_t^4 h\right|^2\dy+\|\wt\p_t h\|_{L^{\infty}}\left\|(\wt)^{\frac32}\FP\p_t^4 h\right\|_0^2\\
&+\left(\left\|\wt\FP h\right\|_{L^{\infty}}+\left\|\dive \F_j\right\|_{L^{\infty}}\right)\left\|(\wt)^{\frac32}\FP\p_t^4 h\right\|_0\left\|(\wt)^{\frac32}\p_t^5 h\right\|_0
\end{aligned}
\end{equation}

The analysis of the second term in \eqref{h5e2} will be postponed to the end of this part. Now we analyze \eqref{h5z4}.
\begin{equation}
\begin{aligned}
&\frac32\io(\wt)^2e''(h)\p_t h\left|\p_t^5 h\right|^2\dy+\io\wt e''(h)\p_t h\left|\p_t^4\pa h\right|^2\dy\\
\lesssim&\|\wt\p_t h\|_{L^{\infty}}\left(\left\|(\wt)^{\frac32}\p_t^5 h\right\|_0^2+\left\|\wt \p_t^4\pa h\right\|_0^2\right)
\end{aligned}
\end{equation}

Then we analyze \eqref{h5z5} term by term (we omit the coefficients)
\begin{align*}
-\left[\pa,\p_t^4\right]\p_t h=\p_t^4a\cdot\p\p_t h+\p_t^3a\cdot\p\p_t^2 h+\p_t^2a\cdot\p\p_t^3 h+\p_ta\cdot\p\p_t^4 h,
\end{align*}and thus
\begin{equation}
\begin{aligned}
&\io(\wt)^2\left[\pa,\p_t^4\right]\p_t h\dy\\
\lesssim&\left\|\swt\p_t^3 v\right\|_1\left\|\p_t h\right\|_1\|\wt^{\frac32}\|_{L^{\infty}}+\|\p_t^2 v\|_1\|\wt\p_t^2 h\|_1\|\wt\|_{L^{\infty}}\\
&+\|\p_t v\|_1\cdot\|\wt\p_t^3 h\|_1\|\wt\|_{L^{\infty}}+\|\p v\cdot a\|_1\|\wt\pa \p_t^4 h\|_0\|\wt\|_{L^{\infty}}
\end{aligned}
\end{equation} So \eqref{h5z5} is estimated as follows
\begin{equation}\label{h5z51}
\begin{aligned}
\eqref{h5z5}\lesssim&\left\|\swt\p_t^3 v\right\|_1\left\|\p_t h\right\|_1\|\wt^{\frac32}\|_{L^{\infty}}+\|\p_t^2 v\|_1\|\wt\p_t^2 h\|_1\|\wt\|_{L^{\infty}}\\
&+\|\p_t v\|_1\cdot\|\wt\p_t^3 h\|_1\|\wt\|_{L^{\infty}}+\|\p v\cdot a\|_1\|\wt\pa \p_t^4 h\|_0\|\wt\|_{L^{\infty}}\\
&+\left\|\swt\pa h\right\|_{L^{\infty}}\left\|\wt\p_t^4\pa h\right\|_0\left\|(\wt)^{\frac32}\p_t^5 h\right\|_0
\end{aligned}
\end{equation}

Finally, we analyze the second integral in \eqref{h5e2}. In fact, it suffices to compute $\left\|\swt(Z_1+Z_2+Z_3)\right\|_0$. We first write $Z_1$ term by term,
\begin{align*}
Z_1=&2\p_t^4(\pa v)\cdot(\pa v)-2\p_t^4(\pa F)\cdot(\pa F)\\
&+8\p_t^3(\pa v)\cdot\p_t(\pa v)-8\p_t^3(\pa F)\cdot\p_t(\pa F)\\
&+6\p_t^2(\pa v)\cdot\p_t^2(\pa v)-6\p_t^2(\pa F)\cdot\p_t^2(\pa F), 
\end{align*}and thus
\begin{equation}\label{h5z1}
\begin{aligned}
\left\|\swt Z_1\right\|_0\lesssim P(\|\eta\|_4)&\bigg(\left\|\swt\p_t^4 v\right\|_1\|\p v\|_{L^{\infty}}+\left\|\swt\p_t^4 F\right\|_1\|\p F\|_{L^{\infty}}\\
&+\left\|\swt\p_t^3 v\right\|_1\|\p\p_t v\|_{L^{\infty}}+\left\|\swt\p_t^3 F\right\|_1\|\p\p_t F\|_{L^{\infty}}\\
&+\left\|\swt\right\|_{L^{\infty}}\left(\left\|\p_t^2 v\right\|_2^2+\left\|\p_t^2 F\right\|_2^2\right)\bigg).
\end{aligned}
\end{equation}

For $Z_2$, it suffices to analyze its first term because $\p_t h$ and $\FP h$ always have the same power of weight functions. We have
\begin{align*}
\swt&\p_t^4\left(e''(h)(\p_th)^2\right)=\swt e''(h)\left(2\p_t^5 h\p_t h+8\p_t^4 h\p_t^2 h+12\p_t^3 h\p_t^3 h\right)\\
&+\swt e^{(3)}(h)\p_t h\left(2\p_t^4 h\p_t h+6\p_t^3 h\p_t^2 h\right)\\
&+\swt\left(e^{(4)}(h)(\p_t h)^2+e^{(3)}(h)\p_t^2 h\right)\left(2\p_t^3 h\p_t h+4\p_t^2 h\p_t^2 h\right)\\
&+\swt\left(e^{(5)}(h)(\p_t h)^3+3e^{(4)}(h)\p_t h\p_t^2 h+e^{(3)}(h)\p_t^3 h\right)\left(2\p_t^2 h\p_t h\right)\\
&+\swt\left(e^{(6)}(h)(\p_t h)^4+6e^{(5)}(h)\p_t^2 h(\p_t h)^2+3e^{(4)}(h)(\p_t^2 h\p_t^2 h+\p_t h\p_t^3h)+e^{(3)}(h)\p_t^4 h\right)\left(\p_t h\right)^2
\end{align*}
Invoking the physical constraints \eqref{sound}, we know 
\begin{equation}\label{h5z2}
\begin{aligned}
&\left\|\swt\p_t^4\left(e''(h)(\p_th)^2\right)\right\|_0\\
\lesssim& P\left(\left\|\swt\right\|_{L^{\infty}},\left\|(\wt)^{\frac32}\p_t^5 h\right\|_0,\left\|(\wt)^{\frac32}\p_t^4 h\right\|_0,\left\|\wt\p_t^3 h\right\|_1,\left\|(\wt)\p_t^2 h\right\|_2,\left\|\p_t h\right\|_2\right)
\end{aligned}
\end{equation}
Similar result holds for the second term by replacing $\p_t$ by $\FP$ so we omit the details.

For $Z_3$, we have (omitting the coefficients)
\begin{align*}
Z_3:=&-\left[\p_t^4,e'(h)\right]\p_t^2 h+\left[\p_t^4,\dive_a\right]\pa h\\
=&-e''(h)(\p_t h\p_t^5 h+\p_t^2 h\p_t^4 h+\p_t^3 h\p_t^3 h)-e^{(3)}(h)((\p_t h)^2\p_t^4 h+\p_t h\p_t^2 \p_t^3 h+(\p_t^2 h)^3)\\
&-e^{(4)}(h)((\p_th)^3\p_t^3 h+(\p_t h)^2(\p_t^2 h)^2)-e^{(5)}(h)(\p_t h)^4\p_t^2 h\\
&+\p_t^4 a\cdot\p(\pa h)+\p_t^3a\cdot\p\p_t(\pa h)+\p_t^2a\cdot\p\p_t^2(\pa h)+\p_ta\cdot\p\p_t^3(\pa h).
\end{align*}
Therefore
\begin{equation}\label{h5z3}
\begin{aligned}
&\left\|\swt Z_3\right\|_0\\
\lesssim P\bigg(& \|\wt\|_{L^{\infty}},\left\|(\wt)^{\frac32}\p_t^5 h\right\|_0,\left\|(\wt)^{\frac32}\p_t^4 h\right\|_0,\left\|\wt\p_t^3 h\right\|_1,\left\|(\wt)\p_t^2 h\right\|_2,\left\|\p_t h\right\|_3,\\
&\|h\|_4,\left\|\swt\p_t^3 v\right\|_1,\|\p_t^2 v\|_2,\|\p_t v\|_2\bigg)+\|\p v\|_{L^{\infty}}\|\p\eta\|_{L^{\infty}}\left\|\swt\p_t^3\pa h\right\|_1
\end{aligned}
\end{equation}

In the last term we again invoke Lemma \ref{GLL} to get
\begin{equation}\label{h5z31}
\begin{aligned}
\left\|\swt\p_t^3\pa h\right\|_1\leq&\left\|\swt\pa \p_t^3h\right\|_1+\left\|\swt[\p_t^3,\pa] h\right\|_1\\
\lesssim&\left\|\swt\Delta_a\p_t^3h\right\|_0+\left\|\swt[\p_t^3,\pa] h\right\|_1+L.O.T.\\
\leq&\left\|\swt\p_t^3\Delta_a h\right\|_0+\left\|\swt[\p_t^3,\Delta_a]\right\|_0+\left\|\swt[\p_t^3,\pa] h\right\|_1+L.O.T.
\end{aligned}
\end{equation}

Invoking the $\p_t^3$-differentiated wave equation \eqref{htttwave}, we know the two commutators above has been controlled in the analysis of \eqref{htttwave} so we do not repeat here. The only term that we have to be careful is the first source term in \eqref{htttwave}, i.e., $\sum\limits_{j=1}^3\wt\FP^2\p_t^3 h$. The inequality \eqref{h5z31} now becomes
\begin{equation}\label{h5z32}
\left\|\swt\p_t^3\pa h\right\|_1\lesssim\sum_{j=1}^3\left\|\swt\wt\FP^2\p_t^3 h\right\|_0\lesssim\|\F\|_2^2\|\wt\|_{L^{\infty}}\left\|\swt\p_t^3 h\right\|_2
\end{equation}
Since $\left\|\swt\p_t^3\pa h\right\|_1\approx\left\|\swt\p_t^3 h\right\|_2$ and $|\wt|\ll 1$ for a slightly compressible elastic medium, the RHS can be absorbed to LHS by letting $\eps$ sufficiently large (or say $e'(h)$ sufficiently small). 

Thus, the energy of $\p_t^4$-differentiated wave equation is controlled as follows
\begin{align}\label{div5vF}
\frac12\frac{d}{dt}\left(\io\left|(\wt)^{\frac32}\p_t^5 h\right|^2+\left|\wt \p_t^4\pa h\right|^2+\sum_{j=1}^3\left|(\wt)^{\frac32}\FP\p_t^4 h\right|^2\dy\right)\lesssim P(\EEE(T)).
\end{align}

\subsection{Curl control}\label{curl5}

According to Lemma \ref{hodge}, we need to control the $L^2$-norm of divergence, curl, and tangential derivative of $\p_t^4 v$ and $\p_t^4 \FF$. In Section \ref{div5}, we control the weighted divergence of $v$ and $F$ because 
\[
\swt\dive\p_t^4 h\approx -(e'(h))^{\frac32}\p_t^5 h,~~\swt\dive\p_t^4 F_j\approx\swt\dive_a\p_t^4(\FP\eta)=-(\wt)^{\frac32}\FP\p_t^4 h.
\] This part we are going to control the curl. Keep in mind $F_j=\FP\eta$.

We take $\p_t^4\curl_a$ in the seond equation of \eqref{elastol} to get the evolution equation of curl
\[
\p_t^4(\curl_a\p_t v)=\sum_{j=1}^3\p_t^4\left(\curl_a(\FP F_j)\right).
\]and thus
\begin{equation}\label{curl5eq}
\begin{aligned}
\p_t(\curl_a\p_t^4 v)-\sum_{j=1}^3\FP\curl_a\p_t^4F_j=&-\left[\FP,\curl_a\right]\p_t^4 F_j\\
&-3\curl_{\p_ta}\p_t^4 v-6\curl_{\p_t^2 a}\p_t^3 v\\
&-4\curl_{\p_t^3 a}\p_t^2 v-\curl_{\p_t^4 a}\p_t v\\
&+\sum_{j=1}^3\bigg(4\curl_{\p_ta}\FP\p_t^3F_j+6\curl_{\p_t^2 a}\FP\p_t^2F_j\\
&+4\curl_{\p_t^3 a}\FP\p_t F_j+\curl_{\p_t^4 a}\FP F_j\bigg).
\end{aligned}
\end{equation}

Then we compute the $L^2$ inner product of \eqref{curl5eq} and $\wt \curl_a\p_t^4 v$ to get
\begin{equation}
\begin{aligned}
&\frac12\frac{d}{dt}\io\wt\left|\curl_a\p_t^4 v\right|^2+\sum_{j=1}^3\wt\left|\curl_a\p_t^4\FP\eta\right|^2\dy\\
=&\sum_{j=1}^3\io\wt\left(\curl_a\p_t^4 v\right)\cdot\left[\FP,\curl_a\right]\p_t^4 v\dy\\
&+\text{RHS of }\eqref{curl5eq}+\cdots,
\end{aligned}
\end{equation}where the omitted terms correspond to the case that dertivatives fall on the weight function (which must be lower order and of higher power of weight functions). 

There is only one term that cannot be controlled by direct computation: $4\sum\limits_{j=1}^3\curl_{\p_ta}\FP\p_t^3F_j$ which requires $\|\p_t^3F\|_2$. However, we invoke again the second equation of \eqref{elastol} to replace $\sum\limits_{j=1}^3\FP F_j$ by $\p_t v-\pa h$:
\[
4\sum\limits_{j=1}^3\curl_{\p_ta}\FP\p_t^3 F_j=4\curl_{\p_t a}\p_t^4 v-4\curl_{\p_t a}\p_t^3\pa h,
\] where the first term can be directly controlled by $\|\p_t a\|_{L^{\infty}}\|\p_t^4 v\|_1$. The second term requires the control of $\|\p_t^3 h\|_2$. This can again be reduced to $\p_t^5 h$ by Lemma \ref{GLL}. One can mimic the proof in \eqref{h5z32} to control that.

Therefore, the following curl estimate holds
\begin{equation}\label{curl5vF}
\frac12\frac{d}{dt}\io\wt\left|\curl_a\p_t^4 v\right|^2+\sum_{j=1}^3\wt\left|\curl_a\p_t^4\FP\eta\right|^2\dy\lesssim P(\EEE(T)).
\end{equation}

\subsection{Tangential estimates: Alinhac good unknown}\label{tg14}

It remains to do the tangential estimates. Although $\TP\p_t^4$ contains time derivative, we cannot directly commute it with $\pa$ because this will produce a term like $\TP\p_t^4 a=\p_t^3\p^2 v\times\p\eta+\cdots$ which cannot be controlled in $L^2$. To avoid such problem, we again use Alinhac good unknown method. Define the Alinhac good unknowns for $v,h$ with respect to $\TP\p_t^4$ by
\begin{equation}\label{good5vF}
\begin{aligned}
\VVV:=&\TP\p_t^4v-\TP\p_t^4\eta\cdot\pa v=\TP\p_t^4v-\TP\p_t^3v\cdot\pa v,\\
\HHH:=&\TP\p_t^4 h-\TP\p_t^4\eta\cdot\pa h=\TP\p_t^4 h-\TP\p_t^3v\cdot\pa h.
\end{aligned}
\end{equation}

Similarly as in Section \ref{tgspace}, we have the following identity for the Alinhac good unknown $\mathfrak{g}:=\TP\p_t^4g-\TP\p_t^4\eta\cdot\pa g$ for a function $g$ with respect to $\TP\p_t^4$:
\begin{equation}\label{id5}
\TP\p_t^4(\pa g)=\pa \mathfrak{g}+\CC(g),
\end{equation}where
\begin{align}\label{com5}
\CC^i(g):=&\TP\p_t^4\eta_r\pa^i(\pa^r g)-\left(\left[\p_t^4,a^{lr}a^{mi}\right]\TP\p_m\eta_r\right)\p_l f+\left[\TP\p_t^4,a^{li},\p_lf\right].
\end{align}

Taking $\wt\TP\p_t^4$ in the second equation of \eqref{elastol} yields that
\begin{equation}\label{good5eq}
\wt\p_t\VVV-e'(h)\pa\HHH=\wt\sum_{j=1}^3\FP\TP\p_t^4F_j+\underbrace{\wt\left(\p_t(\TP\p_t^3 v\cdot\pa v)-\CC(h)+\left[\TP\p_t^4,\FP\right]F_j\right)}_{\ff}.
\end{equation}

Taking $L^2$ inner product of \eqref{good5eq} and $\VVV$ yields
\begin{equation}\label{good5e0}
\begin{aligned}
\frac12\frac{d}{dt}\io\wt|\VVV|^2\dy=\io\wt(\pa\HHH)\cdot\VVV\dy+\sum_{j=1}^3\io\wt\FP\TP\p_t^4F_j\cdot\VVV\dy+\underbrace{\io\ff\cdot\VVV\dy}_{\KK_0},
\end{aligned}
\end{equation}
subjected to 
\begin{align}
\pa\cdot\VVV=\TP\p_t^4(\dive_a v)-\CC^i(v_i),
\end{align}and
\begin{align}
\HHH|_{\Gamma}=(-\p_3 h)a^{3k}\TP\p_t^4\eta_k.
\end{align}
Here in \eqref{good5e0} all the terms containing derivatives that fall on the weight function are again omitted. 

We analyze the second term on RHS of \eqref{good5e0} by integrating $\FP$ by parts
\begin{equation}\label{tg5F}
\begin{aligned}
&\sum_{j=1}^3\io\wt\FP\TP\p_t^4F_j\cdot\VVV\dy\\
\overset{\FP}{=}&-\sum_{j=1}^3\io\wt\TP\p_t^4\FF_j\cdot\FP\left(\TP\p_t^4 v-\TP\p_t^4\eta\cdot\pa v\right)+\sum_{j=1}^3\io\wt(\dive\F_j)\TP\p_t^4F_j\cdot\VVV\\
=&-\sum_{j=1}^3\frac12\frac{d}{dt}\io\wt\left|\TP\p_t^4\left(\FP\eta\right)\right|^2\dy-\io\wt\TP\p_t^4F_j\cdot\left[\FP,\TP\right]\p_t^4 v\dy\\
&-\io\wt\TP\p_t^4F_j\cdot\TP\p_t^4\left(\FP\eta\right)\cdot \pa v\dy-\io\wt\TP\p_t^4F_j\cdot\left(\left(\left[\FP,\TP\right]\p_t^4\eta\right)\cdot\pa v\right)\\
&-\io\wt\TP\p_t^4F_j\cdot\left(\TP\p_t^4\eta\cdot\left(\FP\pa v\right) \right)\dy+\io\wt(\dive\F_j)\TP\p_t^4F_j\cdot\VVV\dy\\
\lesssim&-\sum_{j=1}^3\frac12\frac{d}{dt}\left\|\swt\TP\p_t^4\left(\FP\eta\right)\right\|_0^2+P\left(\|v\|_4,\|\swt\TP\p_t^4 v\|_0,\|\TP\p_t^4F\|_0,\|\swt\p_t^3 v\|_1\right).
\end{aligned}
\end{equation}

Next we analyze the first term in RHS of \eqref{good5e0} which will produce boundary energy. First we integrate by parts to get
\begin{equation}\label{good5e1}
\begin{aligned}
&\io\wt(\pa\HHH)\cdot\VVV\dy\\
=&-\io\wt\HHH(\pa\cdot\VVV)\dy-\io\wt\p_la^{li}\HHH\VVV_i\dy\underbrace{+\ig\wt\frac{\p h}{\p N}\TP\p_t^4\eta_ka^{3k}a^{3i}\VVV_i\dS}_{=:\BB}\\
=&-\io\wt\HHH\TP\p_t^4(\dive_a v)\dy+\io\wt\HHH\CC^i(v_i)\dy-\io\wt\p_la^{li}\HHH\VVV_i\dy+\BB\\
=:&\KK_1+\KK_2+\KK_3+\BB.
\end{aligned}
\end{equation}

We have that
\begin{equation}\label{KK3}
\KK_3\lesssim\|\p a\|_{L^{\infty}}\left\|\swt\HHH\right\|_0\left\|\swt\VVV\right\|_0.
\end{equation}

Next we plug the expression of $\HHH$ into $\KK_1$ to get
\begin{equation}\label{KK1}
\begin{aligned}
\KK_1=&-\io\wt\HHH\TP\p_t^4(\dive_a v)\dy=\io(\wt)^2(\TP\p_t^4 h-\TP\p_t^4\eta\cdot\pa h)\TP\p_t^5 h\dy\\
=&\frac12\frac{d}{dt}\left\|\swt\TP\p_t^4 h\right\|_0^2-\underbrace{\io(\wt)^2\TP\p_t^3 v_ka^{lk}\p_l h\TP\p_t^5 h\dy}_{\KK_4}.
\end{aligned}
\end{equation}

Then $\KK_4$ can be controlled after integrating $\p_t$ by parts under time integral
\begin{equation}\label{KK4}
\begin{aligned}
\int_0^T\KK_4(t)\dt:=&\int_0^T\io(\wt)^2\TP\p_t^3 v_ka^{lk}\p_l h\TP\p_t^5 h\dy\dt\\
\overset{\p_t}{=}&-\int_0^T\io(\wt)^2\TP\p_t^4 v_k a^{lk}\p_l h\TP\p_t^4 h\dy\dt+\io(\wt)^2\TP\p_t^3 v_ka^{lk}\p_l h\TP\p_t^4 h\dy\bigg|^T_0\\
&-\int_0^T\io(\wt)^2\TP\p_t^3v_k\p_t(\pa h)^k\TP\p_t^4 h\dy\\
\lesssim&\int_0^T P\left(\|\eta\|_{3},\|h\|_4,\|\p_t h\|_3,\|\p_t^3 v\|_1,\left\|\swt\TP\p_t^4 v\right\|_0,\left\|(\wt)^{\frac32}\TP\p_t^4 h\right\|_0 \right)\dt\\
&+\PP_0+\delta\left\|\wt\TP\p_t^4 h\right\|_0^2+\left\|\swt\right\|_{L^{\infty}}^2\left\|\swt\TP\p_t^3 v\right\|_{0}^2\|\pa h\|_{L^{\infty}}^2.
\end{aligned}
\end{equation}Here in the last step we use Young's inequality such that the $\delta$-term can be absorbed, and the last term can also be absorbed by $\EE(T)$ since $|\wt|\ll 1$ is assumed now. Therefore the term $\KK_1$ has been controlled.

Now we come to analyze the boundary integral. We follow the same method as in Section \ref{tgspace}. Since we do not have tangential smoothing now (and thus $\psi=0$), we can find the cancellation structure in the boundary integral.
\begin{equation}\label{B51}
\begin{aligned}
\BB:=&\ig\wt\left(\frac{\p h}{\p N}\right)\TP\p_t^4\eta_ka^{3k}a^{3i}\VVV_i\dS\\
=&-\ig\wt\left(-\frac{\p h}{\p N}\right)\TP\p_t^4\eta_ka^{3k}a^{3i}\left(\TP\p_t^5\eta_i-\TP\p_t^4\eta\cdot\pa v_i\right)\dS\\
=&-\frac12\frac{d}{dt}\ig\left(-\frac{\p h}{\p N}\right)\left|\swt a^{3i}\TP\p_t^4\eta_i\right|^2\dS+\frac12\ig\p_t\left(-\frac{\p h}{\p N}\right)\left|\swt a^{3i}\TP\p_t^4\eta_i\right|^2\dS\\
&+\underbrace{\ig\wt\left(-\frac{\p h}{\p N}\right)\p_ta^{3i}a^{3j}\TP\p_t^4\eta_j\TP\p_t^4\eta_i\dS}_{\BB_1}+\underbrace{\ig\wt\left(-\frac{\p h}{\p N}\right)\TP\p_t^4\eta_ja^{3j}a^{3i}\TP\p_t^4\eta_ka^{lk}\p_lv_i\dS}_{\BB_2}\\
\leq&-\frac{c_0}{4}\frac{d}{dt}\left|\swt a^{3i}\TP\p_t^4\eta_i\right|_0^2+\frac12\|\p_t h\|_3\left|\swt a^{3i}\TP\p_t^4\eta_i\right|_0^2+\BB_1+\BB_2.
\end{aligned}
\end{equation}

Plugging $\p_ta^{3i}=-a^{3r}\p_lv_ra^{li}$ into $\BB_1$ yields that
\begin{equation}\label{B52}
\begin{aligned}
\BB_1=&-\ig\wt\left(-\frac{\p h}{\p N}\right)a^{3r}\p_lv_ra^{ri}a^{3j}\TP\p_t^4\eta_j\TP\p_t^4\eta_i\dS\\
=-&\ig\wt\left(-\frac{\p h}{\p N}\right)a^{3r}\p_3v_ra^{3i}a^{3j}\TP\p_t^4\eta_j\TP\p_t^4\eta_i\dS\\
&\underbrace{-\sum_{L=1}^2\ig\wt\left(-\frac{\p h}{\p N}\right)a^{3r}\p_Lv_ra^{Li}a^{3j}\TP\p_t^4\eta_j\TP\p_t^4\eta_i\dS}_{\BB_3}\\
\lesssim&|\p_3 ha^{3r}\p_3 v_r|_{L^{\infty}}\left|\swt a^{3i}\TP\p_t^4\eta_i\right|_0^2+\BB_3.
\end{aligned}
\end{equation}

Next we compare $\BB_2$ with $\BB_3$
\begin{equation}\label{B53}
\begin{aligned}
\BB_2=&\ig\wt\left(-\frac{\p h}{\p N}\right)\TP\p_t^4\eta_ja^{3j}a^{3i}\TP\p_t^4\eta_ka^{3k}\p_3v_i\dS\\
&+\sum_{L=1}^2\ig\wt\left(-\frac{\p h}{\p N}\right)\TP\p_t^4\eta_ja^{3j}a^{3i}\TP\p_t^4\eta_ka^{Lk}\p_Lv_i\dS\\
=&\ig\wt\left(-\frac{\p h}{\p N}\right)\TP\p_t^4\eta_ja^{3j}a^{3i}\TP\p_t^4\eta_ka^{3k}\p_3v_i\dS-\BB_3\\
\lesssim&-\BB_3+|\p_3 ha^{3i}\p_3 v_i|_{L^{\infty}}\left|\swt a^{3i}\TP\p_t^4\eta_i\right|_0^2.
\end{aligned}
\end{equation}

Combining \eqref{B51}-\eqref{B53} we get the control of the boundary integral.
\begin{equation}\label{B5}
\BB\lesssim-\frac{c_0}{4}\frac{d}{dt}\left|\swt a^{3i}\TP\p_t^4\eta_i\right|_0^2+\left(\|\p_th\|_3+|\p_3 ha^{3i}\p_3 v_i|_{L^{\infty}}\right)\left|\swt a^{3i}\TP\p_t^4\eta_i\right|_0^2.
\end{equation}

It now remains to control $\KK_0$ and $\KK_2$. The reason for us to postpone the proof is that these two terms contain the commutator $\CC(v),\CC(h)$ in the Alinhac good unknown. Previously in Section \ref{tgspace} we do not have to care too much of it because all the derivatives are spatial in that case. However, the tangential derivative now becomes $\TP\p_t^4$ and the commutator might contain terms of the form $\p^2\p_t^3 v$ which cannot be controlled. 

In $\KK_0$, the third term in the commutator $\CC(h)$ contributes the following
\[
\KK_5:=-\io\wt\p_t a^{li}~\TP\p_l\p_t^3 h~\VVV_i\dy
\]
Note that one can apply Lemma \ref{GLL} to do elliptic estimates on $\left\|\swt\p_t^3 h\right\|_2$ as in \eqref{h5z31}-\eqref{h5z32} to reduce this term to $\left\|(\wt)^{\frac32}\p_t^5 h\right\|_0$ which has been controlled by the energy functional of $\p_t^4$-differentiated wave equation \ref{hwave5}.
\begin{equation}\label{KK5}
\KK_5\lesssim\left\|\p_t a\right\|_{L^{\infty}}\left\|(\wt)^{\frac32}\p_t^5 h\right\|_0\left\|\swt\VVV\right\|_0
\end{equation}

Another term we need to study carefully is $\KK_2$, where the commutator $\CC(v)$ produce the following terms
\begin{align*}
&-\io\wt\HHH\p_t a^{li}\TP\p_t^3\p_l v_i\\
=&-\io\wt\TP\p_t^4 h\p_t a^{mi}\TP\p_t^3\p_m v_i\dy+\io\wt\TP\p_t^4\eta_k a^{lk}\p_l h\p_t a^{mi}\TP\p_t^3\p_m v_i\dy\\
=:&\KK_6+\KK_7.
\end{align*}

We should control these two terms under time integral since we have to integrate $\p_t$ by parts. For simplicity we only list the most difficult terms. The omitted terms are always of lower order.
\begin{equation}\label{KK6}
\begin{aligned}
\int_0^T\KK_6:=&-\int_0^T\io\wt\TP\p_t^4 h\p_t a^{mi}\TP\p_t^3\p_m v_i\dy\dt\\
\overset{\TP}{=}&\int_0^T\io\wt\TP^2\p_t^4 h\p_t a^{mi}\p_t^3\p_m v_i\dy\dt+\cdots \\
\overset{\p_t}{=}&-\int_0^T\io\wt\TP^2\p_t^3 h\p_t a^{mi}\p_t^4\p_m v_i\dy\dt-\io\wt\TP^2\p_t^3 h\p_t a^{mi}\p_t^3\p_m v_i\dy+\cdots\\
\lesssim&\int_0^T\left\|\swt\TP^2\p_t^3 h\right\|_0\left\|\swt\TP\p_t^4 v\right\|_0\|\p_t a\|_{L^{\infty}}\dt+\left\|\swt\TP^2\p_t^3 h\right\|_0\left\|\swt\TP\p_t^3 v\right\|_0\|\p_t a\|_{L^{\infty}}\\
\lesssim&\int_0^TP(\EEE(t))\dt+\delta\left\|(\wt)^{\frac32}\p_t^5 h\right\|_0^2+\PP_0+\int_0^TP(\EE(t))\dt
\end{aligned}
\end{equation}

\begin{equation}\label{K7}
\begin{aligned}
\int_0^T\KK_7:=&\io\wt\TP\p_t^4\eta_k a^{lk}\p_l h\p_t a^{mi}\TP\p_t^3\p_m v_i\dy\dt\\
\overset{\p_t}{=}&-\int_0^T\io\wt\TP\p_t^5\eta_k a^{lk}\p_l ha^{mi}\TP\p_m\p_t^2v_i\dy\dt+\io\wt\TP\p_t^4\eta_k a^{lk}\p_l ha^{mi}\TP\p_m\p_t^2v_i\dy\\
=&-\int_0^T\io\wt\TP\p_t^4v_k a^{lk}\p_l ha^{mi}\TP\p_m\p_t^2v_i\dy\dt+\io\wt\TP\p_t^3v_k a^{lk}\p_l ha^{mi}\TP\p_m\p_t^2v_i\dy\\
\lesssim&\int_0^T\left\|\swt\right\|_{L^{\infty}}\left\|\swt\TP\p_t^4 v\right\|_0\|a\|_2^2\|h\|_3\|\p_t^2 v\|_2\dt\\
&+\left\|\swt\right\|_{L^{\infty}}\left\|\swt\TP\p_t^3 v\right\|_0\|a\|_2^2\|h\|_3\|\p_t^2 v\|_2\\
\lesssim&\int_0^TP(\EEE(t))\dt+\delta\|\p_t^2 v\|_2^2+\frac{1}{4\delta}\left\|\swt\right\|_{L^{\infty}}^2\left\|\swt\TP\p_t^3 v\right\|_0^2\|a\|_2^4\|h\|_3^2.
\end{aligned}
\end{equation}

By picking a suitably small $\delta>0$, the $\delta$-term is absorbed by $\EE(T)$. The last term can also be absorbed by using $|\wt|\ll 1$.

Summarizing \eqref{good5e0}, \eqref{tg5F}-\eqref{KK4}, \eqref{B5}-\eqref{K7}, we get the estimates for the Alinhac good unknowns $\VVV$ and $\HHH$
\begin{equation}\label{good5e}
\left\|\swt\VVV\right\|_0^2+\left\|\wt\TP\p_t^4 h\right\|_0^2+\sum_{j=1}^3\left\|\swt\TP\p_t^4\FP\eta\right\|_0^2+\frac{c_0}{4}\left|\swt a^{3i}\TP\p_t^4\eta_i\right|_0^2\lesssim\PP_0+\int_0^T P(\EEE(t))\dt.
\end{equation}

Finally, by the definition of Alinhac good unknowns \eqref{good5vF}, we have
\[
\left\|\swt\TP\p_t^4 v\right\|_0\lesssim\left\|\swt\VVV\right\|_0+\left\|\swt\TP\p_t^3 v\right\|_0\|\pa v\|_{L^{\infty}}\lesssim P(\EE(T))+\left\|\swt\VVV\right\|_0,
\] and thus we finalize the tangential estimates as
\begin{equation}\label{tg5vF}
\begin{aligned}
&\left\|\swt\TP\p_t^4 v\right\|_0^2+\left\|\wt\TP\p_t^4 h\right\|_0^2+\sum_{j=1}^3\left\|\swt\TP\p_t^4\FP\eta\right\|_0^2+\frac{c_0}{4}\left|\swt a^{3i}\TP\p_t^4\eta_i\right|_0^2\\
\lesssim&\PP_0+\EEE(0)+\int_0^T \EEE(t)^2\dt.
\end{aligned}
\end{equation}

\subsection{A posteriori enhanced regularity of full time derivatives}\label{enhance5}

Combining Lemma \ref{hodge} and divergence estimate \eqref{div5vF}, curl estimate \eqref{curl5vF} and tangential estimates \eqref{tg5vF}, we are able to get the enhanced regularity of full time derivatives. 
\begin{equation}\label{energy5}
\begin{aligned}
&\left\|\swt\p_t^4 v\right\|_1^2+\left\|\wt\TP\p_t^4 h\right\|_0^2+\sum_{j=1}^3\left\|\swt\p_t^4\FP\eta\right\|_1^2+\frac{c_0}{4}\left|\swt a^{3i}\TP\p_t^4\eta_i\right|_0^2\\
+&\left\|(\wt)^{\frac32}\p_t^5 h\right\|_0^2+\left\|\wt\p_t^4 h\right\|_1^2\\
\lesssim&\PP_0+\EEE(0)+\int_0^T P(\EEE(t))\dt,
\end{aligned}
\end{equation} and thus by Gronwall's inequality, there exists some $0<T_1<T$ such that
\begin{equation}\label{energy25}
\sup_{0\leq t\leq T_1}\EEE(t)\leq \PP_0\EE(0).
\end{equation}

Such estimate recovers the energy bound in Lindblad-Luo \cite{lindblad2018priori} and Luo \cite{luo2018ww}. This demonstrates that our proof is completely applicable to the study of compressible Euler equations. To finalize to proof of Theorem \ref{enhance}, it remains to construct the initial data satisfying the compatibility conditions \eqref{ccd} up to 5-th order such that $\EEE(0)\leq \PP_0$. This will be proved in Section \ref{data}.

\section{Construction of initial data satisfying the compatibility conditions}\label{data}

The last section of this manuscript presents the construction of initial data satisfying the compatibility conditions \eqref{ccd} up to 5-th order. The compatibility conditions come from the boundary conditions and higher order wave equations. We start the initial data $(\vvv,\G,Q_0)$ of free-boundary incompressible elastodynamic equations: $\vvv$ is a divergence-free vector field and $\G$ is a divergence-free matrix in the sense of $\p_k\G_{kj}=0$, then $Q_0$ is defined by the elliptic system $-\Delta Q_0=(\p_i\vvv^k)(\p_k\vvv^i)+\sum\limits_{j=1}^3(\p_i\GG_{kj}^0)(\p_k\GG_{ij}^0)$ subject to $Q_0|_{\Gamma}=0$ and also the Rayleigh-Taylor sign condition $-\frac{\p Q_0}{\p n}|_{\Gamma}\geq c_0>0$. We are going to construct a sequence of compressible data $(v_0^{\eps},\F_{\eps},\h_0^{\eps})$ satisfying the compatibility conditions \eqref{ccd} up to 5-th order and strongly converging to $(\vvv,\G,Q_0)$ as the sound speed $\eps\to+\infty$. For simplicity\footnote{The proof in general case can be similarly proceeded. See Luo \cite{luo2018ww}.} we suppose that $e(\h)=\eps^{-1} h$ in order to omit the smaller and lower order terms in the wave equation of $h$. Then the compressible elastodynamic equations at time 0 ($\eta=$ Id) becomes
\begin{equation}\label{elasto0}
\begin{cases}
\p_t v=-\p h+\sum\limits_{j=1}^3\FP^2\eta~~~&\text{in }\Omega,\\
\dive v=-\eps^{-1}\p_t h~~~&\text{in }\Omega,\\
\p_tF_j=\FP v~~~&\text{in }\Omega,\\
\dive F_j:=\p_k F_{kj}=-\eps^{-1}\FP h~~~&\text{in }\Omega,\\
\p_t|_{\Gamma}\in\mathcal{T}([0,T]\times\Gamma)~~~&\text{on }\Gamma,\\
h=0~~~&\text{on }\Gamma,\\
\F_j\cdot N=0~~~&\text{on }\Gamma,\\
-\frac{\p \h_0}{\p N}\geq c_0>0~~~&\text{on }\Gamma.
\end{cases}
\end{equation}

\subsection{Compatibility conditions and constraints on the initial data}

From \eqref{elasto0} we find that the initial data $(v_0,\h_0)$ should satisfy
\begin{align}\label{ccd0}
\h_0|_{\Gamma}=0,~~\dive v_0|_{\Gamma}=0,~~\dive\F_j|_{\Gamma}=0
\end{align}since $\p_t$ and $\FP$ are tangential derivatives on the boundary $\Gamma$.

We now define $h_{(k)}:=\p_t^k h|_{t=0},~v_{(k)}:=\p_t^k v|_{t=0}$ and $F_{(k)}:=\p_t^k F|_{t=0}$. Recall the wave equation of $h$ now becomes the following equation at $t=0$
\begin{equation}
\eps^{-1}\p_t^2 h-\Delta h_0=(\p_iv_0^{k})(\p_kv_0^i)-\sum_{j=1}^3(\p_i\F_{kj})(\p_k\F_{ij})+\eps^{-1}\sum_{j=1}^3\FP^2 h_0,
\end{equation}and thus
\begin{equation}\label{h0eq}
-\Delta \h_0=-\eps^{-1}h_{(2)}+\eps^{-1}\sum_{j=1}^3\FP^2 \h_0+\underbrace{(\p_iv_0^{k})(\p_kv_0^i)-\sum_{j=1}^3(\p_i\F_{kj})(\p_k\F_{ij})}_{\NN_0},
\end{equation} So we also have
\begin{align}\label{ccd1}
\Delta\h_0+\NN_0=0~~\text{ on }\Gamma.
\end{align}

Time differentiating the wave equation yields that
\begin{equation}\label{h1eq}
-\Delta h_{(1)}=-\eps^{-1}h_{(3)}+\eps^{-1}\sum_{j=1}^3\FP^2 h_{(1)}+\NN_1,
\end{equation}and thus we must have
\begin{align}\label{ccd2}
\Delta h_{(1)}+\NN_1=0~~\text{ on }\Gamma.
\end{align}Here $\NN_1$ denotes the nonlinear quantities of $\p_t\NN_0|_{t=0}$.

Repeat this step, we get $\p_t^k$-differentiated wave equation at time $t=0$
\begin{equation}\label{hkeq}
-\Delta h_{(k)}=-\eps^{-1}h_{(k+2)}+\eps^{-1}\sum_{j=1}^3\FP^2 h_{(k)}+\NN_k,
\end{equation}and thus we must have
\begin{align}\label{ccdk}
\Delta h_{(k)}+\NN_k=0~~\text{ on }\Gamma.
\end{align}Here $\NN_k$ denotes the nonlinear quantities of $\p_t\NN_{k-1}|_{t=0}$.

\subsection{Construction of initial data}

Now we construct the compressible data based on the incompressible data. Let $(\vvv,\G,Q_0)$ be the incompressible data and then we define 
\begin{align}
\label{v00} v_0=&\vvv+\p\phi,\\
\label{F00} \F_j=&\G_j+\p\varphi_j.
\end{align} Therefore the continuity equation, the divergence constraints and the boundary condition of $\FF$ yield the following equations
\begin{align}
\label{v000} -\Delta\phi=\eps^{-1}h_{(1)},&~~~\frac{\p\phi}{\p N}|_{\Gamma}=0\\
\label{F000} -\Delta\varphi_j=\eps^{-1}\FP \h_0=\eps^{-1}\p_k\varphi_j\p_k \h_0,&~~~\frac{\p\varphi_j}{\p N}|_{\Gamma}=0.
\end{align} Then \eqref{h0eq}, \eqref{h1eq}, \eqref{hkeq} require that for $k=0,1,\cdots,m$
\begin{equation}\label{hk00}
\begin{cases}
-\Delta h_{(k)}=-\eps^{-1}h_{(k+2)}+\eps^{-1}\sum\limits_{j=1}^3\FP^2 h_{(k)}+\NN_k~~~&\text{ in }\Omega,\\
h_{(k)}=0~~~&\text{ on }\Gamma.
\end{cases}
\end{equation} Here $\NN_k$ is a function of $v_0,\F,\h_0,h_{(1)},\cdots,h_{(k-1)}$ and their spatial derivatives. If we set $h_{(m+1)}=h_{(m+2)}=0$. Then \eqref{v00}-\eqref{hk00} gives a system of $(v_0,\F,\h_0,h_{(1)},\cdots,h_{(N)})$ such that the compressible data $(v_0,\F,\h_0)$ strongly converges to the incompressible data as $\eps\to+\infty$ and for each $\eps>0$ the compressible data satisfies the compatibility conditions up to order $m$. Then one can invoke $\p_t v=-\p h+\sum\limits_{j=1}^3\FP\F_j$ and $\p_t\FF_j=\FP v$ to recover $v_{(k)}$ and $\FF_{(k)}$ in terms of $(v_0,\F,\h_0,h_{(1)},\cdots,h_{(N)})$. Thus, it remains to show that the system \eqref{v00}-\eqref{hk00} has a solution if $\eps$ is suitably large with uniform-in-$\eps$ energy bound.

\subsection{Existence of initial data satisfying the compatibility conditions}

In the previous proof of local well-posedness, incompressible limit and enhaced regularity, we need to find the compressible data satisfying the compatibility conditions up to 5-th order. According to the analysis above, we need to solve the following elliptic system
\begin{equation}\label{datak}
\begin{cases}
v_0=\vvv+\p\phi~~~&\text{ in }\Omega,\\
\F_j=\G_j+\p\varphi_j~~~&\text{ in }\Omega,\\
-\Delta\phi=\eps^{-1}h_{(1)}~~~&\text{ in }\Omega,\\
-\Delta\varphi_j=\eps^{-1}\FP\h_0~~~&\text{ in }\Omega,\\
-\Delta h_{(k)}=-\eps^{-1}h_{(k+2)}+\eps^{-1}\sum\limits_{j=1}^3\FP^2 h_{(k)}+\NN_k~~~&\text{ in }\Omega,~0\leq k\leq 3\\
h_{(k)}=0,\frac{\p\phi}{\p N}=\frac{\p\varphi_j}{\p N}=0~~~&\text{ on }\Gamma,\\
h_{(4)}=h_{(5)}=0~~~&\text{ in }\Omega.\\
\end{cases}
\end{equation}
Here the quantity $\NN_k$ has the following expression
\begin{equation}\label{NNk}
\NN_k=\sum C^{\alpha_1\cdots\alpha_m\beta_1\cdots\beta_n\gamma_1\cdots\gamma_p}_{\lambda_1\cdots\lambda_p,k}(\p^{\alpha_1}v_0)\cdots(\p^{\alpha_m}v_0)(\p^{\beta_1}\F)\cdots(\p^{\beta_n}\F)(\p^{\gamma_1}h_{(\lambda_1)})\cdots(\p^{\gamma_p}h_{(\lambda_p)}),
\end{equation}where
\begin{align*}
\alpha_1+\cdots+\alpha_m+\beta_1+\cdots+\beta_n+(\gamma_1+\lambda_1)+\cdots+(\gamma_p+\lambda_p)=k+2,\\
0\leq\alpha_i,\beta_i,(\gamma_i+\lambda_i)\leq k+1,0\leq\lambda_i\leq k-1.
\end{align*}We compute $\NN_0,\NN_1$ for example to illustrate how we get such a formula. When $k=0$, \eqref{h0eq} shows that
\[
\NN_0=(\p v_0)(\p v_0)-(\p\F)(\p\F).
\] When $k=1$, we have $\NN_1=(\p v_{(1)})(\p v_0)-(\p\FF_{(1)})(\p\F)$. Invoking \eqref{elasto0} we get $v_{(1)}=-\p \h_0+(\F\cdot\p)\F$ and $\FF_{(1)}=(\F\cdot\p)v_0$, and thus
\[
\NN_1=-(\p^2\h_0)(\p v_0)+\left(\F\cdot(\p^2\F)+(\p \F)(\p \F)\right)(\p v_0)-\F\cdot(\p^2 v_0)(\p\F).
\] The expression of $\NN_k$ can also be similarly computed and we omit the details.

\subsubsection{A priori estimates}

We first prove the uniform-in-$\eps$ a priori estimate for the elliptic system \eqref{datak}. This is necessary for us to do iteration to construct the solution to \eqref{datak}.

By standard elliptic estimate one has
\begin{equation}\label{phiee}
\|\p\phi\|_5\approx\|\Delta\phi\|_4\lesssim\eps^{-1}\|h_{(1)}\|_4,
\end{equation}and thus
\begin{equation}\label{v0ee}
\|v_0\|_5\lesssim\|\vvv\|_5+\eps^{-1}\|h_{(1)}\|_4.
\end{equation}
Similarly one has
\begin{equation}\label{phijee}
\|\p\varphi_j\|_5\approx\|\Delta\varphi_j\|_{4}\lesssim\eps^{-1}\|\F\|_4\|\h_0\|_5,
\end{equation}and 
\begin{equation}\label{F0ee}
\|\F\|_5\lesssim\|\G\|_5+\eps^{-1}\|\F\|_4\|\h_0\|_5.
\end{equation}

Then we estimate $\h_0$ by 
\begin{equation}
\|\h_0\|_5\lesssim\eps^{-1}\left(\|h_{(2)}\|_3+\|\h_0\|_5\|\F\|_3^2\right)+\|v_0\|_4^2+\|\F\|_4^2.
\end{equation} This together with \eqref{v0ee} and \eqref{F0ee} gives 
\begin{equation}\label{h0ee}
\|\h_0\|_5\lesssim \|\vvv\|_4^2+\|\G\|_4^2+\eps^{-1}P\left(\|h_{(1)}\|_3,\|h_{(2)}\|_3,\|\h_0\|_5,\|\F\|_4\right).
\end{equation}

Invoking the first and third equations in \eqref{elasto0} we get
\begin{equation}\label{v1ee}
\begin{aligned}
\|v_{(1)}\|_4\leq&\|\h_0\|_5+\|(\F\cdot\p)\F\|_4\lesssim P\left(\|\h_0\|_5,\|\F\|_5\right)\\
\lesssim&P(\|\ww\|_5,\|\G\|_5)+\eps^{-1}P\left(\|\h_0\|_5,\|h_{(1)}\|_4,\|\F\|_4\right),
\end{aligned}
\end{equation} and
\begin{equation}\label{F1ee}
\begin{aligned}
\|F_{(1)}\|_4\leq&\|(\F\cdot\p)v_0\|_4\lesssim \|\F\|_4\|v_0\|_5\\
\lesssim&P(\|\ww\|_5,\|\G\|_4)+\eps^{-1}P\left(\|\h_0\|_4,\|h_{(1)}\|_4,\|\F\|_3\right).
\end{aligned}
\end{equation}

Next we analyze $h_{(1)}$. By elliptic estimates and previous estimates, we have
\begin{equation}\label{h1ee}
\begin{aligned}
\|h_{(1)}\|_4\lesssim&\eps^{-1}\left(\|h_{(3)}\|_2+\|\F\|_3\|\F\|_2\|h_{(1)}\|_4\right)+\|v_0\|_3\|v_{(1)}\|_3+\|\F\|_3\|F_{(1)}\|_3\\
\lesssim&\eps^{-1}P\left(\|h_{(3)}\|_2,\|h_{(2)}\|_3,\|h_{(1)}\|_4,\|\h_0\|_5,\|\F\|_4\right)+P(\|\ww\|_5,\|\G\|_4).
\end{aligned}
\end{equation}

Again invoking the $\p_t$-differentiated first and third equations in \eqref{elasto0} we get
\begin{equation}\label{v2ee}
\begin{aligned}
\|v_{(2)}\|_3\leq&\|h_{(1)}\|_4+\|(\F\cdot\p)F_{(1)}\|_3\lesssim \|h_{(1)}\|_4+\|\F\|_3\|F_1\|_4\\
\lesssim&\eps^{-1}P\left(\|h_{(3)}\|_2,\|h_{(2)}\|_3,\|\h_0\|_5,\|\F\|_4\right)+P(\|\ww\|_5,\|\G\|_4),
\end{aligned}
\end{equation} and
\begin{equation}\label{F2ee}
\begin{aligned}
\|F_{(2)}\|_3\lesssim\eps^{-1}P\left(\|\h_0\|_5,\|\F\|_3\right)+P(\|\ww\|_5,\|\G\|_4).
\end{aligned}
\end{equation}

Once again we use elliptic estimates and \eqref{v1ee}-\eqref{F2ee} and $h_{(4)}=0$ to get
\begin{equation}\label{h2ee}
\begin{aligned}
\|h_{(2)}\|_3\lesssim&\eps^{-1}P\left(\|h_{(3)}\|_2,\|h_{(2)}\|_3,\|h_{(1)}\|_4,\|\h_0\|_5,\|\F\|_4\right)+P(\|\ww\|_5,\|\G\|_4).
\end{aligned}
\end{equation}

Again invoking the $\p_t^2$-differentiated first and third equations in \eqref{elasto0} we get
\begin{equation}\label{v3ee}
\begin{aligned}
\|v_{(3)}\|_2&\eps^{-1}P\left(\|h_{(3)}\|_2,\|h_{(2)}\|_3,\|h_{(1)}\|_4,\|\h_0\|_5,\|\F\|_4\right)+P(\|\ww\|_5,\|\G\|_4),
\end{aligned}
\end{equation} and
\begin{equation}\label{F3ee}
\begin{aligned}
\|F_{(3)}\|_2\leq\|(\F\cdot\p)v_{(2)}\|_2\lesssim \eps^{-1}P\left(\|h_{(3)}\|_2,\|h_{(2)}\|_3,\|h_{(1)}\|_4,\|\h_0\|_5,\|\F\|_4\right)+P(\|\ww\|_5,\|\G\|_4).
\end{aligned}
\end{equation}

Next we use elliptic estimates and \eqref{h2ee}-\eqref{F3ee} and $h_{(5)}=0$ to get
\begin{equation}\label{h3ee}
\begin{aligned}
\|h_{(3)}\|_2\lesssim&\eps^{-1}P\left(\|h_{(3)}\|_2,\|h_{(2)}\|_3,\|h_{(1)}\|_4,\|\h_0\|_5,\|\F\|_4\right)+P(\|\ww\|_5,\|\G\|_4).
\end{aligned}
\end{equation}

Again invoking the $\p_t^3$-differentiated first and third equations in \eqref{elasto0} we get
\begin{equation}\label{v4ee}
\begin{aligned}
\|v_{(4)}\|_1&\eps^{-1}P\left(\|h_{(3)}\|_2,\|h_{(2)}\|_3,\|h_{(1)}\|_4,\|\h_0\|_5,\|\F\|_4\right)+P(\|\ww\|_5,\|\G\|_4),
\end{aligned}
\end{equation} and
\begin{equation}\label{F4ee}
\begin{aligned}
\|F_{(4)}\|_1\lesssim\eps^{-1}P\left(\|h_{(3)}\|_2,\|h_{(2)}\|_3,\|h_{(1)}\|_4,\|\h_0\|_5,\|\F\|_4\right)+P(\|\ww\|_5,\|\G\|_4).
\end{aligned}
\end{equation}

Define the energy
\begin{equation}
\EEE_I:=\sum_{k=0}^4\|v_{(k)}\|_{5-k}^2+\|F_{(k)}\|_{5-k}^2+\|h_{(k)}\|_{5-k}^2.
\end{equation} Then our computation above shows that
\[
\EEE_I\lesssim\eps^{-1} P(\EEE_I)+P(\|\ww\|_3,\|\G\|_4),
\] and thus by choosing $\eps>0$ suitably large we have proved that
\begin{equation}
\EEE_I\leq P(\|\ww\|_5,\|\G\|_4),
\end{equation}which is the uniform-in-$\epsilon$ estimates for the elliptic system \eqref{datak}.

\subsubsection{Existence by iteration scheme}

Having the a priori estimates in hand, it remains to proceed the standard iteration scheme. We start from the solution $(\h_0^{(0)},h_{(1)}^{(0)},\cdots,h_{(4)}^{(0)})$ which solves
\begin{equation}\label{eeh0}
-\Delta h_{(k)}^{(0)}=\NN_k(\p^{\alpha}\vvv,\p^{\beta}\F,\p^{\gamma_0}\h_0^{(0)},\cdots,\p^{\gamma_{k-1}}h_{(k-1)}^{(0)}),~~0\leq k\leq 3.
\end{equation} 
Then we inductively define  $(\h_0^{(m)},h_{(1)}^{(m)},\cdots,h_{(3)}^{(m)})$ by 
\begin{equation}\label{datakn}
\begin{cases}
v_0^{(n)}=\vvv+\p\phi^{(n)}~~~&\text{ in }\Omega,\\
{\F_j}^{(n)}=\G_j+\p\varphi_j^{(n)}~~~&\text{ in }\Omega,\\
-\Delta\phi^{(n)}=\eps^{-1}h_{(1)}^{(n)}~~~&\text{ in }\Omega,\\
-\Delta\varphi_j^{(n)}=\eps^{-1}({\F_j}^{(n)}\cdot\p)\h_0^{(n)}~~~&\text{ in }\Omega,\\
-\Delta h_{(k)}^{(n)}=-\eps^{-1}h_{(k+2)}^{(n)}+\eps^{-1}\sum\limits_{j=1}^3({\F_j}^{(n)}\cdot\p)^2 h_{(k)}^{(n)}+\NN_k^{(n)}~~~&\text{ in }\Omega,~0\leq k\leq 3\\
h_{(k)}^{(n)}=0,\frac{\p\phi^{(n)}}{\p N}=\frac{\p\varphi_j^{(n)}}{\p N}=0~~~&\text{ on }\Gamma,\\
h_{(4)}^{(n)}=h_{(5)}^{(n)}=0~~~&\text{ in }\Omega,
\end{cases}
\end{equation}where $\NN_k^{(n)}$ satisfies the form \eqref{NNk}.

Define the difference 
\begin{align*}
[\phi]^{(n)}:=&\phi^{(n)}-\phi^{(n-1)},~~[\varphi_j]^{(n)}:=\varphi_j^{(n)}-\varphi_j^{(n-1)},\\
[v]_{(k)}^{(n)}:=&v_{(k)}^{(n)}-v_{(k)}^{(n-1)},[\FF]_{(k)}^{(n)}:=\FF_{(k)}^{(n)}-\FF_{(k)}^{(n-1)},[h]_{(k)}^{(n)}:=h_{(k)}^{(n)}-h_{(k)}^{(n-1)},\\
[\NN]_k^{(n)}:=&\NN_k^{(n)}-\NN_k^{(n-1)},\\
[\EEE_I]^{(n)}:=&\sum_{k=0}^4\|v_{(k)}^{(n)}\|_{5-k}^2+\|\FF_{(k)}^{(n)}\|_{5-k}^2+\|h_{(k)}^{(n)}\|_{5-k}^2.
\end{align*}

Then we have
\begin{equation}\label{datadiff}
\begin{cases}
[v]_0^{(n)}=\p[\phi]^{(n)}~~~&\text{ in }\Omega,\\
{[\F_j]}^{(n)}=\p[\varphi_j]^{(n)}~~~&\text{ in }\Omega,\\
-\Delta[\phi]^{(n)}=\eps^{-1}[h]_{(1)}^{(n)}~~~&\text{ in }\Omega,\\
-\Delta[\varphi_j]^{(n)}=\eps^{-1}({[\F_j]}^{(n)}\cdot\p)\h_0^{(n)}+\eps^{-1}({\F_j}^{(n-1)}\cdot\p)[\h]_0^{(n)}~~~&\text{ in }\Omega,\\
-\Delta [h]_{(k)}^{(n)}=-\eps^{-1}[h]_{(k+2)}^{(n)}+\eps^{-1}\sum\limits_{j=1}^3\bigg(({[\F_j]}^{(n)}\cdot\p)({\F_j}^{(n)}\cdot\p) h_{(k)}^{(n)}~~~\\
~~~~~~~~+({\F_j}^{(n-1)}\cdot\p)({[\F_j]}^{(n)}\cdot\p)h_{(k)}^{(n)}+({\F_j}^{(n-1)}\cdot\p)^2[h]_{(k)}^{(n)}\bigg)+[\NN]_k^{(n)}~~~&\text{ in }\Omega,~0\leq k\leq 3\\
[h]_{(k)}^{(n)}=0,\frac{\p\phi^{(n)}}{\p N}=\frac{\p\varphi_j^{(n)}}{\p N}=0~~~&\text{ on }\Gamma,\\
[h]_{(4)}^{(n)}=[h]_{(5)}^{(n)}=0~~~&\text{ in }\Omega,
\end{cases}
\end{equation}

Following the same method as in the a priori estimates, we are able to prove that
\[
[\EEE_I]^{(n)}\lesssim\eps^{-1} P(\|\vvv\|_5,\|\G\|_5)[\EEE_I]^{(n-1)},
\]and inductively we can get
\[
[\EEE_I]^{(n)}\lesssim\eps^{-1} P(\|\vvv\|_5,\|\G\|_5)^{(n)}[\EEE_I]^{(0)}.
\]

By choosing a suitably large $\eps>0$ such that $\eps^{-1} P(\|\vvv\|_5,\|\G\|_5)<0.99$, we finally prove that
\[
[\EEE_I]^{(n)}+\cdots+[\EEE_I]^{(n+m)}\to 0
\] as $m,n\to+\infty$. The iteration scheme is finished and thus the existence of the solution to the elliptic system \eqref{datak}, i.e., Theorem \ref{dataexist} is proven.


\end{document}